%% file: einart.tex
\documentclass[11pt]{amsart}

\usepackage{mystyle}

\title[Formal cornered asymptotically hyperbolic Einstein metrics]{Formal theory of cornered asymptotically hyperbolic Einstein metrics}
\author{Stephen E. McKeown}
\address{Department of Mathematics, Princeton University, Princeton, NJ 08544, USA}
\email{smckeown@math.princeton.edu}
\thanks{Research partially supported by NSF RTG Grants DMS-0838212 and DMS-1502525 and Grant DMS-1161283}

\begin{document}

\input{tex/abs}

\maketitle

\input{tex/intro}

\input{tex/back}
\input{tex/smooth}

\input{tex/lapl}
\input{tex/ein}

\bibliographystyle{alpha}
\bibliography{einart}

\end{document}

%% file: tex/abs.tex
\begin{abstract}
This paper makes a formal study of asymptotically hyperbolic Einstein metrics given, as conformal infinity, a conformal manifold with boundary.
The space on which such an Einstein metric exists thus has a finite boundary in addition to the usual infinite boundary and a corner where the two meet. On the finite boundary a constant mean curvature umbilic condition is imposed.
First, recent work of Nozaki, Takayanagi, and Ugajin is generalized and extended showing that such metrics cannot have smooth compactifications for generic corners embedded in
the infinite boundary. A model linear problem is then studied: a formal expansion at the corner is derived for eigenfunctions of the scalar Laplacian subject to certain boundary conditions.
In doing so, scalar ODEs are studied that are of relevance for a broader class of boundary value problems and also for
the Einstein problem. Next, unique formal existence at the corner, up to order at least equal to the boundary dimension, of Einstein metrics in a cornered asymptotically hyperbolic normal form which are polyhomogeneous in
polar coordinates is demonstrated for arbitrary smooth conformal infinity. Finally it is shown that, in the special case that the finite boundary is taken to be totally geodesic, there
is an obstruction to existence beyond this order, which defines a conformal hypersurface invariant.
\end{abstract}

%% file: tex/intro.tex
\section{Introduction}

This paper studies the formal existence and expansion of asymptotically hyperbolic (AH) Einstein metrics on manifolds with corner.
The setting of interest to us is manifolds with two codimension-one boundary faces, one of which is the conformal
infinity for an AH metric, and the other of which is an ordinary embedded hypersurface at which, away from the corner, the metric is regular.
Following our earlier study \cite{m16}, we will refer to these spaces as
cornered asymptotically hyperbolic (CAH) spaces.

One of the seminal problems in the theory of AH spaces, resolved in \cite{fg85,fg12}, 
was the formal existence of Einstein metrics given a particular conformal infinity $(M^n,[h])$.
This study had a significant impact on conformal geometry, and the formal expansion
developed in those works has found myriad applications from geometric analysis to the AdS/CFT conjecture
of physics. We will consider the same question of formal existence, but will take $M$ to be a manifold with boundary.
This necessitates equipping the Einstein space with a finite boundary, so that the boundary $S$ of $M$ becomes a corner,
and we have a boundary condition on the new finite boundary as well as at $M$. For reasons discussed below, we select the
CMC umbilic condition on the finite boundary. We will then consider formal existence of a CAH Einstein metric
at the boundary $S$ of $M$, viewed as the corner of the $(n+1)$-space.

AH Einstein spaces with corners have been considered in at least two previous settings. Local regularity at the boundary of AH
Einstein metrics was studied in \cite{bh14}, and studying a neighborhood of a point on the boundary necessarily introduces a
corner where the inner boundary of the neighborhood meets the boundary at infinity. The authors therefore developed
doubly weighted function spaces to analyze regularity. Since they were interested primarily in behavior away
from the corner, however, they successfully ``washed out'' the behavior of the metric at the corner itself,
and thus said relatively little about the metric there.

The paper \cite{ntu12} reflects an interest in this setting from physicists who wish to study the AdS/CFT correspondence when
the conformal field theory is on a space with boundary; see the discussion and references given there.
In the first part of that paper, the authors
considered a conformally compact manifold $X$ whose infinite boundary, $(M^n,[h])$, was a piece of $\mathbb{R}^2$ or
$\mathbb{R}^3$ with smooth boundary $S$ and endowed with the conformal class of the flat metric. They then added
a finite boundary $Q$ of $X$, intersecting $M$ precisely at $S$, imposing the boundary condition that $Q$ was umbilic with
constant mean curvature (CMC umbilic) 
with respect to the hyperbolic metric $g_+$ on the upper half-space. They concluded that, for $n = 2$, the boundary $S$ was
unrestricted, but that for $n = 3$ (or, they posited, larger), $S$ must be a sphere
or a plane. They therefore, in the latter part of the paper, considered a particular family of
perturbations of the 4-dimensional hyperbolic metric at the corner. They found that 
the Einstein equations could be solved to first order in the (compactified) distance to the corner for arbitrary $S$.

We will adopt the same CMC umbilic boundary conditions, which are geometrically natural and require a minimum of data
beyond the conformal infinity: only a single scalar, the mean curvature. We first, in section \ref{smoothsec}, undertake a study
of the case where the metric $g_+$ has a smooth compactification. We observe that in the case of hyperbolic space, the fact observed
in \cite{ntu12} for $n \geq 3$ -- that the corner $S$ must be a sphere or a hyperplane -- follows in all dimensions $n \geq 2$
from the classical theorem characterizing umbilic hypersurfaces in hyperbolic space. We then obtain a series of conditions on $S$
for smooth CAH Einstein metrics with arbitrary conformal infinity by repeatedly differentiating the condition of umbilicity and applying
the Einstein condition. We also explain the absence of the restriction on $S$ for the case $n = 2$ in \cite{ntu12}:
the restriction appears at one higher order in the expansion for a two-dimensional conformal infinity than for any higher dimension
(and in particular, at one higher order than was considered in that paper).
Also in this section, we show that the CMC umbilic boundary condition at $Q$ implies that $Q$ and $M$ make constant angle with respect
to any compactification of $g_+$.

Having confirmed that the requirement of smoothness of $g_+$ at the corner generically imposes severe restrictions on $S$ in $(M,[h])$, 
we turn to a general theory of formal existence.
As a first step, we must find an appropriate weakening of the smoothness condition. The most obvious condition, and essentially that which we impose,
is instead to require smoothness in polar coordinates at the corner, a condition which has a long history in problems with such singularities (see \cite{kon67,maz91}).
Invariantly, this means blowing up the corner in the sense of Melrose (throughout this paragraph, see
Section \ref{backsec} for details), obtaining a blown-up space $\tX$ and a blowdown map $b:\tX \to X$. We define
$\tM = b^{-1}(M)$ and $\tQ = b^{-1}(Q)$, to each of which $b$ restricts as a diffeomorphism; and $\tS = b^{-1}(S)$, to which
$b$ restricts as a fibration with fibers diffeomorphic to $[0,1]$. We note that $b$ induces a natural edge structure over both $M$ and $S$ in the sense
of \cite{maz91}, the former
with trivial fibers (a $0$-structure) and the latter with interval fibers.
In our prior paper, \cite{m16}, we considered a class of CAH metrics intermediate between those smooth on $X$ and those smooth on $\tX$, a restricted
subclass of the latter which we called \emph{admissible metrics} on $\tX$. In that paper, we proved a normal-form theorem, stated later here as
Theorem \ref{normformthm}. In brief, if $M$ and $Q$ make a constant angle $\theta_0$ with respect to compactified metrics and
$g$ is an admissible metric, then there is a diffeomorphism $\zeta: [0,\theta_0]_{\theta} \times W \hookrightarrow \tX$ (where
$W$ is a neighborhood of $S$ in $M$) such that
\begin{equation}
  \label{polarform}
  \zeta^*g = \frac{d\theta^2 + h_{\theta}}{\sin^2(\theta)},
\end{equation}
where $h_{\theta}$ is a smooth one-parameter family of smooth asymptotically hyperbolic metrics on $(W,S)$, and $\zeta$
is unique subject to some technical conditions. This generalizes the ordinary hyperbolic metric $g_+ = \frac{dy^2 + |dx|^2}{y^2}$
on the upper half-space $\mathbb{H}^{n + 1}$, which under the change of variables $\psi$ given by $y = \rho \sin \theta$ and $x^n = \rho \cos\theta$
takes the form
\begin{equation*}
  \psi^*g_+ = \frac{1}{\sin^2(\theta)}\left[ d\theta^2 + \frac{d\rho^2 + (dx^1)^2 + \cdots + (dx^{n - 1})^2}{\rho^2} \right].
\end{equation*}

Now, as mentioned before, the umbilic boundary condition $K_Q = \lambda g_+|_{TQ}$ (where $K_Q$ is the scalar second fundamental form) implies that $Q$ and $M$ do make constant angle
$\theta_0 = \cos^{-1}(-\lambda)$, a fact that remains true if $g_+$ is just admissible. In light of this fact and the preceding theorem,
it is natural to break the gauge of the Einstein equations by looking for an Einstein metric $g$ on the blowup in the form (\ref{polarform}).
In this gauge, the CMC umbilic condition at $Q$ takes the form $\partial_{\theta}h_{\theta}|_{\theta = \theta_0} = 0$, as shown in Lemma
\ref{klem}.
Let $M$ be a manifold with boundary $S$. Let $\rho$ be a defining function for $S$ in $M$ (throughout, we will use the same notation for the lift of $\rho$ to $[0,\theta_0] \times M$).
We define $\mathcal{M}(\theta_0,M)$ to be the set of families
$h_{\theta} (0 \leq \theta \leq \theta_0)$ of smooth AH metrics on $M$ such that
$\bar{h}_{\theta} = \rho^2h_{\theta}$ is smooth in $\theta$ (in case $n = 2$ or $n$ is odd) or such that it is smooth in $\theta$
and $\theta^n\log(\theta)$ (if $n \geq 4$ is even). See page \pageref{qdef} for full details. The following is the main theorem of
this paper. In the statement, $T = O_g(f)$ for $T$ a tensor field means $|T|_g = O(f)$. Here and throughout the paper, we take
$n \geq 2$.

\begin{theorem}
  \label{einthm}
  Let $M^n$ be a manifold with boundary, let $\lambda \in (-1,1)$, let $\rho$ be a defining function for $S = \partial M$ in $M$,
  and let $[h]$ be a conformal class on $M$. Set $\theta_0 = \cos^{-1}(-\lambda)$.
  Then there exists a one-parameter family $h_{\theta} \in \mathcal{M}(\theta_0,M)$ of smooth AH metrics on $M$, such that if $g$ is the normal-form
  metric
  \begin{equation*}
    g = \csc^2(\theta)[d\theta^2 + h_{\theta}]
  \end{equation*}
  on $\tX = [0,\theta_0]_{\theta} \times M$, then
  \begin{enumerate}
    \item $h_0 \in [h]$;\label{einfirst}
    \item $\partial_{\theta}\bar{h}_{\theta}|_{\rho = 0} = 0$, where $\bar{h}_{\theta} = \rho^2h_{\theta}$;\label{constrho0}
    \item the second fundamental form $K_{\tQ}$ of $\tQ \setminus \tS = \left\{ \theta_0 \right\} \times (M \setminus S)$ satisfies
      $K_{\tQ} = \lambda g|_{T\tQ}$; and\label{einsecondlast}
    \item the formal Einstein condition
      \begin{equation*}
        \ric(g) + ng = O_g(\rho^n)
      \end{equation*}
      is satisfied.\label{einlast}
  \end{enumerate}

  If $\theta_0$ is such that Equation (\ref{hyperchareq}) has no integral solutions $\nu$ when $s = n$, then in fact we may choose $h_{\theta}$ so that
  \begin{equation*}
    \ric(g) + ng = O_g(\rho^{\infty}).
  \end{equation*}

  In all cases, if $h_{\theta}, h_{\theta}'$ are two families satisfying the above conditions, then
  $\bar{h}_{\theta} - \bar{h}_{\theta}' = O(\rho^n)$. If Equation (\ref{hyperchareq}) has no integral solutions, as above, then two infinite-order solutions
  satisfying $\bar{h}_{\theta} - \bar{h}_{\theta}' = o(\rho^{n})$ satisfy $\bar{h}_{\theta} - \bar{h}_{\theta}' = O(\rho^{\infty})$.
\end{theorem}

Notice that the given data is only $(M,[h])$ and $\lambda$, and that we get an Einstein metric
in normal form that is unique up to order $n$. In particular, the induced metric $h_0 \in [h]$ is an AH representative of the conformal class
that is invariantly defined to order $n$ given only $[h]$ and $\lambda$.

Uniqueness in Theorem \ref{einthm} should hold without \ref{constrho0}, but we do not yet have a proof of this.

The proof of Theorem \ref{einthm} involves the complications generally associated with Einstein's equations: nonlinearity, tensor fields, and of course (as already mentioned)
gauge invariance. Before taking up its proof, then, we analyze a simpler linear problem that already raises several of the distinctive issues of analysis in our setting; namely,
we consider the formal expansion, at the corner of a constant-angle CAH space, of eigenfunctions of the Laplacian with inhomogeneous Dirichlet condition at the infinite boundary $\tM$ and a homogeneous Robin condition
at the finite boundary, $Q$. (Besides being a natural boundary condition, this will prove to be of direct relevance to the Einstein problem.) In the usual way, we are using Dirichlet condition
to mean prescribing the coefficient at the power of the leading indicial root at $\theta = 0$; see \cite{gz03}.
The problem we wish to solve is 
\begin{equation}
  \label{lineq}
  \Delta_gu + s(n - s)u = 0,
\end{equation} where the expression of the eigenvalue in terms of the spectral parameter $s > \frac{n}{2}$ is traditional (see e.g. \cite{mm87},\cite{gz03}). 
Thus, if $s = n$, and we are looking for harmonic functions, then indeed our boundary condition at $\tM$ is an inhomogeneous Dirichlet condition.

The analysis of the linear problem proceeds in several steps. The key idea, as in general for such constructions, is to determine and then study the
\emph{indicial operator} of $\Delta + s(n - s)$, which is an operator on $C^{\infty}(\tS)$ defined by $I_{s,\nu}(u) = \rho^{-\nu}
\left[(\Delta_g + s(n - s))(\rho^{\nu}\tilde{u})\right]|_{\rho = 0}$,
where $\rho$ is a particular defining function for $\tS$ in $\tX$ and $\tilde{u}$ is an extension of $u$ to $\tX$.
However, we here meet a significant difference from the usual AH case: whereas the indicial operator
is there an algebraic operator, due to the edge structure of $g$ at $\tS$, it here restricts to
a second-order ordinary differential operator on each fiber of $\tS$, with a regular singularity at $\theta = 0$:
\begin{equation}
  \label{scalindop}
  I_{s,\nu}(u) = \sin^2(\theta)\partial_{\theta}^2u
    + (1 - n)\sin(\theta)\cos(\theta)\partial_{\theta}u
    + \nu(\nu + 1 - n)\sin^2(\theta)u + s(n - s)u
\end{equation}
For any $\nu$, the indicial roots of $I_{s,\nu}$ at $\theta = 0$ are $n - s$ and $s$. The key content of Section \ref{odesubsec},
then, is an analysis of this operator with the ``Dirichlet'' boundary condition $u(\theta) = o(\theta^{n - s})$ at $\theta = 0$ and the Robin condition 
$\partial_{\theta}u(\theta_0) + (s - n)\cot(\theta_0)u(\theta_0) = 0$ at $\theta = \theta_0$, for $0 < \theta_0 < \pi$.
We study the mapping properties of the Green's operator, and also study the
indicial roots of the Laplacian, or values of $\nu$ for which the indicial operator fails to be injective
with the given boundary conditions. The latter is equivalent to studying the singular Sturm-Liouville eigenvalue problem for the operator
$L = -\partial_{\theta}^2 + (2s - n - 1)\cot(\theta)\partial_{\theta} + (s - 1)(s - n)$. We estimate the lowest eigenvalue, and can characterize the eigenvalues in general as
the roots of an equation involving hypergeometric functions.
In the case that $\theta_0 = \frac{\pi}{2}$, we can calculate explicitly that they are $\lambda_{k} = \nu_k(\nu_k + 1 - n)$ for
$\nu_k = s + 2k$ ($k \geq 0$). For this reason, we restrict our full analysis to the case $\theta_0 = \frac{\pi}{2}$: although similar ideas would
apply in the general case, it would be difficult to be as specific as we can be when we know the eigenvalues explicitly.
For that case, in Section \ref{lapsec} we formally construct a harmonic function order by order in $\rho$, solving at each order 
$j$ an equation of the form $I_ju = f_j$. When $2s \notin \mathbb{Z}$, we can construct a unique solution iteratively without ever encountering an indicial root, as
$s - (n - s)$ is likewise non-integral. Otherwise,
when $j = s + 2k$ is
an indicial root, we show that we can proceed by including powers of $\log(\rho)$ in the solution, although uniqueness is lost.
In such a case, a formal solution could be uniquely parametrized by $u|_{\tM}$ and by $\left\{ \upsilon_k \right\}_{k = 0}^{\infty}$, where $\upsilon_k \in C^{\infty}(S)$ 
parametrizes the formal freedom at order $s + 2k$.

Depending on $s$, we define $\mathcal{P}_s(\tX)$ to be either smooth functions, or those functions on $\tX$ that have an asymptotic expansion in $\rho, \theta, \theta^{2s - n}\log(\theta)$
to the first power, and $\rho^{2s - n + 2k}\log(\rho)^k (k \geq 0)$, where $\rho$ is a defining function for $\tM \cap \tS$ in $\tM$; see page \pageref{Px} for a precise definition. Our result
is as follows.

\begin{theorem}
  \label{lapthm}
   Let $(\tX^{n+1},\tM,\tQ,\tS)$ be the blowup of the cornered space $(X,M,Q)$, with $g = b^*g_X$ an admissible metric such that $M$ and $Q$ make constant angle
   $\frac{\pi}{2}$ with respect to $g_X$. Let $\psi \in C^{\infty}(\tM)$ and $s > \frac{n}{2}$. There exists $F \in \mathcal{P}_s(\tX)$ such that, if $u = \rho^{n - s}\sin^{n - s}(\theta)F$, then
   $\Delta_gu + s(n - s)u = O(\rho^{\infty})$, such that $F|_{\tM} = \psi$, and such that $\partial_{\nu}u + s(n - s)\cot(\theta_0)u\equiv 0$ along $Q$, where $\partial_{\nu}$ is the normal
   derivative at $\tQ$. If $2s \notin \mathbb{Z}$, then such $F$ is unique to infinite order. Otherwise,
   if $F_1$ and $F_2$ are two such functions, then $F_1 - F_2 = O(\rho^{2s - n})$, and $\rho^{n - 2s}(u_1 - u_2)|_{\tS} = \upsilon_0\sin^{2s - n}\Theta$, where
   $\upsilon_0 \in C^{\infty}(\tS)$ is constant on fibers, and $\Theta$ is the natural angle function on $\tS$ induced by $g$.
\end{theorem}
We note that the power $n - s$ of $\rho$ appearing in the definition of $u$ in this theorem is not fully determined by the problem, and could be chosen differently depending on $s$. This is discussed
more after the proof of the theorem in Section \ref{lapsec}. (In a paper in preparation, we will consider existence and polyhomogeneity for eigenfunctions of the Laplacian, with boundary conditions as above.)

The proof of Theorem \ref{einthm} is conceptually similar to that of the simpler Theorem \ref{lapthm}, but significant complications
arise in the Einstein setting, as mentioned earlier. We define an indicial operator for the Einstein operator $E(g) = \ric(g) + ng$, as in the scalar case, by
$I^{\gamma}(\varphi) = \rho^{-\gamma}(E(g + \rho^{\gamma}\varphi) - E(g))|_{\rho = 0}$, where $\varphi$ is a section
of an appropriate bundle; and as in \cite{gl91}, we decompose it
into its irreducible parts, in this case seven of them. Once again, and unlike in that paper, the indicial operator is a second-order
system of
regular singular ordinary differential operators as opposed to algebraic operators.
As in the AH case, the part of the indicial operator acting on the trace-free part of the metric perturbations tangent to $S$
is identical with the indicial operator of the scalar Laplacian. We construct the Einstein metric term-by-term in $\rho$. At each order, this gives
us a system of second-order regular singular ordinary differential equations to solve, which is overdetermined because of the gauge-broken form
(\ref{polarform}). An additional complication in the analysis comes from the fact that, since the Einstein metric is unique to order $n$,
then as observed above and unlike in the case of the usual AH Einstein existence problem,
the induced metric $h_0$ in the conformal infinity is uniquely determined and
cannot be chosen arbitrarily in the conformal class. These two problems are solved in tandem. For our boundary data,
we take $h \in [h]$ to be arbitrary (but AH), and then impose the boundary condition $h_0 = \chi h$, where $\chi \in C^{\infty}(M)$ is some
scalar function to be determined order by order in $\rho$ along with $g$. Thus the induced metric is 
determined simultaneously with the metric $g$. At each order,
we use four of the seven irreducible parts of the indicial operator to solve uniquely for the perturbation of the metric at that order.
We then use the Bianchi identity to show that the remaining three equations are also satisfied. However, this turns out to be
true only for a unique choice of the perturbation of $\chi$, and thus we get uniqueness both for $g$ and $\chi h$ (within $[h]$) up to order $n$.

The behavior of the system changes at order $n$: at that order, $\chi$ is no longer determined, but may be chosen freely, and different components
of the indicial operator must be used to determine $\varphi$. If this can be done successfully (i.e., if $n$ is not an indicial root), then
at higher orders, the system once again acts as at lower powers, and $\chi$ is uniquely determined at each step.
The trace-free part of the indicial operator has a set of eigenvalues going to infinity; in the special case that $\lambda = 0$ ($\theta_0
= \frac{\pi}{2}$), the first of these is at order $n$, as mentioned above, and this allows us to identify a conformal hypersurface
invariant obstructing smooth existence.

\begin{theorem}
  \label{invthm}
  Let $M^n$ ($n \geq 2$) be a manifold with boundary $S$, and $\tau$ a smooth
  metric on $M$. Let $[h]$ be the AH conformal class corresponding to $[\tau]$.
  There is a generically nontrivial symmetric, trace-free $2$-tensor field $\mathcal{K}(\tau)$
  on $S$, defined by (\ref{hinvform}), whose nonvanishing obstructs the formal existence
  of a smooth normal-form metric $g = \csc^2(\theta)[d\theta^2 + h_{\theta}]$ on $[0,\frac{\pi}{2}] \times M$ satisfying
  \ref{einfirst} - \ref{einsecondlast} from Theorem \ref{einthm}, and also satisfying
  $\ric(g) + ng = O_g(\rho^{n + 1})$.

  Moreover, if $\hat{\tau} = \Omega^2\tau$ for $\Omega \in C^{\infty}(M)$, then $\mathcal{K}(\hat{\tau}) = (\Omega|_S)^{2-n}\mathcal{K}(\tau)$.
\end{theorem}

Theorem \ref{einthm} is concerned with smooth formal series in $\rho$.
As we will show, the lowest non-negative indicial root $\gamma_0$ in $\rho$ of the Einstein operator satisfies $n - 1 < \gamma_0 < n$ 
if $\theta_0 > \frac{\pi}{2}$, $\gamma_0 = n$ 
if $\theta_0 = \frac{\pi}{2}$, and $n < \gamma_0 < \infty$ if $\theta_0 < \frac{\pi}{2}$.
Thus, if $\frac{\pi}{2} < \theta_0 < \pi$, then we would expect additional solutions with leading asymptotics $\rho^{\gamma}$, where
$n - 1 < \gamma < n$. If $\theta_0 \leq \frac{\pi}{2}$, then uniqueness would hold mod $O(\rho^n)$ even allowing non-integral powers of $\rho$.
If $\theta_0 = \frac{\pi}{2}$, then a term of the form $\rho^n\log(\rho)$ would generically appear, as suggested by Theorem \ref{invthm} above. The form
of the solution to higher order depends on indicial roots, which depend on $\theta_0$. Uniqueness fails at order $n$ in every case, however.

The paper is organized as follows.

In Section \ref{backsec}, we define cornered AH spaces and review the results on them from \cite{m16}, including the definition of
admissible metrics, the 0-edge structure on CAH spaces, and the normal form. 

In Section \ref{smoothsec}, we study smooth Einstein
metrics and deduce the compatibility conditions for smoothness discussed earlier, under the assumption of the CMC umbilic condition
at the finite boundary.

In Section \ref{laplsec}, we study the scalar Laplacian on constant-angle CAH spaces. After calculating the scalar Laplace operator in coordinates on the blowup,
we then compute its indicial operator (\ref{scalindop}), and prove theorems about its eigenvalues and the mapping properties of its Green's operator. 
We do this for general $\theta_0$ and arbitrary integral powers of $\log(\theta)$, since although these features are unnecessary
for our analysis of the linear scalar problem with $\theta_0 = \frac{\pi}{2}$, they will be used in
the nonlinear Einstein setting. We then prove Theorem \ref{lapthm}.

Finally, in Section \ref{einsec}, we study formal existence for aribtrary $S$, enlarging the
class of metrics from those smooth on $X$ to those polyhomogeneous on $\tX = [0,\theta_0]_{\theta} \times M$ in 
the form (\ref{polarform}); and prove Theorems \ref{einthm} and \ref{invthm}. We also discuss an approach to finding
formal solutions for which the Einstein tensor vanishes at both $\tS$ and $\tM$.

\subsection*{Acknowledgements} This is doctoral work under the supervision of C. Robin Graham at the University of Washington (UW). I am most grateful to him for suggesting this and related problems, and
for the really extraordinary time and attention he has given to answering questions and making suggestions large and small.
I am also grateful to Andreas Karch for bringing the topic to both of our attention
in the first place, to John Lee and Daniel Pollack for numerous helpful conversations, and to Hart Smith
for financial support. I also greatly appreciate the excellent environment and support at Princeton University while I completed the paper. This research was partially supported by the National Science Foundation under RTG Grants
DMS-0838212 at UW and DMS-1502525 at Princeton, and Grant DMS-1161283 at UW.

%% file: tex/back.tex
\section{Background}\label{backsec}

Recall that an asymptotically hyperbolic (AH) space is a compact manifold $X^{n + 1}$ with
boundary $\partial X = M^n$, equipped on the interior with a metric $g$ such that, for any
defining function $\varphi$ of $M$ in $X$, the metric $\varphi^2g$ extends smoothly to all
of $X$; and such that, for any such defining function, $|d\varphi|^2_{\varphi^2g} = 1$
along $M$.

We now review the definition, properties, and blowup of a cornered asymptotically hyperbolic
(CAH) metric, as given in \cite{m16}.

\begin{definition}
  \label{intrindef}
  A \emph{cornered space} is a smooth manifold with codimension-two corners, $X^{n + 1}$, such that
  \begin{enumerate}[(i)]
    \item there are submanifolds with boundary $M^n \subset \partial X$ and $Q^n \subset \partial X$ of the boundary
      $\partial X$, such that $\emptyset \neq S = M \cap Q$ is the mutual boundary, and is the entire codimension-two corner of $X$, and such that
      $\partial X = M \cup Q$; and
    \item the corner $S \subset M$ is a smooth, compact hypersurface in $M$.
  \end{enumerate}
  We denote a cornered space by $(X,M,Q)$, and we set $\mathring{X} = X \setminus(Q \cup M)$.

  Given a cornered space $(X,M,Q)$, a smooth (resp. $C^k$) \emph{cornered conformally compact metric} on $X$ is a smooth Riemannian metric
  $g_+$ on $X \setminus M$ such that, for any smooth defining function $\varphi$ for $M$, the metric $\varphi^2g_+$ extends
  to a smooth (resp. $C^k$) metric on $X$. We call such a metric a \emph{cornered asymptotically hyperbolic (CAH) metric} if for some
  (hence any) such defining function $\varphi$, the condition $|d\varphi|_{\varphi^2g_+} = 1$ holds along $M$.

  A smooth (resp. $C^k$) \emph{cornered asymptotically hyperbolic (CAH) space} is a cornered space $(X,M,Q)$ together with a smooth (resp. $C^k$)
  CAH metric $g_+$. We denote such a space by $(X,M,Q,g_+)$. The definition for cornered conformally compact space is analogous.
\end{definition}
For a cornered conformally compact space $(X,M,Q,g_+)$, the \emph{conformal infinity} $[h]$ is the conformal class
$[\varphi^2g_+|_{TM}]$ on $M$, where $\varphi$ is a defining function for $M$. Notice that a consequence of the
fact that $X$ is a manifold with corners
is that the boundary components $M$ and $Q$ intersect transversely.

For each $x \in S$, we define $\theta_0(x)$ to be the angle between $M$ and $Q$ at $X$ with respect to $\varphi^2g_+$, where $\varphi$ is any smooth defining
function for $M$. Plainly $\theta_0 \in C^{\infty}(S)$.

As described in more detail in \cite{m16}, we may blow up the cornered space $X$ along $S$ as follows, along the lines of \cite{mel08}.
Let $s \in S$, and set $N_sS = T_sX/T_sS$, which is a two-dimensional vector space. Set $NS = \sqcup_{s \in S}N_sS$, and let $N_+S \subset NS$ be the
vectors pointing into $X$ (including into $M, Q$, or $S$). Thus, $N_+S$ is a bundle with fiber a closed cone in $\mathbb{R}^2$ and base $S$. Let $\tS =
(N_+S\setminus \left\{ 0 \right\}) / \mathbb{R}^+$, which is the total space of a fibration over $S$ with fiber $[0,1]$.
Now define the \emph{blow-up space} $\tX$ by $\tX = (X \setminus S) \sqcup \tS$, and the blow-down map $b:\tX: \tX \to X$ by
$b(x) = x$ if $x \in X \setminus S$ and $b(\ts) = \pi(\ts)$ for $\ts \in \tS$, where $\pi$ is the basepoint projection.
Then $\tX$ has a unique smooth structure as a manifold with corners of codimension two such that $b$ is smooth, $b|_{\tX \setminus \tS}:
\tX \setminus \tS \to X \setminus S$ is a diffeomorphism, and $db|_{\ts}$ has rank $n$ for $\ts \in \tS$.
We set $\tM = \overline{b^{-1}(M\setminus S)}$ and $\tQ = \overline{b^{-1}(Q\setminus S)}$.

Recall that an edge structure on a manifold with boundary is a fibration of the boundary, and the associated edge vector fields are the vector
fields that are tangent to the fibers at the boundary (\cite{maz91}). 
An important special case is a 0-structure (\cite{mm87}), for which the boundary fibers are points
and the edge vector fields are those that vanish at the boundary. The vector fields in that setting may be viewed as sections of the 0-bundle,
${}^0TM$, with dual bundle ${}^0T^*M$. On our blowup space $\tX$, the blown-up face $\tS$ is the total space
of the fibration $b|_{\tS}:\tS \to S$ with interval fibers, while we can view $b|_{\tM}:\tM \to M$ as a fibration whose fibers are points. We will
refer to the structure defined by these two fibrations as a 0-edge structure, and the associated 0-edge vector fields are the smooth vector fields
on $\tX$ which are tangent to the fibers at $\tS$, and which vanish at $\tM$.

The 0-edge vector fields may be easily expressed in appropriate local coordinates. Let $\theta$ be a defining function for $\tM$ whose restriction to
each fiber of $\tS$ is a fiber coordinate taking values in $[0,\pi)$; let $\rho$ be any defining function for $\tS$; and locally let
$x^s, 1 \leq s \leq n - 1$, be the lifts to $\tX$ of functions on $X$ that restrict to local coordinates on $S$. Then the vector fields
\begin{equation*}
  \sin\theta \frac{\partial}{\partial \theta}, \quad \rho\sin\theta \frac{\partial}{\partial x^s}, \quad \rho\sin\theta \frac{\partial}{\partial \rho}
\end{equation*}
span the 0-edge vector fields over $C^{\infty}(\tX)$. As in the usual edge case, there is a well-defined vector bundle ${}^{0e}T\tX$
whose smooth sections are the 0-edge vector fields. The smooth sections of the dual bundle ${}^{0e}T^*\tX$ are locally spanned by
\begin{equation}
  \label{dualframe}
  \frac{d\theta}{\sin\theta},\quad\frac{dx^s}{\rho\sin\theta}, \quad\frac{d\rho}{\rho\sin\theta}.
\end{equation}
By a 0-edge metric we will mean a smooth positive definite section $g$ of $S^2({}^{0e}T^*\tX)$. This is equivalent to the condition that
locally $g$ may be written as
\begin{equation*}
  g = \left( \begin{array}{ccc}
    \frac{d\theta}{\sin\theta},&\frac{dx^s}{\rho\sin\theta},&\frac{d\rho}{\rho\sin\theta}\end{array}\right)
    G
  \left(
  \begin{array}{ccc}
    \frac{d\theta}{\sin\theta}\\
    \frac{dx^s}{\rho\sin\theta}\\
    \frac{d\rho}{\rho\sin\theta}
  \end{array}
    \right),
\end{equation*}
where $G$ is a smooth, positive-definite matrix-valued function on $\tX$.

We may now define admissible metrics.

\begin{definition}
  \label{admis}
  An admissible metric on $\tX$ is a 0-edge metric $g$ on $\tX$ which can be written in the form
  \begin{equation*}
    g = b^*g_+ + \mathcal{L},
  \end{equation*}
  where $g_+$ is a smooth cornered asymptotically hyperbolic metric on $X$ and $\mathcal{L}$ is a smooth section of
  $S^2({}^{0e}T^*\tX)$ that vanishes on $\tS$ and $\tM$.
\end{definition}
Note that it was shown in \cite{m16} that $b^*g_+$ is, itself, always a 0-edge metric.

Since $b|_{\tX \setminus (\tM \cup \tS)}:\tX \setminus (\tM \cup \tS) \to X \setminus M$ is a diffeomorphism, an admissible
$g$ uniquely determines a smooth metric $g_X$ on $X \setminus M$ satisfying $b^*g_X = g$ on
$\tX \setminus (\tM \cup \tS)$. Since $\mathcal{L}$ vanishes on $\tS$ and $\tM$, it is not hard
to see that $g_X$ is a $C^0$ CAH metric on $X$. Thus we will call a metric $g_X$ on $X \setminus M$ an admissible metric on $X$ if
$b^*g_X$ extends to an admissible metric on $\tX$.

Observe that an admissible metric $g_X$ on $X$ determines a well-defined angle function $\Theta$ on the blown-up face $\tS$,
which serves as a smooth fiber coordinate.
Let $\ts \in \widetilde{S}$, with $s = b(\ts) \in S$. Then, under one interpretation, 
$\ts$ naturally represents a hyperplane $P_{\ts}$ in $T_sX$
containing $T_sS$. The angle $\Theta(\ts)$ between $P_{s}$ and $T_sM$ is well-defined. 
It can be computed as follows: let $\varphi$ be any defining function for $M$, 
and $\bar{g}_X = \varphi^2g_X$. Let $\bar{\nu}_M \in T_sM$
be normal to $T_sS$, inward pointing in $M$, and unit $\bar{g}_X$-length. 
Similarly, let $\bar{\nu}_{P_{\ts}}$ be inward-pointing in $P_{\ts}$, normal to $T_sS$, and unit length. Then
$\Theta(\ts) = \cos^{-1}(\bar{g}_X(\bar{\nu}_M,\bar{\nu}_{P_{\ts}}))$. We could also have defined $\Theta$ using
$g_+$, and in particular, it is clear that $\Theta \in C^{\infty}(\tS)$.
It is easy to show that this is defined independently of $\varphi$. Thus, $\Theta$ is well-defined.

It will be convenient to recall that it was shown in Section 2 of \cite{m16} that there are
coordinates $(\theta,x^s,\rho)$ in which $\theta$ is a defining function for $\tM$ and restricts at $\Theta$ to
$\tS$, in which $\rho$ is a defining function for $\tS$, and in which $\left\{ x^s \right\}$ restrict as coordinates
to $\tS \cap \tM$ and are constant on the fibers of $\tS$, such that the metric $g$ takes the form
\begin{equation}
  \label{g}
g_{ij} = \csc^2(\theta)\left(\begin{array}{ccc}
    1 + O(\rho\sin\theta) & O(\sin \theta) & O(\sin \theta)\\
    O(\sin \theta) & \rho^{-2}k_{\rho} + O(\rho^{-1}\sin\theta) & O(\rho^{-1}\sin\theta)\\
    O(\sin\theta) & O(\rho^{-1}\sin\theta) & \rho^{-2} + O(\rho^{-1}\sin\theta)
  \end{array}\right),
\end{equation}
and
\begin{equation}
  \label{ginv}
  g^{ij} = \sin^2(\theta)\left(\begin{array}{ccc}
    1 + O(\rho\sin\theta) & O(\rho^2\sin\theta) & O(\rho^2\sin\theta)\\
    O(\rho^2\sin\theta) & \rho^2k_{\rho}^{-1} + O(\rho^3\sin\theta) & O(\rho^3\sin\theta)\\
    O(\rho^2\sin\theta) & O(\rho^3\sin\theta) & \rho^2 + O(\rho^3 \sin\theta)\end{array}\right).
\end{equation}
Much more can be said in the case that $M$ and $Q$ make constant angle, i.e., $\Theta$ is constant on $\tS \cap \tQ$.

\begin{theorem}[\cite{m16}, Corollary 1.5]
  \label{normformthm}
  Let $(X,M,Q,g)$ be an admissible CAH space in which $M$ and $Q$ make constant angle $\theta_0$ with respect to compactifications of $g$;
  and let $b:\tX \to X$ be the blowup of $X$. Then for sufficiently small neighborhoods $W$ of $S$ in $M$, there exist a unique
  neighborhood $\tU$ of $b^{-1}(W)$ in $\tX$ and a unique diffeomormphism $\zeta:[0,\theta_0]_{\theta} \times W \to \tU$
  such that $\zeta|_{\{0\}\times W} = \id_W$ and
  \begin{equation}
    \zeta^*g = \frac{d\theta^2 + h_{\theta}}{\sin^2(\theta)},
  \end{equation}
  where $h_{\theta} (0 \leq \theta \leq \theta_0)$ is a smooth one-parameter family of smooth asymptotically hyperbolic metrics on $(W,S)$,
  and such that $b^{-1}(M) = \zeta(\left\{ \theta = 0 \right\})$ and $b^{-1}(Q) = \zeta\left( \theta = \theta_0 \right))$. Moreover,
  $\partial_{\theta}\bar{h}_{\theta}|_{\rho = 0} = 0$, where $\bar{h}_{\theta} = \rho^2h_{\theta}$, and $\rho$ is any defining function
  for $S$ in $W$.
\end{theorem}

(Again we use $\rho$ both for the function on $W$ and for its pullback to $[0,\theta_0] \times W$.)

As mentioned in the introduction, this theorem provides the gauge we will use in looking for an Einstein metric.

The following corollary, also proved in \cite{m16}, is a straightforward consequence that will also be of use to us.

\begin{corollary}
  \label{polarcor}
  Let $(\tX,\tM,\tQ,\tS)$, $(X,M,Q)$, and $g$ be as in Theorem \ref{normformthm}, with again a constant angle
  $\theta_0$ between $Q$ and $M$, and let $[k]$ be the conformal class induced on $\tS \cap \tM$ (thus on $S$
  by any compactification of $g$.
  For any $k \in [k]$ and for sufficiently small
  $\varepsilon > 0$, there is a neighborhood $\tU$ of $\tS$ in $\tX$
  and a unique diffeomorphism $\chi:[0,\theta_0]_{\theta} \times S \times [0,\varepsilon)_{\rho} \to \tU$ such that
  $b\circ\chi|_{ \left\{ 0 \right\}\times S \times \left\{ 0 \right\}} = \id_S$ and
    \begin{equation}
      \chi^*g = \frac{d\theta^2 + h_{\theta}}{\sin^2\theta},
    \end{equation}
    where $h_{\theta}$ is a smooth one-parameter family of smooth AH metrics on $S \times [0,\varepsilon)$ with
    \begin{equation*}
      h_0 = \frac{d\rho^2 + k_{\rho}}{\rho^2},
    \end{equation*}
    where $k_{\rho}$ is a smooth one-parameter family of smooth metrics on $S$ with $k_0 = k$, and where 
    $\tM = \chi(\left\{ \theta = 0 \right\})$, $\tQ = \chi(\left\{ \theta = \theta_0 \right\})$, and
    $\tS = \chi(\left\{ \rho = 0 \right\})$. Moreover, $\partial_{\theta}(\bar{h}_{\theta}|_{\rho = 0}) = 0$,
    where $\brh = \rho^2h$.
\end{corollary}

Finally, it will be helpful\label{confclassrev} to review the relationship between two types of conformal classes
on a manifold with boundary. Let $M$ be a manifold with boundary $S$. The first type is the usual conformal class,
$[\tau]$, where $\tau$ is a smooth metric on $M$. Here, $[\tau]$ is the family of metrics $\tau'$ such that
$\tau' = \Omega^2\tau$ for some nonvanishing $\Omega \in C^{\infty}(M)$.
The second type is an AH conformal class, $[h]$, where $h$ is an AH metric on $M$. Here, $[h]$ is precisely the
set of metrics $\Psi^2h$, where $\Psi \in C^{\infty}(M)$ is nonvanishing and $\Psi|_{S} \equiv 1$.

Observe that there is a one-to-one correspondence between ordinary conformal classes $[\tau]$ and
AH conformal classes $[h]$. Given a conformal class $[\tau]$, let $\tau \in [\tau]$ and let
$\varphi \in C^{\infty}(M)$ be any defining function for $S$ in $M$ such that $|d\varphi|_{\tau} = 1$ along $S$.  Set
$h = \varphi^{-2}\tau$. Then the conformal class
$[h]$ is independent of the choices of $\varphi$ and of $\tau$. To see this, suppose $\psi$ is some other
defining function satisfying $|d\psi|_{\tau} = 1$ along $S$. Then $\psi^{-2}\tau = \psi^{-2}\varphi^2(\varphi^{-2}\tau)$. But
$\Psi = \psi^{-2}\varphi^2$ extends smoothly to all of $M$, and $\Psi|_{S} \equiv 1$ by the choice of $\varphi$, $\psi$.
Thus $[h]$ does not depend on $\varphi$. Similarly, suppose $\tau' = \Omega^{2}\tau$, and let $\varphi'$ be a defining function for $S$ such
that $|d\varphi'|_{\tau'} = 1$ along $S$. Then it is easy to check that if $\varphi = \Omega^{-1}\varphi'$, then
$|d\varphi|_{\tau} = 1$ along $S$; and $\varphi^{-2}\tau = (\varphi')^{-2}\tau'$. Thus the map taking $[\tau]$ to $[h]$ is well-defined.
It is an easy exercise to reverse these steps and show that it is a bijection.

\subsection*{Notation}
Throughout, $X$ will be of dimension $n + 1$, where unless otherwise stated, $n \geq 2$.
We use index notation in polar coordinates (such as those given by Theorem \ref{normformthm}). When doing so, except where stated otherwise, we let
$\rho$ be a special defining function for $S$ in $M$ corresponding to $h_0$, i.e., a function
such that $h_0 = \frac{d\rho^2 + k_{\rho}}{\rho^2}$ (see \cite{gl91}). Then the space $[0,\theta_0] \times W$ given by our normal-form theorem
decomposes into a product $[0,\theta_0]_{\theta} \times S \times [0,\varepsilon)_{\rho}$. The coordinate index
$0$ will refer to the first factor $\theta$, and $n$ will refer to the last factor $\rho$. The indices
$1 \leq s,t \leq n - 1$ will refer to local coordinates on $S$, while the indices $0 \leq i, j \leq n$ will run over all $n + 1$ coordinates.
Finally, $1 \leq \mu,\nu \leq n$ will run over $S$ and $\rho$.

The metric $g$ will be used to raise and lower indices, except that $g^{ij}$ is the inverse metric, and so also for $h^{ij}$ and any other metrics
with raised indices. We write $\bar{g} = \rho^2\sin^2(\theta)g$
and $\bar{h}_{\theta} = \rho^2h_{\theta}$, where again $h_{\theta}$ is as in Theorem \ref{normformthm}. Note that $\bar{g}$ is degenerate along
$\tS$.

If $A$ is a covariant $k$-tensor, we write $A = O_g(f)$ to indicate that $|A|_g = O(f)$; or equivalently, that if
$Y_1,\dots,Y_k$ are $g$-unit vector fields, then the condition $A(Y_1,\dots,Y_k) = O(f)$ holds with constant independent of the $Y_i$.
Similarly, if $Y$ is a vector field, we write $Y = O_g(f)$ to indicate $|Y|_g = O(f)$. Note that this condition 
is independent of the particular admissible metric $g$.

%% file: tex/smooth.tex
\section{Compatibility Conditions for Smooth Solutions}\label{smoothsec}

In this section, we prove basic results about Einstein spaces satisfying our
boundary condition $K_Q = \lambda g|_{T\tQ}$, such as the fact that $\lambda \in (-1,1)$ necessarily. We also study
compatibility conditions that are imposed on the corner $S$ if the Einstein metric $g$ is
to have smooth compactification. These prove to be severe with the given boundary conditions.

Suppose a manifold $M^n$ with boundary $S$ is given, and is equipped with a Riemannian conformal class
$[h]$, which throughout this section will be taken to be a smooth conformal class. In this section, we will investigate, if $(M,[h])$ is to be the conformal infinity of a \emph{smooth}
CAH Einstein space $(X,M,Q,g_+)$ satisfying $K_Q = \lambda g|_{TQ}$, what we can deduce from $M$ and $S$ about the developments of
$g$ and $Q$; and what constraints are put on $S$. Because we will frequently want
to use indices on $g_+$, and no blowups occur, for convenience we break in this section with the notation used in the rest of this paper
and write $g = g_+$. Elsewhere, $g$ will remain an admissible metric on $\tX$. Our use of indices will also vary slightly from
the remainder of the paper, as we describe below.
As mentioned in the introduction, this situation has been 
analyzed to first order in \cite{ntu12} with the \emph{a priori} assumption that $g$
is the hyperbolic metric. We remove this assumption to investigate the general case. In this section, we will work directly on
$(X,M,Q)$ and not the blowup.

First, we state a classical result.

\begin{theorem}
  \label{umbthm}
  Let $n \geq 2$, and let
  $Q^n \subset \mathbb{H}^{n + 1}$ be an umbilic hypersurface. 
  Then the umbilic coefficient $\lambda_p = \lambda$ is the same at all points $p \in Q$.
  Moreover, as a subset of $\mathbb{R}^{n + 1}$, $Q$ is either part of a sphere or part of a hyperplane. In particular,
  $Q$ falls into one of the following classes:
  \begin{enumerate}[(i)]
    \item $Q$ is part of a geodesic sphere. In this case, $Q$ is part of a Euclidean sphere entirely contained in
      $\mathbb{H}^{n + 1}$. In this case, $|\lambda| > 1$;
    \item $Q$ is a horosphere: either it is part of a Euclidean sphere contained entirely in $\mathbb{H}^{n + 1}$ except for one
      point, which lies on the boundary $\mathbb{R}^{n}$; or it is part of a Euclidean
      plane contained in $\mathbb{H}^{n + 1}$ and parallel to $\mathbb{R}^n$. In either case, $|\lambda| = 1$;
    \item $Q$ is totally geodesic. In this case, $M$ is either part of a Euclidean hemisphere that meets the boundary 
      $\mathbb{R}^n$ normally, or part of a Euclidean hyperplane that meets the boundary normally. In either event,
      $\lambda = 0$; or
    \item $Q$ is part of a Euclidean sphere or Euclidean hyperplane that intersects the boundary
      $\mathbb{R}^n$ non-normally. In the case of a sphere, $M$ lies in the upper part
      of a sphere of radius $r$, such that its center $(y,x)$ satisfies
      $0 < |y| < r$. Either case is called an equidistant hypersurface,
      because it lies at fixed distance from the totally geodesic hypersurface $S$ such that
      the intersection of $S$ with $\mathbb{R}^n$ is the same as that of the plane or sphere in which $M$ lies.
      In this situation, $0 < |\lambda| < 1$.
  \end{enumerate}
\end{theorem}

See \cite[chap. IV.7]{spi99} for a discussion and proof. Now this theorem implies that if an umbilic hypersurface $Q$ in hyperbolic space
intersects the boundary $M = \mathbb{R}^n$ at infinity, the intersection $S$ must be a hyperplane or a sphere. This duplicates the result found in
\cite{ntu12}, although it depends on global geometry. On the other hand, and unlike the finding of that paper, this shows that the result
must hold even for $n = 2$.

We next prove a lemma in a more general context.

\begin{lemma}
  \label{subint}
  Let $(X,g)$ be a Riemannian manifold with a smooth metric $g$, and with embedded hypersurfaces $Q$ and $M$ that intersect
  transversely in an embedded submanifold $S$. Denote by $K$ the scalar second fundamental form with respect to a fixed
  unit normal vector.
  By $K_S$, we will mean the second fundamental form of $S$
  considered as a submanifold of $M$, while $K_M$ and $K_Q$ will
  mean the second fundamental forms of $M$ and $Q$ as submanifolds of $X$. Then

  \begin{equation}
    \label{subinteq}
    K_S =  \left.\frac{K_Q - \langle \nu_Q,\nu_M\rangle K_M}{\langle\nu_Q,\nu_S\rangle}\right|_{TS},
  \end{equation}
  where $\nu_S$ is the unit $g$-normal to $S$ in $M$, and $\nu_Q$ and $\nu_M$ are the unit $g$-normal vectors to $Q$ and
  $M$ in $X$.
\end{lemma}

\begin{proof}
  Denote by $II$ the vector second fundamental form, with the same conventions as for $K$ in the statement:
  so, for example, $II_S$ is the vector second fundamental form of $S$ as a submanifold of $(M,g|_{TM})$.
  We compute each second fundamental form. We will use $P$ to denote
  an orthogonal projection operator, while
  $\nabla$ will denote the Levi-Civita connection on $X$. 
  First, for $p \in S$ and $X, Y \in
  T_pS$, extended smoothly to a neighborhood,
  \begin{align*}
    II_S(X,Y) &= P_{TS^{\bot}}P_{TM}\nabla_XY\\
    &= \langle \nabla_XY,\nu_S\rangle \nu_S.
  \end{align*}
  Next,
  \begin{align*}
    II_M(X,Y) &= P_{TM^{\bot}}\nabla_XY\\
    &= \langle \nabla_XY,\nu_M\rangle\nu_M.
  \end{align*}
  Now write $\nu_Q = \nu_Q^S\nu_S + \nu_Q^M\nu_M$ (such a decomposition
  must be possible at $S$ since $TS \subset TQ$). Then we have
  \begin{align*}
    II_Q(X,Y) &= \langle \nabla_XY,\nu_Q\rangle\nu_Q\\
    &= (\nu_Q^S\langle\nabla_XY,\nu_S\rangle + \nu_Q^M\langle
    \nabla_XY,\nu_M\rangle)(\nu_Q^S\nu_S + \nu_Q^M\nu_M)\\
    &= (\nu_Q^S)^2 \langle \nabla_XY,\nu_S\rangle\nu_S + \nu_Q^S\nu_Q^M(
    \langle \nabla_XY,\nu_S\rangle\nu_M +
    \langle \nabla_XY,\nu_M\rangle\nu_S)\\
    &\qquad+ (\nu_Q^M)^2 \langle \nabla_XY,\nu_M\rangle\nu_M.\\
    \intertext{Hence,}
    \langle II_Q(X,Y), \nu_S\rangle &= (\nu_Q^S)^2\langle \nabla_XY,\nu_S\rangle
    + \nu_Q^S\nu_Q^M \langle \nabla_XY,\nu_M\rangle\\
    &= \langle\nu_Q,\nu_S\rangle^2 \langle II_S(X,Y),\nu_S\rangle\\
    \\&\quad+ \langle \nu_Q,\nu_S\rangle \langle \nu_Q,\nu_M\rangle \langle
    K_M(X,Y),\nu_M\rangle,
  \end{align*}
  where the last line follows from our prior computations. Now since
  $M$ and $Q$ are transverse, $\langle \nu_Q,\nu_S\rangle \neq 0$.
  Moreover, $\langle II_Q(X,Y),\nu_S\rangle = K_Q(X,Y) \langle \nu_Q,\nu_S
  \rangle$. Thus, we find
  \[K_Q(X,Y) = \langle \nu_Q,\nu_S\rangle K_S(X,Y) + \langle \nu_Q,\nu_M\rangle
  K_M(X,Y),\]
  which yields the claim.\
\end{proof}

We now continue in the CAH Einstein context, where we get an immediate corollary from the preceding result.

\begin{corollary}
  \label{sumbcor}
  Let $(M,[h])$ be a compact manifold with boundary $S$, equipped with a conformal class $[h]$.
  Suppose $(X,M,Q,g)$ is a cornered AH Einstein space, such that
  $K_{Q} = \lambda g|_{TQ}$, and with smooth conformal infinity $[h]$.
  Further, suppose that, for smooth defining functions, the compactified metric $\bar{g}$ is smooth.
  Then $S$ is umbilic in $M$ with respect to any metric $h \in [h]$.
\end{corollary}
\begin{proof}
  It certainly suffices to consider only $h$, as umbilicity is a conformally invariant condition.
  For the same reason, it follows that $Q$ is $\bar{g}$-umbilic. Let $r$ be a geodesic defining function for $M$,
  as in \cite{gl91}, and $\bar{g} = r^2g$. Because $g$ is a CAH Einstein
  metric, $M$ is $\bar{g}$-totally geodesic (see e.g. \cite{g99}). Thus, since $h = \bar{g}|_{TM}$ (where $r$ is an appropriate
  defining function), the claim follows directly from (\ref{subinteq}).\
\end{proof}

This corollary gives a substantial obstruction to the existence of a smooth CMC-umbilically cornered CAH Einstein
space realizing $(M,[h])$ as its conformal infinity. For $n > 2$, for example, it provides a proof relying only on
the boundary geometry
that, if $\widetilde{X}$ is hyperbolic space $\mathbb{H}^{n + 1}$ and $M$ is a subset of $\mathbb{R}^n$,
then the boundary $S$ of $M$ must be a sphere or a hyperplane, as these are the only umbilic surfaces in Euclidean
space. (This proof does not work if $n = 2$, since every hypersurface of a 2-space is umbilic.)
We may even further characterize the geometry near the boundary as we further expand the
umbilic condition.

First we define some helpful coordinates.
Fix a metric $h \in [h]$. On a neighborhood $V$ near a point of $S$ in $M$, we may always choose geodesic normal
coordinates $x^1,\dots,x^n$ such that $S = \left\{ x^n = 0 \right\}$ and
such that $\frac{\partial}{\partial x^n} \perp_h TS$, with
$\left|\frac{\partial}{\partial x^n}\right| = 1$ and
$\frac{\partial}{\partial x^n} (\langle \frac{\partial}{\partial x^n},\frac{\partial}{\partial x^n}\rangle)
= 0$ at $S$. Let $r$ be a geodesic defining function
for $M$, so that $h = r^2g|_{TM}$, such that $|dr|^2_{\bar{g}} = 1$, 
and such that $x^0 = r,x^1,\dots,x^n$ are
coordinates for some $X \supseteq U \simeq [0,\varepsilon) \times V$, where
$\frac{\partial}{\partial r} \perp_{\overline{g}} TM$.

When working with these coordinates, we will use the Roman indices
$0 \leq i, j, k \leq n$ to label coordinates on $X$; the Greek indices
$0 \leq \alpha, \beta, \gamma \leq n - 1$ to label coordinates on $Q$; and the Roman indices
$1 \leq s,t,u \leq n - 1$ to label coordinates on $S$.

Before continuing to explore the consequences of smoothness, we state a useful result that is
true more generally.

\begin{proposition}
  \label{smoothprop}
  Let $(X,M,Q)$ be a cornered space, and let $g$ be a smooth CAH Einstein metric on $X$ satisfying
  $K_{Q} = \lambda g|_{TQ}$, where $K_Q$ is the second fundamental form of $Q$ with respect
  to the inward-pointing normal vector.
  Let $\bar{\nu}_M$ be
  the $X$-inward $\bar{g}$-unit normal to $M$, and let $\bar{\nu}_S$ be the 
  $M$-inward $h$-unit normal to $S$ in $M$. 
  Then at every point $p \in S$, $\cos(\theta_0(p)) = -\lambda$. 
  In particular, $\lambda \in (-1,1)$.
\end{proposition}

Notice that the proof given here depends only on the continuity of $K_Q$ up to $S$; thus, by
Lemma 4.3 of \cite{m16}, this proposition remains true if $g$ is only admissible.

\begin{proof}
Let $\nu$ be the inward-pointing unit normal field on $Q$, and $K$ the second fundamental form of $Q$,
both with respect to $g$; and let $\bar{\nu} = \bar{\nu}_Q$ be the inward-pointing unit normal field
on $Q$ with respect to $\bar{g}$. Now the umbilic condition
is equivalent by Weingarten to
\[\langle \nabla_X \nu,Y \rangle_g = -\lambda \langle X,Y\rangle_g\]
for all $X, Y \in C^{\infty}(TQ)$. For $r \neq 0$, the unit normal to $Q$
with respect to $\overline{g}$ is given by $\bnu = r^{-1}\nu$. We wish
to compute $\overline{K}(X,Y)$, the second fundamental form of $Q$ with respect to $\bar{g}$.

A straightforward computation shows that for any vector fields $X, Y$,
we have 
\begin{equation*}
  \overline{\nabla}_XY = \nabla_XY + r^{-1}\left[ dr(X)Y + dr(Y)X
- \langle X,Y\rangle_{\overline{g}}\grad_{\overline{g}} r \right].
\end{equation*}
For $q \in Q$ and $X, Y \in TQ|_q$, it follows (taking extensions where necessary) that
\begin{align*}
\overline{K}(X,Y) &= -\langle \overline{\nabla}_X (r^{-1}\nu),Y\rangle_{\overline{g}}\\
&= -r^{-1}\left\langle \nabla_X\nu + dr(X) \overline{\nu} + dr(\overline{\nu})X
- \langle X,\overline{\nu}\rangle \grad_{\bar{g}} r - dr(X)\bnu,Y\right\rangle_{\overline{g}}\\
&= r^{-1}\left( r^2K(X,Y) - dr(\overline{\nu})\langle X,Y\rangle_{\overline{g}}
\right)\numberthis\label{secformrel}.\\
\intertext{Therefore,}
\overline{K} &= \frac{\lambda - dr(\overline{\nu})}{r}\overline{g}(X,Y)\numberthis\label{baseq}
\end{align*}
is equivalent to $K_Q = \lambda g$.

Thus, we see that for $r \neq 0$, 
$Q$ is $\bar{g}$-umbilic with possibly non-constant umbilic coefficient. But we also see
that, for $\overline{K}$ to remain smooth up to the boundary -- which it surely
must, since $\bar{g}$ is a smooth metric -- we must have
\begin{equation*}
  dr(\overline{\nu}) \xrightarrow[r \to 0]{} \lambda.
\end{equation*}
In particular, since $\partial_r$ is the inward-pointing unit $\bar{g}$-normal at $M$, which we
denote by $\overline{\nu}_M$, we find that (denoting $\overline{\nu} = \overline{\nu}_Q$ for clarity)
\begin{equation}
\label{lambdeq}
\left\langle \overline{\nu}_Q,\overline{\nu}_M\right\rangle_{\overline{g}}
= \lambda,\qquad (r = 0),
\end{equation}
is equivalent to $K = \lambda g + O(r^{-1}).$
Since $\cos(\theta) = -\langle \bar{\nu}_Q,\bar{\nu}_M\rangle$, our claim is
established.
\end{proof}

As mentioned, the above result actually holds even for admissible metrics. We now obtain a result that in general
does not. We will henceforth in this section assume that inner products
are with respect to $\overline{g}$ if not otherwise specified. 

\begin{proposition}
  \label{smoothprop2}
  Let $X, M, Q, g, \lambda$, and $K_Q$ be as in Proposition \ref{smoothprop}.
  Let $\bar{\nu}_M$ be the $X$-inward $\bar{g}$-unit normal to $M$, and let $\bar{\nu}_S$ be the 
  $M$-inward $h$-unit normal to $S$ in $M$. 
  Moreover, let $\overline{R}$ be the curvature tensor of $\bar{g}$.
  Write $\bar{K}_S = \eta h$, where $\bar{K}_S$ is the second fundamental form of $S$ in $M$ with respect
  to $h$ and $\bar{\nu}_S$, and $\eta \in C^{\infty}(S)$; we can write this by Corollary \ref{sumbcor}. 
  Then for any $Z \in TS$,
  \begin{equation}
    \label{umbcoefeq}
    Z(\eta) = -\overline{R}(\bar{\nu}_M,Z,\bar{\nu}_M,\bar{\nu}_S).
  \end{equation}
\end{proposition}

\begin{proof}
Let $\nu, K$ and $\bnu$ be as in the previous proof.
We will assume that (\ref{lambdeq}) holds, and we multiply through by $r$
in (\ref{baseq}) to obtain
\begin{equation}
r\overline{K}(X,Y) = (\lambda - \langle \overline{\nu}_Q,\partial_r\rangle)
\overline{g}(X,Y)\label{multeq}.
\end{equation}
This equation holds on $Q$ for any $X,Y \in C^{\infty}(TQ)$ if and only if
$K = \lambda g$. We will
repeatedly differentiate it covariantly to obtain new equations.

Before proceeding, we write $Q$ locally as the graph of a function,
\[x^n = \varphi(r,x^1,\dots,x^{n - 1}).\]
Notice that $\varphi(0,x^1,\dots,x^{n - 1}) \equiv 0$ by our choice of
coordinates on $M$.
We next define a local frame on $Q$ by
\[E_{\alpha} = \frac{\partial}{\partial x^{\alpha}} + 
\frac{\partial \varphi}{\partial x^{\alpha}}\frac{\partial}{\partial x^n}.\]
In particular, $\left\{ E_s \right\}_{s = 1}^{n - 1}$ is also a local frame
for $S$ at $r = 0$, in fact the coordinate frame.

Define $f \in C^{\infty}([0,\varepsilon) \times V)$ by $f = \varphi - x^n$. Then $f$ vanishes precisely on
$Q$, and we may write $\overline{\nu} = \frac{-\grad_{\overline{g}} f}{|\grad_{\overline{g}}f|}$.
Using this and the fact that, at $r = 0$, the normal $\overline{\nu}$ may be written as
\begin{equation}
  \label{bnueq}
  \bar{\nu} = \lambda \partial_r + \sqrt{1 - \lambda^2}\partial_{x^n}
\end{equation}
by (\ref{lambdeq}), it is straightforward to show that
\begin{equation}
  \left.\frac{\partial \varphi}{\partial r}\right|_{r = 0} = \frac{-\lambda}{\sqrt{1 - \lambda^2}}\label{firstder}.
\end{equation}

We intend to apply $\overline{\nabla}^Q_{E_0}$, the Levi-Civita connection of
$\bar{g}|_{TQ}$, to both sides of (\ref{multeq}). Doing this once,
and utilizing the metric property of the Levi-Civita connection, we obtain
\begin{equation}
\overline{K}_{\alpha\beta} + r\overline{\nabla}^Q_{E_0}\overline{K}_{\alpha\beta} = -E_0\left( \langle \partial_r,
\overline{\nu}\rangle \right)\overline{g}_{\alpha\beta}\label{diffonce},
\end{equation}
which should hold for all $r$, along $Q$. Taking $r = 0$, we get
\begin{align*}
  \overline{K}_{\alpha\beta} &= -E_0(\langle \partial_r,\overline{\nu}\rangle)
  \overline{g}_{\alpha\beta}\numberthis\label{firstcalcia}\\
\intertext{Now}
E_0(\langle \partial_r,\overline{\nu}\rangle) &= \langle \overline{\nabla}_{E_0}\partial_r,\overline{\nu}\rangle + \langle \partial_r,\overline{\nabla}_{E_0}\overline{\nu}\rangle\numberthis\label{firstcalci}
\end{align*}
(where $\overline{\nabla}$ is still the Levi-Civita connection for $\bar{g}$ on $TX$).
Since $E_0 = \partial_r + \frac{\partial \varphi}{\partial r}\partial_{x^n}$, we have
$\partial_r = E_0 - \frac{\partial \varphi}{\partial r}\partial_{x^n}$. Hence,
\begin{align}
  \langle \partial_r,\overline{\nabla}_{E_0}\overline{\nu}\rangle &= \langle E_0,\overline{\nabla}_{E_0}\overline{\nu}\rangle
  - \frac{\partial \varphi}{\partial r}\langle \partial_{x^n},\overline{\nabla}_{E_0}\overline{\nu}\rangle\nonumber\\
  &= -\bK(E_0,E_0) - \frac{\partial \varphi}{\partial r}\langle \partial_{x^n},\overline{\nabla}_{E_0}
  \bnu\rangle\label{firstcalcii}.
\end{align}
Moreover,
\begin{align}
  \langle \overline{\nabla}_{E_0}\partial_r,\bnu\rangle &= \langle\overline{\nabla}_{E_0}E_0,\bnu\rangle
  - \langle \overline{\nabla}_{E_0}\left( \frac{\partial \varphi}{\partial r}\partial_{x^n} \right),\bnu\rangle\nonumber\\
  &= \overline{K}(E_0,E_0) - E_0\left( \frac{\partial \varphi}{\partial r} \right)\langle \partial_{x^n},
  \bnu\rangle - \frac{\partial \varphi}{\partial r} \langle \overline{\nabla}_{E_0}\partial_{x^n},\bnu\rangle\nonumber\\
  &= \overline{K}(E_0,E_0) - \frac{\partial^2\varphi}{\partial r^2} \langle \partial_{x^n},\bnu\rangle
  - \frac{\partial \varphi}{\partial r} \langle \overline{\nabla}_{E_0} \partial_{x^n},\bnu\rangle\label{firstcalciii}.
\end{align}
From (\ref{firstcalci}),(\ref{firstcalcii}), and (\ref{firstcalciii}), we get
\begin{equation}
   E_0\langle\partial_r,\bnu\rangle = -\left( \frac{\partial^2\varphi}{\partial r^2}\langle \partial_{x^n},\bnu\rangle
  + \frac{\partial \varphi}{\partial r}E_0\langle \partial_{x^n},\bnu\rangle\right)\label{firstcalciv}.
\end{equation}
Now $1 \equiv |\bnu|_{\overline{g}}^2 = \overline{g}^{ij}\langle\bnu,\partial_{x^i}\rangle\langle\bnu,
\partial_{x^j}\rangle$.
We apply $E_0$ to this equation, take $r = 0$,
and use the facts that $TS \subset TQ$, so that $\bnu \perp_{\overline{g}}
TS$; that $\partial_{x^n}(|\partial_{x^n}|_{\overline{g}}) = 0$ on $S$; and that, because $g$ is Einstein,
$\partial_r\overline{g}|_{r = 0} = 0$ to conclude that
\begin{equation}
  \langle \partial_{x^n},\bnu\rangle E_0\langle \partial_{x^n},\bnu\rangle + \langle \partial_r, \bnu\rangle
  E_0\langle \partial_r,\bnu\rangle = 0\label{firstcalcv}.
\end{equation}
Combining this with (\ref{lambdeq}), (\ref{firstder}), (\ref{firstcalcia}), and (\ref{firstcalciv}), we finally
obtain
\begin{equation}
  \label{secondder}
  \overline{K}_{\alpha\beta} = (1 - \lambda^2)^{3/2}\left(\frac{\partial^2\varphi}{\partial r^2}\right)
  \overline{g}_{\alpha\beta} \qquad (r = 0).
\end{equation}
By applying Lemma \ref{subint}, the fact that $M$ is totally geodesic, and (\ref{bnueq}), we may conclude that
\begin{equation}
  \label{KSeq}
  \overline{K}_S = (1 - \lambda^2)\left(\frac{\partial^2\varphi}{\partial r^2}\right)\overline{g}|_{TS}.
\end{equation}
We have thus expressed the second-order term of the development of $Q$ in terms of the geometry of $S$ in $M$.

We can use the foregoing computations to rewrite (\ref{diffonce}) as
\begin{equation*}
  \overline{K}_{\alpha\beta} + r\overline{\nabla}^Q_{E_0}\overline{K}_{\alpha\beta} =
  \left( \left( \frac{\partial^2\varphi}{\partial r^2} \right)\langle \partial_{x^n},\bnu\rangle
  + \left( \frac{\partial \varphi}{\partial r} \right)E_0\langle \partial_{x^n},\bnu\rangle\right)
  \overline{g}_{\alpha\beta}.
\end{equation*}
To this, we now apply $\overline{\nabla}^Q_{E_0}$ once again. We obtain
\begin{align}
  \label{difftwice}
  2\overline{\nabla}^Q_{E_0}\overline{K}_{\alpha\beta} + r(\overline{\nabla}^Q_{E_0})^2\overline{K}_{\alpha\beta}
  &= \left[ \left( \frac{\partial^3\varphi}{\partial r^3}\right)\langle \partial_{x^n},\bnu\rangle 
  + 2 \left( \frac{\partial^2\varphi}{\partial r^2}\right)E_0\langle\partial_{x^n},\bnu\rangle\right.\\
  &\qquad+ \left.\left( \frac{\partial \varphi}{\partial r} \right)E_0^2\langle \partial_{x^n},\bnu\rangle\right]
  \overline{g}_{\alpha\beta}\nonumber.
\end{align}
We take $r = 0$ and utilize equation (\ref{firstcalcv}) and find in particular that, at $r = 0$,
\begin{align}
  2\overline{\nabla}^Q_{E_0}\overline{K}_{\alpha\beta} &= -E_0^2(\langle \partial_r,\bnu\rangle)\bar{g}_{\alpha\beta}
  \label{secondcalci}.
\end{align}

We next wish to apply the Codazzi-Mainardi equation (see, e.g., \cite{spi99},
Theorem III.1.11). Codazzi states that
\begin{equation*}
  \langle \overline{R}(E_{\gamma},E_{\alpha})E_{\beta},\bnu\rangle = \left(\overline{\nabla}^Q_{E_{\gamma}}
  \overline{K}\right)
  (E_{\alpha},E_{\beta}) - \left( \overline{\nabla}^Q_{E_{\alpha}}\overline{K} \right)\left( E_{\gamma},
  E_{\beta}\right).
\end{equation*}
Hence, taking $\gamma = 0$,
\begin{equation}
  \label{codazzi}
  \overline{\nabla}^Q_{E_0}\overline{K}_{\alpha\beta} = \overline{\nabla}^Q_{E_{\alpha}}K_{0\beta}
  + \langle\overline{R}(E_0,E_{\alpha})E_{\beta},\bnu\rangle.
\end{equation}

We next take $\alpha = s$ (i.e., $1 \leq \alpha \leq n - 1$) and $\beta = 0$. Hence, we have
\begin{equation}
  \bnabla^Q_{E_0}\overline{K}_{s0} = \bnabla^Q_{E_s}\bK_{00} + \langle R(E_0,E_s)E_0,\bnu\rangle.\label{secondcalcii}
\end{equation}
But along $S$, $\overline{K}$ is known by (\ref{secondder}), and we find that
\begin{equation}
  \bnabla^Q_{E_s}\overline{K}_{00} = (1 - \lambda^2)^{3/2}\left( \frac{\partial^3\varphi}{\partial r^2 \partial x^s} \right)
  \overline{g}_{00}.\label{secondcalciii}
\end{equation}
At $r = 0$, $\overline{g}_{00} = \frac{1}{1 - \lambda^2}$ (recall that this is expressed in the $\left\{ E_{\alpha} \right\}$
frame, not the coordinate frame). This allows us to conclude, via (\ref{secondcalcii}) and (\ref{secondcalciii}), that
\begin{equation*}
  \bnabla^Q_{E_0}\bK_{s0} = \sqrt{1 - \lambda^2}\left( \frac{\partial^3\varphi}{\partial r^2\partial x^s} \right)
  + \langle \overline{R}(E_0,E_s)E_0,\bnu\rangle.
\end{equation*}
We can now substitute $\bnu$ and the definitions of $E_{\alpha}$ in terms of the coordinate frame to compute
this curvature term in the coordinate basis. (Every appearance of $\overline{R}$ in index notation will
be with reference to local coordinates, not the frame on $Q$). We also utilize that, because
$g$ is Einstein, $M$ is totally geodesic in $X$ with respect to $\overline{g}$, and so
again by Codazzi, $\langle \overline{R}(\partial_{x^n},\partial_{x^s})\partial_r,\partial_{x^n}\rangle = 0$.
Carrying this straightforward computation out
at $r = 0$, we find that
\begin{equation*}
  \langle \overline{R}(E_0,E_s)E_0,\bnu\rangle = \frac{1}{\sqrt{1 - \lambda^2}}\overline{R}_{0s0n}.
\end{equation*}
Applying these computations to (\ref{secondcalci}) and noting that $\overline{g}_{s0} = 0$ at $r = 0$,
we conclude that
\begin{equation}
  \label{thirddera}
  \frac{\partial^3\varphi}{\partial r^2\partial x^s} = -\frac{1}{1 - \lambda^2}\overline{R}_{0s0n}.
\end{equation}
In conjunction with (\ref{KSeq}), this yields the claim.\
\end{proof}

\begin{remark}
Because Euclidean space is flat, this is precisely what we needed to ensure that, in the $n = 2$ hyperbolic case,
only circles and lines can occur at the corner boundary: if $g$ is the hyperbolic metric on $\mathbb{H}^3$,
then one choice for the compactified metric $\bar{g}$ is the Euclidean metric itself. In this case,
$\overline{R} = 0$, so by (\ref{umbcoefeq}), $S$ has a constant umbilic coefficient in $(M,\bar{g}|_{TM})$.
But the only umbilic hypersurfaces in Euclidean 2-space with constant umbilic coefficients are circles and
straight lines. Notice that in terms purely of the
local boundary geometry, this restriction
occurs at higher order for $n = 2$ than for $n \geq 3$, where it is implied by Corollary \ref{sumbcor}.
\end{remark}

The arguments of the preceding proposition can be extended to yield further conditions on $\bar{g}$ and $Q$, and perhaps
$S$.
For example, let us return
to (\ref{codazzi}), this time with $\alpha = s$ and $\beta = t$ ($1 \leq s,t \leq n -1$). The term
$\overline{\nabla}^Q_{E_s}\overline{K}_{0t}$ can be evaluated using our earlier calculations. We get
\begin{equation*}
  \bnabla^Q_{E_s}\bK_{0t} = \bnabla_{E_s}\left[ (1 - \lambda^2)^{3/2}\left( \frac{\partial^2\varphi}
  {\partial r^2} \right)\right]\overline{g}_{0t} = 0.
\end{equation*}
Hence, at $r = 0$,
\begin{equation*}
  \bnabla^Q_{E_0}\bK_{st} = \overline{R}(E_0,E_s,E_t,\bnu).
\end{equation*}
Expanding the right-hand side in the coordinate frame again yields
\begin{equation}
  \label{Rbareq}
  \bnabla^Q_{E_0}\bK_{st} = \lambda(\overline{R}_{0st0} - \overline{R}_{nstn}).
\end{equation}
By (\ref{secondcalci}), $\tf_{\bar{g}}\bnabla^Q_{E_0}\bK_{st} = 0$; thus we conclude that along $S$,
\begin{equation*}
  \lambda\tf_{\bar{g}|_{TS}}\left(\overline{R}_{0st0} - \overline{R}_{nstn} \right) = 0.
\end{equation*}
This represents an additional $\varphi$-independent condition on $\bar{g}$.

We can generalize this process to see the structure of the conditions that would arise as we differentiated
more. Covariantly differentiating (\ref{multeq}) $k - 1$ times by $E_0$ yields
\begin{equation*}
  k(\overline{\nabla}^Q_{E_0})^{k - 1}\overline{K}_{\alpha\beta} = E_0^{k}(\langle\partial_r,\bnu_Q\rangle)
  \bar{g}_{\alpha\beta}.
\end{equation*}
At each step, taking $\alpha = \beta = 0$ or $\alpha = s, \beta = t$ yields new constraints on
$\frac{\partial^{k}\varphi}{\partial r^{k}}|_{r = 0}$ and also on $\bar{g}$, its curvature tensor,
and its derivatives. Taking $\alpha = 0, \beta = s$ yields additional constraints on lower-order (that is,
already-developed) $r$-derivatives of $\varphi$, possibly also putting further constraints on $\bar{g}$,
as in (\ref{Rbareq}). In this way, a vastly overdetermined system of equations for $\bar{g}$ and $\varphi$
would be produced. Actually carrying these computations out to higher order becomes quickly unwieldly; in any case,
we have already developed extremely restrictive constraints on a smooth solution. In section \ref{einsec}, we will therefore
study formal existence of Einstein metrics on the blowup that are in normal form in the sense of Theorem \ref{normformthm}.
Thus, in particular, the constructed metrics will be smooth on the blowup but not on the base.
As we will see, this setting offers precisely the right relaxation of smoothness to allow relatively unique general solutions.

%% file: tex/lapl.tex
\section{The Laplacian and ODE Analysis}\label{laplsec}

In this section, we study the eigenvalue problem for the scalar Laplacian on cornered asymptotically
hyperbolic spaces and prove Theorem \ref{lapthm}. 
We study the formal asymptotics of eigenfunctions on a CAH space
with given boundary conditions.
Along the way, we prove results about some ODEs that will also be important in the next section.

Throughout, let $(X,M,Q)$ be a cornered space and
$(\tX,\tM,\tQ,\tS)$ its blowup, with $g$ an
admissible metric on $\tX$. We will let $\theta, x, \rho$ be coordinates
in which $g$ takes the form (\ref{g}), with $k_{\rho}$ as in that equation.
We also introduce the following notation, motivated by \cite{gl91}.
We will write $E^k$ to denote
any polynomial of degree less than or equal to $k$ in
$\sin(\theta)\frac{\partial}{\partial \theta}$
and $\rho\sin(\theta)\frac{\partial}{\partial x^{\mu}}$ ($1 \leq \mu \leq n$), with
coefficients in $C^{\infty}(\tX)$. Finally, we fix $s > \frac{n}{2}$.

We first compute the Laplace operator of $g$.

\begin{lemma}
  \label{laplem}
  Let $(\tX^{n + 1},\tM,\tQ,\tS)$ be the blowup of a cornered space, and $g$
  an admissible metric on $\tX$ expressed in the form (\ref{g}). Then for 
  $u \in C^{\infty}(\tX \setminus (\tM \cup \tS))$,
  \begin{multline*}
    \Delta_gu = \sin^2(\theta)\partial_{\theta}^2u + 
  (1 - n) \sin(\theta)\cos(\theta)\partial_{\theta}u 
  + \rho^2\sin^2(\theta)\partial_{\rho}^2u +\\
  \qquad(2 - n)\rho\sin^2(\theta)\partial_{\rho}u +
  \rho^2\sin^2(\theta)\Delta_{k_{\rho}}u + \rho\sin(\theta)E^2(u).
  \end{multline*}
\end{lemma}

\begin{proof}
  Using the formula $\Delta_gu = g^{-1/2}\partial_i(g^{ij}g^{1/2}
  \partial_ju)$, this follows easily from (\ref{g}) and (\ref{ginv}).
\end{proof}

We wish to carry out an asymptotic analysis of solutions
to a boundary-value problem for the equation $\Delta_g + s(n - s)u = 0$.
To carry out our analysis, we will expand the solution order-by-order
in $\rho$. Consequently, we work in terms of the 
\emph{indicial operator} of $P_s = \Delta_g + s(n - s)$. For $\nu \in \mathbb{R}$,
this is the operator $I_{s,\nu}:C^{\infty}(\tS) \to C^{\infty}(\tS)$ defined by
extending $u \in C^{\infty}(\tS)$ to $\tu \in C^{\infty}(\tX)$, and then setting
$I_{s,\nu}(u) = \rho^{-\nu}P_s(\rho^{\nu}\tu)|_{\rho = 0}$. The following
is immediate from Lemma \ref{laplem}.

\begin{lemma}
  \label{indexlem}
  Let $\nu \in \mathbb{R}$. Then
  \begin{equation}
    \label{iform}
    I_{s,\nu}(u) = \sin^2(\theta)\partial_{\theta}^2u
    + (1 - n)\sin(\theta)\cos(\theta)\partial_{\theta}u
    + \nu(\nu + 1 - n)\sin^2(\theta)u + s(n - s)u.
  \end{equation}
\end{lemma}

We now specialize to the constant-angle case, supposing that $M$ and $Q$ make
constant angle $\theta_0$ with respect to $g_X$. We will assume the metric
is in the form given by Corollary \ref{polarcor};
however, we will suppress the diffeomorphism $\chi$ and will simply
regard $(\theta,x,\rho)$ as a parametrization of $\tX = [0,\theta_0] \times S \times [0,\varepsilon)$ near $\tS$. Thus,
we will take $g$ to be given by
\begin{equation}
  \label{formcopy2}
  g = \frac{1}{\sin^2(\theta)}\left[ d\theta^2 + h_{\theta} \right],
\end{equation}
where $h_{\theta}$ is a smooth family of smooth AH metrics on the hypersurface $\tM = \left\{ 0 \right\} \times S \times [0,\varepsilon)$
and where $h_0 = \frac{d\rho^2 + k_{\rho}}{\rho^2}$.
We impose the inhomogeneous boundary condition $\sin^{s - n}(\theta)u|_{\tM} = \psi$ at $\tM$ (where $\psi \in C^{\infty}(\tM)$) and
the homogeneous Robin condition $\partial_{\nu}u + (s - n)\cot(\theta)u(\theta)|_{\tQ} = 0$ at $\tQ$. Notice that for the metric
(\ref{formcopy2}), this last is equivalent to $\partial_{\theta}u(\theta_0) + (s - n)\cot(\theta_0)u(\theta_0) = 0$. Notice also
that for $s = n$ or for $\theta_0 = \frac{\pi}{2}$, this reduces to a homogeneous Neumann condition.
These boundary conditions are motivated partially by their relevance, in the case $s = n$, to the Einstein problem considered in the next section.
The form for $s \neq n$ is motivated by the naturalness of the analysis it leads to; however, the below analysis could easily be modified to study
a variety of other boundary conditions.

Notice that, with the constant-angle hypothesis, the indicial operator (\ref{iform}) restricts to each
fiber as an operator independent of
the fiber, if each fiber is identified with the interval $[0,\theta_0]$. Since we will construct
solutions order-by-order in $\rho$ using the indicial operator, our problem turns on an analysis
of the equation $I_{s,\nu}u = f$, with homogeneous Robin boundary condition at $\theta = \theta_0$
and Dirichlet boundary condition at $\theta = 0$. We turn to an analysis of this equation.

\subsection{Function Spaces}\label{funcsubsec}
It will be useful, before undertaking an analysis of the indicial operator, to define some function spaces.

First we define several spaces of functions on the interval $[0,\theta_0]$, where $\theta_0 \in (0,\pi)$ is fixed. Let $n \geq 2$,
$s > \frac{n}{2}$, and $k \geq 1$. If $2s \notin \mathbb{Z}$, define,

\begin{equation}
  \label{Asnk}
  \mathcal{A}_{n,s,k}(\theta_0) = \theta^{n - s}C^{\infty}([0,\theta_0]).
\end{equation}

Otherwise, if $2s \in \mathbb{Z}$, define
\begin{equation}
  \label{Ank}
  \mathcal{A}_{n,s,k}(\theta_0) = \theta^{n - s}C^{\infty}([0,\theta_0]) \oplus \theta^{n - s}\bigoplus_{i = 1}^k (\theta^{2s - n}\log\theta)^iC^{\infty}([0,\theta_0]).
\end{equation}
In either case, define
\begin{equation}
  \label{Ank0}
  \begin{split}
  \mathcal{A}_{n,s,k}^0(\theta_0) =& \left\{ u \in \mathcal{A}_{n,s,k}: \sin^{s - n}(\theta)u(\theta)|_{\theta = 0} = 0\right.\\
  & = \left.u'(\theta_0) + (s - n)\cot(\theta_0)u(\theta_0)\right\}.
  \end{split}
\end{equation}
Next, if $2s \notin \mathbb{Z}$, define
\begin{equation*}
  \label{Bn1noz}
  \mathcal{B}_{n,s,k}(\theta_0) = \theta^{n - s + 1}C^{\infty}([0,\theta_0]).
\end{equation*}
On the other hand, if $2s \in \mathbb{Z}$, define
\begin{equation}
  \label{Bn1}
  \mathcal{B}_{n,s,1}(\theta_0) = \theta^{n - s + 1} C^{\infty}([0,\theta_0]) \oplus \theta^{s + 1}\log(\theta)C^{\infty}([0,\theta_0]),
\end{equation}
and for $k \geq 2$, define
\begin{equation}
  \label{Bnk}
  \mathcal{B}_{n,s,k}(\theta_0) = \mathcal{B}_{n,s,1}(\theta_0) \oplus \theta^{s - n}\bigoplus_{i = 2}^k (\theta^{2s - n}\log\theta)^iC^{\infty}([0,\theta_0]).
\end{equation}
In general, $\theta_0$ will be fixed and clear from the context, and we will refer to these spaces simply as $\mathcal{A}_{n,s,k}$ and
$\mathcal{B}_{n,s,k}$.

Now let $Y$ be any manifold with or without boundary, and $\mathcal{V}$ a vector bundle over $Y$. We define spaces of one-parameter families of smooth sections $u_{\theta}:Y \to \mathcal{V}$
($0 \leq \theta \leq \theta_0$) as follows. Let $\pi:[0,\theta_0] \times M \to M$ be the projection on the second factor, and $\pi^*\mathcal{V}$ be the pullback bundle of $\mathcal{V}$ to the
product. Thus we let $C^{\infty}([0,\theta_0] \times M,\pi^*\mathcal{V})$ be the space of smooth sections of the pullback bundle.
Then, if $2s \notin \mathbb{Z}$, we set
\begin{equation*}
  \mathcal{A}_{n,s,k}(\theta_0,Y,\mathcal{V}) = \theta^{n - s}C^{\infty}([0,\theta_0] \times M, \pi^*\mathcal{V}),
\end{equation*}
and otherwise, we set
\begin{equation*}
  \begin{split}
  \mathcal{A}_{n,s,k}(\theta_0,Y,\mathcal{V}) = \theta^{n - s}C^{\infty}([0,&\theta_0] \times M,\pi^*\mathcal{V}) \oplus\\
  &\theta^{n - s}\bigoplus_{i = 1}^{k}(\theta^{2s - n}\log\theta)^iC^{\infty}
  ([0,\theta_0]\times M,\pi^*\mathcal{V}).
\end{split}
\end{equation*}
We similarly define $\mathcal{A}_{n,s,k}^0(\theta_0,Y,\mathcal{V})$ and $\mathcal{B}_{n,s,k}$ in the obvious fashion, generalizing their scalar counterparts in the same way as for $\mathcal{A}_{n,s,k}$.
Thus, in particular, $\mathcal{A}_{n,s,k}(\theta_0)$ is canonically isomorphic to $\mathcal{A}_{n,s,k}(\theta_0,\left\{ 0 \right\},\mathbb{R})$.
In each of these function spaces, if $\mathcal{V}$ is omitted, then it is taken
to be the trivial vector bundle $\mathbb{R} \times Y$.

We next define two families of function spaces on $[0,\theta_0]_{\theta} \times [0,\varepsilon)_{\rho}$. For $q \geq 0$ an integer and $2s \in \mathbb{Z}$, we let
$\mathcal{E}_{n,s,q}$\label{Enq} be the space of functions $\eta$ that can be written
\begin{equation*}
  \eta(\theta,\rho) = \rho^{s + q}\sum_{i = 0}^{\left\lfloor \frac{q}{2} + 1\right\rfloor}\log(\rho)^{i}b_i(\theta),
\end{equation*}
where each $b_i \in \mathcal{A}_{n,s,1}^0$. Next, we let $\mathcal{F}_{n,s,q}$ be the space of functions $\eta$
on $[0,\theta_0] \times [0,\varepsilon)$ that can be written
\begin{equation*}
  \eta(\theta,\rho) = \rho^{s + q}\sum_{i = 0}^{\left\lfloor\frac{q + 1}{2}\right\rfloor}\log(\rho)^ic_i(\theta),
\end{equation*}
where each $c_i \in \mathcal{B}_{n,s,1}$.

Now let $(\tX,\tM,\tQ,\tS)$ be the blowup of a cornered space $X^{n + 1}$. Let $(\theta,x,\rho)$, where $x \in S$, be a parametrization
of $\tX$ near $\tS$ for which $\theta$ and $\rho$ are defining functions for $\tM$ and $\tS$, respectively. An example would be the polar decompositions
considered in Section \ref{backsec}.
Then if $2s \notin \mathbb{N}$, we define
\begin{equation*}
  \mathcal{P}_{s}(\tX) = \mathcal{R}_{s}(\tX) = \rho^{n - s}\theta^{n - s}C^{\infty}(\tX).
\end{equation*}
Otherwise, we define
\begin{equation}
  \label{R}
  \mathcal{R}_s(\tX) = \rho^{n - s}\theta^{n - s}C^{\infty}(\tX) + \rho^{n - s}\theta^{s}\log(\theta)C^{\infty}(\tX).
\end{equation}
Notice that the definition is independent of the choice of $\theta$. Then, define $\mathcal{P}_s(\tX)$ to be the space of functions $u \in C^{\infty}(\mathring{\tX})$
that have an asymptotic expansion
\begin{equation}
  \label{Px}
  u(\theta,x,\rho) \sim a_0(\theta,x,\rho) + \rho^{2s - n}\sum_{j = 1}^{\infty}\rho^{2(j - 1)}\log(\rho)^ja_j(\theta,x,\rho),
\end{equation}
where each $a_j \in \mathcal{R}_s(\tX)$.

We next define three families of spaces which will be used only in Section \ref{einsec}, but which are sufficiently related to the above spaces that it will be convenient
to have them defined here.
First, if $n \geq 4$ is even, then we define
$\mathcal{H}_{n,k}(\theta_0,Y,\mathcal{V}) = \mathcal{A}_{n,n,k}(\theta_0,Y,\mathcal{V})$; and we let $\mathcal{H}_n(\theta_0,Y,\mathcal{V})$ be the space of maps $u$
in $C^{\infty}((0,\theta_0]\times M,\pi^*\mathcal{V}) \cap C^{n - 1}([0,\theta_0]\times M,\pi^*\mathcal{V})$ that have an infinite asymptotic expansion $u \sim \sum_{i = 0}^{\infty}u_i$ at $\theta = 0$, where
$u_i \in (\theta^n\log\theta)^iC^{\infty}([0,\theta_0]\times M,\pi^*\mathcal{V})$. If $n = 2$ or $n$ is odd, then for consistency of notation, 
we define $\mathcal{H}_{n,k}(\theta_0,Y,\mathcal{V})$ and $\mathcal{H}_n(\theta_0,Y,\mathcal{V})$
to be simply $C^{\infty}([0,\theta_0]\times M,\pi^*\mathcal{V})$.

Finally, if $M^n$ is a manifold with boundary, we define the space of metrics $\mathcal{M}(\theta_0,M) \subset \mathcal{H}_{n}(\theta_0,M,S^2({}^0T^*M))$
\label{qdef} to be the $h_{\theta}$ in $\mathcal{H}_{n}(\theta_0,M,S^2({}^0T^*M))$ such that, for $\theta$ fixed, $h_{\theta}$ is a smooth AH metric on $M$.

\subsection{ODE Analysis}\label{odesubsec}

It is an elementary exercise in the theory of regular singular operators to see that the indicial roots in $\theta$ of $I_{s,\nu}$ at $\theta = 0$ are $n - s$ and $s$.
We will refer to an \emph{indicial root} (in $\rho$) of $\Delta_g + s(n - s)$ as a value of $\nu$ for which
$I_{s,\nu}$ is not injective on the space of smooth functions $u$ on $[0,\theta_0]$ satisfying
$u(\theta) = o(\theta^{n - s})$ and
$u'(\theta_0) + (s - n)\cot(\theta_0)u(\theta_0) = 0$. It will be convenient to write $u(\theta) = \sin^{n - s}(\theta)v(\theta)$. It is then straightforward to compute that,
equivalently, an indicial root is a value of $\nu$
for which $\lambda_{\nu} := \nu(\nu + 1 - n)$ is an eigenvalue of the operator
\begin{equation*}
  \begin{split}
  L_s =& -\frac{d^2}{d\theta^2} - (n + 1 - 2s)\cot(\theta)\frac{d}{d\theta} + (s - 1)(s - n)\\
=&-\sin^{2s - 1 - n}(\theta)d_{\theta}(\sin^{n + 1 - 2s}(\theta)d_{\theta}) + (s - 1)(s - n),
  \end{split}
\end{equation*}
with the
boundary conditions $v(0) = 0 = v'(\theta_0)$. In the above
factoring of the eigenvalues $\lambda_{\nu}$, we call $\nu$ the \emph{spectral parameter}.
It will be useful also to record the relationship
\begin{equation}
  \label{ilrel}
  I_{s,\nu}(u(\theta)) = \sin^2(\theta)(\lambda_{\nu}u(\theta) - \sin^{n - s}(\theta)L_s(\sin^{s - n}(\theta)u(\theta))).
\end{equation}
As an operator on $[0,\theta_0]$, $L_s$ on the interval $(0,\theta_0)$ has a limit point singularity
at $\theta = 0$, and we impose the homogeneous Neumann condition at $\theta = 
\theta_0$. It will follow from results in this section that its spectrum is discrete; and it is then a standard result in Sturm-Liouville theory that the eigenfunctions,
which are smooth, form an orthonormal basis for $L^2([0,\theta_0])$ with the measure $\sin^{n + 1 - 2s}(\theta)d\theta$.
(The discreteness of the spectrum could also be deduced from fairly general theorems.)

We begin with the following elementary result.

\begin{lemma}
  \label{llem}
  Let $u,v \in C^2([0,\theta_0])$ be such that one is $O(\theta)$ and the other is $O(\theta^{2s - n})$,
  and such that $u'(\theta_0) = 0 = v'(\theta_0)$. Then
  \begin{align}\label{llema}\int_0^{\theta_0}(L_su)v \sin^{n + 1 - 2s}(\theta)d\theta &= \int_{0}^{\theta_0}u'(\theta)v'(\theta)
      \sin^{n + 1 - 2s}(\theta)d\theta\\
  &\quad + (s - 1)(s - n)\int_0^{\theta_0}u(\theta)v(\theta)\sin^{n + 1 - 2s}(\theta)d\theta;\nonumber\\
  \intertext{and}
    \label{llemb} \int_0^{\theta_0} (L_su)v \sin^{n + 1 - 2s}(\theta)d\theta &= \int_{0}^{\theta_0}u(L_sv)\sin^{n + 1 - 2s}(\theta)d\theta.
  \end{align}
\end{lemma}
\begin{proof}
  Equation \ref{llema} follows using integration by parts after writing $L_su(\theta)$ in divergence form as
  $L_su(\theta) = -\sin^{2s - 1 - n}(\theta)\frac{d}{d\theta}\left[ \sin^{n + 1 - 2s}(\theta)u'(\theta) \right] + (s - 1)(s - n)$. Equation
  \ref{llemb} follows from part \ref{llema} by symmetry.
\end{proof}

We next study the eigenvalues of $L_s$. First, we state a singular Sturm
comparison theorem from \cite{nai12}.

\begin{proposition}[Theorem 3 from \cite{nai12}]
  \label{sturmthm}
  Suppose that $u \in C^2((a,b))$ satisfies the equation
  \begin{equation*}
    (p(t)u')' + q(t)u = 0,
  \end{equation*}
  on the interval $(a,b)$, where $p \in C^1((a,b))$ and $q \in C^0((a,b))$, and $p(t) \geq 0$. Suppose further 
  that $\int_a\frac{1}{p(t)u(t)^2}dt = \infty$, and that
  $\int^b\frac{1}{p(t)u(t)^2}dt = \infty$. Suppose further that
  $u(t)$ has exactly $n - 1$ zeros on the interval $(a,b)$, where $n \in \mathbb{N}$. 
  
  Now let $P \in C^1((a,b))$ and $Q \in C^0((a,b))$ be such that $p(t) \geq P(t) \geq 0$ and
  $Q(t) \geq q(t)$ on $(a,b)$, and that $Q(t) \not\equiv q(t)$. Suppose that $v \in C^2((a,b))$ satisfies the equation
  \begin{equation*}
    (P(t)v')' + Q(t)v = 0.
  \end{equation*}
  Then $v(t)$ has at least $n$ zeroes in $(a,b)$.
\end{proposition}

\begin{proposition}
  \label{specprop}
  Let $n \geq 2$ and $\theta_0 \in (0,\pi)$. Then
  the smallest eigenvalue of $L_s$ for the boundary conditions $u(0) = 0$ and $u'(\theta_0) = 0$
  \begin{itemize}
    \item lies in $((s - 1)(s - n),s(s + 1 - n))$ if $\frac{\pi}{2} < \theta_0 < \pi$;
    \item is $s(s + 1 - n)$ if $\theta_0 = \frac{\pi}{2}$; and
    \item lies in $(s(s + 1 - n),\infty)$ if $0 < \theta_0 < \frac{\pi}{2}$.
  \end{itemize}
\end{proposition}
\begin{remark}
  In terms of the spectral parameter $\nu$, this says that $\lambda_{\nu}$ is not an eigenvalue for $\frac{n - 1}{2} - \left|\frac{2s - n - 1}{2}\right| \leq \nu \leq \frac{n - 1}{2} + \left|\frac{2s - n - 1}{2}\right|$, and that
  taking $\nu \geq \frac{n - 1}{2}$, the first eigenvalue occurs for $\nu$ in $(\frac{n - 1 + |2s - n - 1|}{2},s)$, at $s$, or in $(s,\infty)$, respectively.
\end{remark}

\begin{proof}
  It is immediate from Lemma \ref{llem} that the lowest eigenvalue $\lambda_0$ satisfies $\lambda_0 \geq (s - 1)(s - n)$. From the same lemma, it follows that
  the only solutions $u$ to $L_su = (s - 1)(s - n)u$ are the constant functions. These, however, do not satisfy $u(0) = 0$. Thus, $\lambda_0 > (s - 1)(s - n)$.

  Suppose $\theta_0 \in (0,\frac{\pi}{2}]$ and that $\lambda \in ((s - 1)(s - n),(s(s + 1 - n))$. Suppose that $u \in C^{\infty}([0,\theta_0])$
  satisfies our boundary conditions, and that
  that $L_su = \lambda u$.
  By considering once again the indicial roots of this equation, we can write
  $u(\theta) = \sin^{2s - n}(\theta)v(\theta)$, where $v(\theta)$ is smooth.
  The equation then transforms to
  \begin{equation}
    \label{tformveq}
    v''(\theta) + (2s + 1 - n)\cot(\theta)v'(\theta) + (\lambda + s(n - s - 1))v(\theta) = 0,
  \end{equation}
  and since $u'(\theta_0) = 0$, we conclude that 
  \begin{equation}
    \label{vbound}
    v'(\theta_0) = (n - 2s)\cot(\theta_0)v(\theta_0).
  \end{equation}
  Then we have
  \begin{equation}
  \begin{split}
    \sin^{2s + 1 - n}(\theta)v''(\theta) + (2s + 1 - n)&\sin^{2s - n}(\theta)\cos(\theta)v'(\theta)\\
      &= (s(s + 1 - n) - \lambda)\sin^{2s + 1 - n}(\theta)v(\theta).\label{veq}\\
  \end{split}
  \end{equation}
  Thus,
  \begin{align*}
    \frac{d}{d\theta}\left[ \sin^{2s + 1 - n}(\theta) v'(\theta) \right] &= (s(s + 1 - n) - \lambda)\sin^{2s + 1 - n}(\theta)v(\theta),\text{ so}\\
    \int_0^{\theta_0}v(\theta)\frac{d}{d\theta}\left[ \sin^{2s + 1 - n}(\theta)v'(\theta) \right] d\theta &= (s(n - s - 1) - \lambda)\int_0^{\theta_0}\sin^{2s + 1 - n}(\theta)
    v(\theta)^2d\theta.
  \end{align*}
    Integrating by parts, we get
  \begin{equation*}
  \begin{split}
    \left.\sin^{2s + 1 - n}(\theta)v(\theta)v'(\theta)\right|_{0}^{\theta_0} - &\int_0^{\theta_0}\sin^{2s + 1 - n}(\theta)v'(\theta)^2d\theta\\
    &= (s(s + 1 - n) - \lambda)\int_0^{\theta_0}\sin^{n + 1}(\theta)v(\theta)^2d\theta.\\
  \end{split}
  \end{equation*}
  Applying our boundary condition gives
  \begin{align*}
    (n - 2s)\sin^{2s - n}(\theta_0)\cos(\theta_0)v(\theta_0)^2 - \int_0^{\theta_0}\sin^{n + 1}(\theta_0)v'(\theta_0)^2d\theta =\\
    \qquad(s(s + 1 - n) - \lambda)
    \int_0^{\theta_0}\sin^{n + 1}(\theta)v(\theta)^2d\theta \geq 0
  \end{align*}
  Since $\theta_0 \leq \frac{\pi}{2}$ and $s > \frac{n}{2}$, it is plain that this equality can hold only if $v(\theta) \equiv 0$.

  Now suppose that $\lambda = s(s + 1 - n)$, and that (\ref{tformveq}) holds for some $v$ satisfying the boundary condition (\ref{vbound}). It follows immediately that
  $\sin(\theta)v''(\theta) + (2s + 1 - n)\cos(\theta)v'(\theta) = 0$, from which we conclude that $v'(\theta) = b\sin^{n - 1 - 2s}(\theta)$ for some $b$. Thus,
  for some $a$,
  \begin{equation*}
    v(\theta) = a + b\int_{\theta_0}^{\theta}\sin^{n - 1 - 2s}(\phi)d\phi.
  \end{equation*}
  Since $v$ is smooth and $s > \frac{n}{2}$, we conclude that $b = 0$. Then $v \equiv a$ can be nonvanishing and satisfy (\ref{vbound}) if and only if
  $\theta_0 = \frac{\pi}{2}$. In that case, we see that $u(\theta) = \sin^{2s - n}(\theta)$ is a solution to $L_su = \lambda u$.

  Before handling the last case, let $v(\theta) = u'(\theta) = \sin^{2s - n - 1}(\theta)\cos(\theta)$. Then by differentiating both sides of the equation
  $L_su = s(s + 1 - n)u$, we find that $v$ satisfies the equation
  \begin{equation*}
  \begin{split}
    -v''(\theta) + (2s - n - &1)\cot(\theta)v'(\theta) + (n + 1 - 2s)\csc^2(\theta)v(\theta) + \\
   &(s - 1)(s - n)v(\theta) = s(n + 1 - s)v(\theta),
  \end{split}
  \end{equation*}
  with boundary conditions $v(0) = 0 = v\left( \frac{\pi}{2} \right)$. Multiplying through by a factor of $-\sin^{n + 1 - 2s}(\theta)$, we can rewrite this equation
  as
  \begin{equation*}
  \begin{split}
    (\sin^{n + 1 - 2s}(\theta)v'(\theta))' +& [(2s - 1 - n)\csc^{2s + 1 - n}(\theta)\\
     & + (2s - n)\csc^{2s - 1 - n}(\theta)]v(\theta) = 0.
  \end{split}
  \end{equation*}
  Obviously from its definition, $v(\theta)$ has no zeros on $\left( 0,\frac{\pi}{2} \right)$. Notice also
  that $\int_0\frac{\sin^{2s - 1 - n}(\theta)}{v(\theta)^2}d\theta = \infty =
  \int^{\frac{\pi}{2}}\frac{\sin^{2s - 1 - n}(\theta)}{v(\theta)^2}d\theta$.
  Now let $\lambda > s(n + 1 - s)$, and suppose that $w(\theta)$ satisfies
  \begin{equation*}
  \begin{split}
    (\sin^{n + 1 - 2s}(\theta)w'(\theta))' + [(2s - &1 - n)\csc^{2s + 1 - n}(\theta)\\
     & + (\lambda - (s - 1)(s - n))\csc^{2s - 1 - n}(\theta)]w(\theta) = 0
  \end{split}
  \end{equation*}
  on $\left( 0,\frac{\pi}{2} \right)$. Then $w$ has at least one zero on $\left( 0, \frac{\pi}{2} \right)$ by Theorem \ref{sturmthm},
  since $\lambda - (s - 1)(s - n) > 2s - n$.

  Now suppose that $\theta_0 \in \left( \frac{\pi}{2},\pi \right)$. Let $\lambda_0 > 0$ be the lowest eigenvalue of $L_s$, with eigenfunction
  $u_0$ satisfying $u_0(0) = 0 = u_0'(\theta_0)$. Set $v_0(\theta) = u_0'(\theta)$. We have already shown that
  $\lambda_0 \neq s(s - 1 - n)$. Now differentiating both sides of $L_su = \lambda_0u$,
  we find that $v$ satisfies the equation
  \begin{equation}
  \begin{split}
    \label{diffeq}
    -v''(\theta) + (2s - 1 - n)\cot(\theta)v'(\theta) + (n + &1 - 2s)\csc^2(\theta)v(\theta)\\
    &+ (s - 1)(s - n)v(\theta) = \lambda_0 v(\theta)
  \end{split}
  \end{equation}
  with homogeneous Dirichlet conditions at both endpoints. Moreover, $\lambda_0$ must be the lowest eigenvalue of this boundary value problem as well, or
  we could produce a lower eigenvalue to $L_su = \lambda u$ by integration. Now as is well known, the lowest eigenfunction of a positive (or boundedly negative) operator with
  homogeneous Dirichlet boundary values is nonvanishing away from the endpoints. This can be shown, for example, by adapting the proof of Proposition
  5.2.4 of \cite{tay1} to the simpler ODE case, using the maximum principle given in Theorem 26.XVIII of \cite{wal98}. So we may conclude that $v(\theta)$ has
  no zeros on $(0,\theta_0)$, and in particular on $\left( 0,\frac{\pi}{2} \right)$. Thus, it follows that $\lambda_0 < s(s + 1 - n)$.
\end{proof}

We can in fact characterize all of the eigenvalues of $L_s$.

\begin{proposition}
  \label{charlem}
  The eigenvalues of $L_s$ are the values $\lambda_{\nu}$, where $\nu \geq \frac{n - 1}{2}$ runs over the non-negative solutions to the equation
  \begin{equation}
    \label{hypergeq}
    \left.\frac{d}{d\theta}\left[ \sin^{2s - n}(\theta)F_{s - \nu,\nu + s - n + 1}^{s - \frac{n}{2} + 1}\left( \sin^2\left( \frac{\theta}{2} \right) \right) \right]\right|_{\theta = \theta_0} = 0.
  \end{equation}
  Here $F_{ab}^c(x)$ is the hypergeomtric function.
\end{proposition}
Note that by the identity $\frac{d}{dx}F_{ab}^c(x) = \frac{ab}{c}F_{a + 1,b + 1}^{c + 1}(x)$, the equation (\ref{hypergeq}) is equivalent to
\begin{equation}
\label{hyperchareq}
\begin{split}
    (2s - n)\cot(\theta_0)\csc(\theta_0)&F_{s - \nu,\nu + s - n + 1}^{s - \frac{n}{2} + 1}\left( \sin^2\left( \frac{\theta_0}{2} \right) \right) =\\
    &\frac{(\nu - s)(\nu + s - n + 1)}{2s - n + 2}F_{s - \nu + 1,\nu + s - n + 2}^{s - \frac{n}{2} + 2}\left( \sin^2\left( \frac{\theta_0}{2} \right) \right).
\end{split}
\end{equation}

\begin{proof}
  We look for solutions $u(\theta)$ to the equation $L_su(\theta) = \lambda_{\nu}u(\theta)$, satisfying $u(0) = 0 = u'(\theta_0)$. As shown before,
  such a solution will necessarily be $O(\theta^{2s - n})$. We thus write $u(\theta) = \sin^{2s - n}(\theta)v(\theta)$; it was shown earlier that
  $v$ then satisfies equation (\ref{tformveq}). We introduce the substitution $x = \sin^{2}\left( \frac{\theta}{2} \right)$, and set
  $v(\theta) = l(x(\theta))$. Then setting $a = s - \nu$, $b = \nu + s - n + 1$, and $c = s - \frac{n}{2} + 1$, equation (\ref{tformveq}) transforms
  to
  \begin{equation*}
    x(1 - x)l''(x) + (c - (1 + a + b)x)l'(x) - abl(x) = 0.
  \end{equation*}
  This is the hypergeometric equation, and the solution that is smooth at $x = 0$ is (up to scaling) the hypergeometric function
  $F_{ab}^c(x)$. Thus, we find $u(\theta) = \sin^{2s - n}(\theta)F_{n - \nu,\nu + 1}^{\frac{n}{2} + 1}\left( \sin^2\left( \frac{\theta}{2} \right) \right)$.
  The claim then follows by differentiating $u$ and requiring that $u'(\theta_0) = 0$.
\end{proof}

In the case $\theta_0 = \frac{\pi}{2}$ we can say much more.

\begin{proposition}
  \label{piprop}
  Let $L_s$ be as in Proposition \ref{specprop}, and $\theta_0 = \frac{\pi}{2}$. Then the spectrum is given in terms of the spectral
  parameter by
  $\spec(L_s) = \left\{\lambda_{\nu}: \nu =  s + 2k, \text{ where } k \in \mathbb{Z}^{\geq 0} \right\}$, and eigenfunctions are given by
  $w_k(\theta) = \sin^{2s - n}(\theta)C_{2k}^{s + \frac{1 - n}{2}}(\cos(\theta))$, where $C_j^{\alpha}$ are
  the Gegenbauer polynomials.
\end{proposition}

\begin{proof}
  In the equation $L_su = \lambda_{\nu}u$, we assume a solution of the form
  $u(\theta) = \sin^{2s - n}(\theta)v(\cos(\theta))$. This gives the equation
  \begin{equation}
    \begin{split}
    \sin^2(\theta)v''(\cos(\theta)) - &(2s - n + 2)\cos(\theta)v'(\cos(\theta))\\
    &+ [\nu(\nu + 1 - n) - s(s + 1 - n)]v(\cos(\theta)) = 0.
  \end{split}
  \end{equation}
  We make the substitution $x = \cos(\theta)$, and get the equation
  \begin{equation*}
    (1 - x^2)v''(x) - (2s - n + 2)xv'(x) + [\nu(\nu + 1 - n) - s(s + 1 - n)]v(x) = 0.
  \end{equation*}
  This is the Gegenbauer equation $(1 - x^2)v''(x) - (1 + 2\alpha)xv'(x) + k(k + 2\alpha)v(x) = 0$ with $\alpha = s + \frac{1 - n}{2}$ and
  $k = \nu - s$. Recall that a solution to this equation for each integer $k \geq 0$ is given by the $k$th-degree Gegenbauer polynomial
  $C_k^{\alpha}(x)$; that the polynomials $\left\{ C_k^{\alpha}(x) \right\}_{k = 0}^{\infty}$ are an orthonormal basis for
  $L^2\left( [-1,1],(1 - x^2)^{\alpha/2 - 1} \right)$;
  and that each polynomial $C_{k}^{\alpha}(x)$ has the same parity as $k$, with the even polynomials being nonvanishing at
  $0$. It follows that if we let $u_k(\theta) = \sin^{2s - n}(\theta)C_k^{\alpha}(\cos(\theta))$, then $u_k$ is a solution of the equation
  $L_su = \lambda_{s + k}u$, and that the solutions for $k$ even satisfy the condition that $u_k'\left( \frac{\pi}{2} \right) = 0$.
  By the orthonormal basis property of the Gegenbauer polynomials, it follows that every solution to the equation is one of the
  $u_k$; and we conclude that precisely the solutions for even $k$ satisfy our boundary conditions.
\end{proof}

We now turn to the mapping properties of $I_{s,\nu}$. Recall that the function spaces mentioned in the following proposition are defined in 
Section \ref{funcsubsec}.

\begin{proposition}
  \label{indexprop}
  Fix $n \geq 2 \in \mathbb{N}$, $s > \frac{n}{2}$, and $\theta_0 \in (0,\pi)$. For $\nu \geq n - s$, and $k \geq 1$,
  $I_{s,\nu}$ given by (\ref{iform}) maps $\mathcal{A}_{n,s,k}^0$ to $\mathcal{B}_{n,s,k}$.
  If $\lambda_{\nu} \notin \spec(L_s)$, then $I_{s,\nu}:\mathcal{A}_{n,s,k}^0 \to \mathcal{B}_{n,s,k}$ is bijective. Otherwise, $\dim \ker I_{s,\nu} = 1$ and
  the kernel is spanned by a smooth function $\beta_{s,\nu}$; the image in that case is the orthogonal complement in $\mathcal{B}_{n,s,k}$ of
  $\beta_{s,\nu}$ with respect to the measure $\sin^{-(1+n)}(\theta)d\theta$.

  When $\nu \notin \spec(L_s)$, the inverse of $I_{s,\nu}$ is given by 
  the Green's operator 
  \begin{equation}
    \label{greensop}
    Gf(\theta) = \sin^{n - s}(\theta)p(\theta)\int_0^{\theta}\frac{q(\phi)f(\phi)}{\sin^{s + 1}(\phi)}d\phi + \sin^{n - s}(\theta)q(\theta)\int_{\theta}^{\theta_0}
    \frac{p(\phi)f(\phi)}{\sin^{s + 1}(\phi)}d\phi,
  \end{equation}
  where $q(\theta) = O(\theta^{2s - n})$ is smooth, with $\sin^{s - n}(\theta)q(\theta)$ even in $\theta$; and
  \begin{itemize}
    \item if $s - \frac{n}{2} \notin \mathbb{Z}$ or $\nu \in \left\{ n - s,n - s + 1,\cdots,\frac{n}{2} - 1 \right\}$, then $p(\theta)$ is smooth and even in $\theta$; while
    \item if $s - \frac{n}{2} \in \mathbb{Z}$ and $\nu \notin \left\{ n - s,n - s + 1,\cdots,\frac{n}{2} - 1 \right\}$, 
      then $p(\theta) \in C^{\infty}([0,\theta_0]) + \theta^{2s - n}\log(\theta)C^{\infty}([0,\theta_0])$, and is even
      up to order $2s - n$.
  \end{itemize}
  If $n$ is odd and $s \geq n$ is an integer, and if $f \in \mathcal{B}_{n,s,k}$ is even in $\theta$ through order $s + 1$, then $Gf$ is smooth.
\end{proposition}

\begin{proof}
  For convenience, we set $I = [0,\theta_0]$.

  It is straightforward that $I_{s,\nu}$ maps $\mathcal{A}_{n,s,k}^0$ into $\mathcal{B}_{n,s,k}$. The anomaly in the definition
  of $\mathcal{B}_{n,s,1}$ is because $s$ is an indicial root of $I_{s,\nu}$, so that $I_{s,\nu}(\theta^su) = O(\theta^{s + 1})$ for any
  $u \in C^{\infty}(I)$.

  We first wish to identify independent global solutions to the equation $I_{s,\nu}u = 0$. In fact, it will again be easier to
  consider the operator $\nu(\nu + 1 - n) - L_s$ acting on $\sin^{s - n}(\theta)u(\theta)$. 
  As before, we set $x = \sin^2\left( \frac{\theta}{2} \right)$,
  and let $l(x)$ be defined by $u(\theta) = \sin^{n - s}(\theta)l(x(\theta))$. Then transforming the equation
  $(\lambda_{\nu} - L_s)u = 0$ yields the equation
  \begin{equation}
  \begin{split}
    \label{xformedeq}
    x(1 - x)l''(x) + &\left(\frac{n}{2} + 1 - s - (n + 2(1 - s))x\right)l'(x) +\\
    &[(s - 1)(n - s) + \nu(\nu + 1 - n)]l(x) = 0.
  \end{split}
  \end{equation}
  Let $a = n - s -\nu$, let $b = \nu + 1 - s$, and let $c = \frac{n}{2} + 1 - s$. Then (\ref{xformedeq}) becomes
  \begin{equation*}
    x(1 - x)l''(x) + (c - (1 + a + b)x)l'(x) - abl(x) = 0,
  \end{equation*}
  which is again the hypergeometric differential equation. The indicial roots of the equation at $x = 0$ are $\gamma = 0$
  and $\gamma = s - \frac{n}{2}$. It follows that we can find two independent solutions $u_0$ and $u_s$ having the following properties:
  $u_s \in C^{\infty}(I)$ and $u_s = O(\theta^s)$, while $u_0(0) \neq 0$. Furthermore, $\sin^{s - n}(\theta)u_s(\theta)$ is even in $\theta$.
  If $s - \frac{n}{2} \notin \mathbb{Z}$, 
  then $u_0 \in C^{\infty}(I)$, while if $s - \frac{n}{2} \in \mathbb{Z}$,
  then it follows by standard hypergeometric theory (for example, paragraph 15.10(c) of \cite{nist}) that $u_0 \in C^{\infty}(I)$ if
  $\nu = n - s, n - s + 1,\cdots,\frac{n}{2} - 1$, and that $u_0 \in C^{\infty}(I) + \theta^{2s - n}\log(\theta)C^{\infty}(I)$ otherwise, with nonzero
  logarithmic coefficient.
  
  We already know that $I_{s,\nu}$ is injective if and only if $\lambda_{\nu} \notin \spec L_s$.
  We now show that $I_{s,\nu}$ is surjective whenever it is injective. Let $q(\theta)$
  be the solution to $I_{s,\nu}(\sin^{n - s}(\theta)q) = 0$ with $\lim_{\theta \to 0^+}q(0)/\theta^{2s - n} = 1$. 
  Let $p(\theta)$ be the solution
  to $I_{s,\nu}(\sin^{n - s}(\theta)p) = 0$ with $p'(\theta_0) = 0$ and $p(0) = 1$ (this can be taken to be nonzero by the assumption
  of injectivity). It follows that
  $p$ is some linear combination of $u_0$ and $u_s$ with nontrivial coefficient for $u_0$.
  Now by Abel's identity, the Wronskian of $\sin^{n - s}(\theta)p(\theta)$ and $\sin^{n - s}(\theta)q(\theta)$ is $c\sin^{n - 1}(\theta)$ for some
  $c \neq 0$. It thus follows from standard formulas that the Green's function for $I_{s,\nu}$ is given by
  \begin{equation*}
    \Phi(\theta,\phi) = \left\{\begin{array}{lr}
      \frac{\sin^{n - s}(\theta)q(\theta)p(\phi)}{c\sin^{s + 1}(\phi)} & \theta < \phi\\
      \frac{\sin^{n - s}(\theta)q(\phi)p(\theta)}{c\sin^{s + 1}(\phi)} & \theta \geq \phi
    \end{array}\right.,
  \end{equation*}
  and it is elementary to show that the Green's operator 
  \begin{equation}
    \label{greensoplocal}
    Gf(\theta) = \sin^{n - s}(\theta)p(\theta)\int_0^{\theta}\frac{q(\phi)f(\phi)}{c\sin^{s + 1}(\phi)}d\phi + \sin^{n - s}(\theta)q(\theta)\int_{\theta}^{\theta_0}\frac{p(\phi)f(\phi)}
    {c\sin^{s + 1}(\phi)}d\phi
  \end{equation}
  is a right inverse to $I_{s,\nu}$ when $f \in \mathcal{B}_{n,s,k}$, and a left inverse
  to $I_{s,\nu}$ when $f = I_{s,\nu}u$ with $u \in \mathcal{A}_{n,s,k}^0$. (For the formula in the statement, we may absorb $c$ into $q$.)
  We wish to show that the equation $I_{s,\nu}u = f$ can be solved whenever $f \in \mathcal{B}_{n,s,k}$, and that
  in particular, $G$ maps $\mathcal{B}_{n,s,k}$ to $\mathcal{A}_{n,s,k}^0$. It is easy to check that
  $(Gf)'(\theta_0) + (s - n)\cot(\theta_0)(Gf)(\theta_0) = 0$ and that, for $f \in \mathcal{B}_{n,s,k}$, $Gf(0) = 0$. 
  Thus, if $Gf \in \mathcal{A}_{n,s,k}$, then
  $Gf \in \mathcal{A}_{n,s,k}^0$. It remains to show that $Gf(\theta) \in \mathcal{A}_{n,s,k}$.

  We do this in two steps. We first show that $G$ maps $\mathcal{B}_{n,s,k}$ into $\mathcal{A}_{n,s,k} + 
  \theta^{n - s + (k + 1)(2s - n)}\log(\theta)^{k + 1}C^{\infty}(I)$ by asymptotically expanding the integrands in (\ref{greensoplocal}) and considering the kinds
  of terms that can arise.

  Suppose $f \in \theta^{n - s + 1}C^{\infty}(I)$. Then since $q(\phi) = O(\phi^{2s - n})$, the integral in the first term of (\ref{greensoplocal}) is smooth.
  The factor $p(\theta)$ may contain a term with a factor of $\theta^{2s - n}\log(\theta)$ if $s - \frac{n}{2} \in \mathbb{Z}$. Thus, the first term is in $\mathcal{A}_{n,s,k}$.

  Because, in general, $\theta^j\int_0^{\theta}\theta^{-k} d\theta$ is smooth for $0 \leq k \leq j$ unless $k = 1$, the integral in the second term is smooth unless either
  $p(\phi)$ has a $\phi^{2s - n}\log\phi$ term in it (which could happen if $s - \frac{n}{2} \in \mathbb{Z}$), or there is a nonvanishing term of order
  $\phi^{-1}$ in the expansion of the integrand. In the first case, when $p(\phi)$ has a term in its expansion of the form
  $\phi^{2s - n}\log(\phi)$, the term yields a log at order $\theta^{s + 1}\log(\theta)$ and at higher powers in $\theta$. When
  the integrand contains a term of order $\phi^{-1}$, the second term yields a term of the form $\theta^s\log(\theta)$. Thus, in this case,
  $Gf \in \mathcal{A}_{n,s,k}$.

  Now suppose that $f \in \theta^{s + 1}\log(\theta)C^{\infty}(I)$.  Then the first integral
  yields a log at order $\theta^{s + 1}\log(\theta)$, and a possible $\log(\theta)^2$ at order $\theta^{3s + 1 - n}\log(\theta)^2$, if $p(\theta)$ has a term
  of the form $\theta^{2s - n}\log(\theta)$.

  When the integrand of the second term is expanded in an asymptotic series, we get several nonsmooth terms. First, at lowest order we get
  a term of the form $\theta^{s + 1}\log(\theta)$, with logs at higher orders as well. Now, it is an elementary result that
  \begin{equation*}
    \int_0^{\theta}\theta^j\log(\theta)^2d\theta = \theta^{j + 1}\left( 2(1 + j)^{-3} - 2(1 + j)^{-2}\log(\theta) + (1 + j)^{-1}\log(\theta)^2 \right).
  \end{equation*}
  Thus, if $p(\phi)$ has a term of the form
  $\phi^{2s - n}\log(\phi)$, then we get a term of the form $\theta^{3s + 1 - n}\log(\theta)^2$ as well.
  If $k = 1$, we have shown, as desired,
  that $Gf \in \mathcal{A}_{n,s,k} + \theta^{n - s + (k + 1)(2s - n) + 1}\log(\theta)^{k + 1}C^{\infty}(I)$. If $k > 1$, then in both the cases so far
  considered, we have seen that $Gf \in \mathcal{A}_{n,s,k}$.

  Now suppose that $k > 1$ and that $f(\theta) \in \theta^{n - s}(\theta^{2s - n}\log(\theta))^jC^{\infty}(I)$ for some $2 \leq j \leq k$. Taking account of
  the possible $\theta^{2s - n} \log(\theta)$ term in $p(\theta)$, an integral formula analogous to the above shows that the first term yields a function in
  $\theta^{n - s}\oplus_{i = 0}^{j + 1}(\theta^{2s - n}\log(\theta))^iC^{\infty}(I)$. This clearly lies in the desired space. It is easy to see that the second term
  lies in the same space.

  Thus, we see that $Gf \in \mathcal{A}_{n,s,k} + \theta^{n - s + (k + 1)(2s - n)}\log(\theta)^{k + 1}C^{\infty}([0,\theta_0])$. Our second step is now to show that, in fact,
  the last term does not arise. We have seen that $I_{s,\nu}Gf = f \in \mathcal{B}_{n,s,k}$. But for no $k \geq 1$ is $n - s + (k + 1)(2s - n)$ an incidial root
  of the operator $I_{s,\nu}$ at $\theta = 0$. Thus, if $0 \not\equiv u \in C^{\infty}(I)$, then $I_{s,\nu}(\theta^{n - s + (k + 1)(2s - n)}\log(\theta)^{k + 1}u)$ yields a term in
  $\theta^{n - s + (k + 1)(2s - n)}\log(\theta)^{k + 1}C^{\infty}(I)$, which is not canceled by any other term, and this precludes $I_{s,\nu}u$ from lying in $\mathcal{B}_{n,s,k}$. Thus, we must in fact have that
  $G:\mathcal{B}_{n,s,k} \to \mathcal{A}_{n,s,k}^0$.
 
  Now $I_{s,\nu}$ is injective, and it has a right inverse. Thus, it is a bijection.

  It remains to show that if $I_{s,\nu}$ is \emph{not} injective, then it is also not surjective. Suppose that $\ker I_{s,\nu} \neq \emptyset$.
  Rename $q$ so that $0 \neq q$ satisfies $\sin^{n - s}(\theta)q(\theta) \in \ker I_{s,\nu} \subseteq \mathcal{A}_{n,s,k}^0$, and note that $q = O(\theta^{2s - n})$. 
  Then Lemma \ref{llem} (\ref{llemb}) together with (\ref{ilrel}) makes it clear that any $f \in \image I_{s,\nu}$ must be orthogonal
  to $q$ with respect to measure $\sin^{-(s + 1)}(\theta) d\theta$. Since $v = \sin(\theta)q \in \mathcal{B}_{n,s,k}$ is clearly not orthogonal
  to $q$, we conclude that $I_{s,\nu}$ is not surjective.

  If $q$ is orthogonal to $f$ with respect to the measure $\sin^{-(s + 1)}(\theta)$, then let $p$ be any solution of the equation
  $I_{s,\nu}(\sin^{n - s}(\theta)p(\theta)) = 0$ with $p(\theta) = 1 + o(1)$. 
  Then it is easy to check that with $p$ and $q$ reinterpreted according to these definitions,
  equation (\ref{greensoplocal}) still gives a solution, of course not unique, to the equation $I_{s,\nu}u = f$, with
  $\left.u(\theta)/\sin^{n - s}(\theta)\right|_{\theta = 0} = 0 = u'(\theta_0) + (s - n)\cot(\theta_0)u(\theta_0)$. 
  The hypothesis that $q$ is orthogonal to $f$ is necessary in showing that the Robin condition holds.
  Note that orthogonality to $q$ with respect to the measure $\sin^{-(s + 1)}(\theta)d\theta$ is equivalent to orthogonality to
  $\beta_{s,\nu} = \sin^{n - s}(\theta)q(\theta)$ with respect to the measure $\sin^{-(n + 1)}(\theta)d\theta$.

  We finally turn to the last claim. If $n$ is odd and $s \geq n$ is an integer, and if $f \in \mathcal{B}_{n,s,k}$ is even through order $s + 1$, then no
  term of the form $\theta^{-1}$ can appear in the integrand in (\ref{greensop}), and no logarithms appear in $p$ or $q$. This
  yields the claim.
\end{proof}

This result will be used, for general $k$, in our formal solution of the Einstein equations in the next section. 

\subsection{The Laplacian}\label{lapsec}

We now continue to study the Laplace boundary value problem at hand. Recall that we are letting $\tX = [0,\theta_0] \times S \times [0,\varepsilon)$ be the pullback of the blowup of a cornered space
by a diffeormophism as in Corollary \ref{polarcor}, and are assuming that $g$ is the pullback of a constant-angle admissible metric on the blowup, in the form
(\ref{formcopy2}). The goal of this section is to prove
Theorem \ref{lapthm}. This problem
requires only the $k = 1$ case of Proposition \ref{indexprop} since it deals with a linear equation.
For this section, we specialize to the case $\theta_0 = \frac{\pi}{2}$, for which we have a full, explicit solution to the
eigenvalue problem. The techniques we use would be of relevance in studying other cases as well, although the behavior would depend crucially on the
spectrum, which might display a variety of behaviors in general.

We first undertake some preparations in the case that $2s \in \mathbb{Z}$, which we now assume.

Now it is easy to show that for $v \in C^{\infty}(\tS)$,
\begin{equation*}
  \begin{split}
    \Delta_g\left( \rho^{\nu}\log(\rho)^kv \right) + s(&n - s)\left( \rho^{\nu}\log(\rho)^kv \right) = \\
    &\rho^{\nu}\left[ \log(\rho)^kI_{s,\nu}v + k(1 + \nu - k)
\log(\rho)^{k - 1}\sin^2(\theta)v\right.\\
&\left.+k(k - 1)\log(\rho)^{k - 2}\sin^2(\theta)v\right] + o(\rho^{\nu}).
\end{split}
\end{equation*}
Motivated by this, we define an operator $J_{s + q}$ on $\mathcal{E}_{s,q}$ (see page \pageref{Enq}) by setting
\begin{equation*}
  \begin{split}
  J_{s + q}(\rho^{s + q}\log(\rho)^kb) =& \rho^{s + q}(\log(\rho)^kI_{s,s + q}(b) + k(1 + s + q - k)\log(\rho)^{k - 1}\sin^2(\theta)b\\
  &+ k(k - 1)\sin^2(\theta)\log(\rho)^{k - 2}b),
\end{split}
\end{equation*}
where $b \in \mathcal{A}_{n,s,1}^0$, and extending linearly.
For $q$ even, we define 
\begin{equation*}
  \mathcal{E}_{s,q}^0 = \left\{ \eta = \rho^{s + q}\sum_{i = 0}^{\frac{q}{2} + 1}\log(\rho)^ib_i(\theta): b_{\frac{q}{2} + 1} \in \ker I_{s,s + q} \right\}.
\end{equation*}
For consistency of notation, we define $\mathcal{E}_{s,q}^0 = \mathcal{E}_{s,q}$ when $q$ is odd.

\begin{proposition}
  \label{jprop}
  Let $2s \in \mathbb{Z}$.
  For $q \geq 1$ odd, $J_{s + q}:\mathcal{E}_{s,q} \to \mathcal{F}_{s,q}$ is an isomorphism.
  For $q \geq 0$ even, $J_{s + q}:\mathcal{E}_{s,q}^0 \to \mathcal{F}_{s,q}$ is surjective, with a one-dimensional kernel spanned
  by $\rho^{s + q}\sin^{n - s}(\theta)w_{\frac{q}{2}}(\theta)$, where $w_{\frac{q}{2}}$ is as in Proposition \ref{piprop}.
\end{proposition}

\begin{proof}
  That $J_{s + q}$ has the given codomains follows immediately from the definitions of $\mathcal{E}_{s,q}$, $\mathcal{E}_{s,q}^0$, $\mathcal{F}_{s,q}$, and $J$.

  For the isomorphism claim, we assume $q = 2j + 1$ is odd. We set $J = J_{s + q}$. By Propositions \ref{piprop} and \ref{indexprop},
  $I_{s,s + 2j + 1}:\mathcal{A}_{n,s,1}^0 \to \mathcal{B}_{n,s,1}$ is a bijection. So suppose we wish to solve
  $J\eta = f$, where $\eta = \rho^{s + 2j + 1}\sum_{i = 0}^{j + 1}\log(\rho)^ib_i(\theta)$ and $f = \rho^{s + 2j + 1}\sum_{i = 0}^{j + 1}\log(\rho)^ic_i(\theta)$.
  We will do this term by term, starting with the highest power of $\log$ and working down.
  We can uniquely solve $I_{s,s + 2j + 1}b_{j + 1} = c_{j + 1}$. Set $\eta^{(j + 1)} = \rho^{s + 2j + 1}\log(\rho)^{j + 1}b_{j + 1}(\theta)$, and note
  that $f - J\eta^{(j + 1)} = \rho^{s + 2j + 1}\sum_{i = 0}^j\log(\rho)^ic_i^{(j + 1)}(\theta)$, where $c_i^{(j + 1)}$ may differ from $c_i$. Now suppose,
  by way of induction, that we have constructed $\eta^{(l)} \in \mathcal{E}_{s,2j + 1}$ such that $f - J\eta^{(l)} = \rho^{s + 2j + 1}
  \sum_{i = 0}^{l - 1}\log(\rho)^ic_i^{(l)}$.
  Then we can uniquely solve the equation $I_{s,s + 2j + 1}b_{l - 1} = c_{l - 1}^{(l)}$, and setting $\eta^{(l - 1)} = \eta^{(l)} + \rho^{s + 2j + 1}\log(\rho)^{l - 1}b_{l - 1}$,
  we easily get $f - J\eta^{(l - 1)} = \rho^{s + 2j + 1}\sum_{i = 0}^{l - 2}\log(\rho)^ic_i^{(l - 1)}$. Thus, by induction, we can solve
  $J\eta = f$. At each step, the solution is unique, and so we see that $J$ is an isomorphism.

  We now address the even case, $q = 2j$. In this case, $I_{s,s + 2j}$ has a one-dimensional kernel spanned by $w_{\frac{q}{2}}$, that is, $w_j$.
  We wish to solve the equation $J\eta = f$, where we take $\eta = \rho^{s + 2j}\sum_{i = 0}^{j + 1}\log(\rho)^ib_i(\theta)$ and
  $f = \rho^{n + 2j}\sum_{i = 0}^j\log(\rho)^ic_i(\theta)$. We cannot necessarily solve the equation $I_{n + 2j}b_j = c_j$, since
  $c_j$ may not be orthogonal to $w_j$, generically (see Proposition \ref{indexprop}). However, we may solve this using our freedom in the
  $\log(\rho)^{j + 1}$ term. First, notice that $\sin^2(\theta)w_j$ is not orthogonal to $w_j$. Now
  we define the coefficient $b_{j + 1} = \frac{1}{(j + 1)(s + j)}\frac{\langle c_j,w_j\rangle}{\langle \sin^2(\theta)w_j,w_j\rangle}w_j$, 
  where $\langle,\rangle$ here refers to the $L^2$ norm on
  $[0,\theta_0]$ with measure $\sin^{-(1 + s)}(\theta)d\theta$. (This is finite since $w_j = O(\theta^{2s - n})$ and $c_j = O(\theta^{n - s + 1})$.)
  Set $\eta^{(j + 1)} = \rho^{n + 2j}\log(\rho)^{j + 1}b_{j + 1}$. Then $f - J\eta^{(j + 1)} = \rho^{n + 2j}\sum_{i = 0}^j\log(\rho)^ic_i^{(j + 1)}$, where
  $c_j^{(j + 1)}$ is orthogonal to $w_j$. We can now proceed by induction as in the odd case; except that at each step $l$, we uniquely
  add a multiple of $w_j$ to $b_l$ to ensure that $c_{l - 1}^{(l)}$ is orthogonal to $w_j$. By induction, we may thus solve the equation.
  However, a solution is unique only up to addition of a multiple of $\rho^{s + 2j}w_j$.
\end{proof}

We now return to letting $s > \frac{n}{2}$ be arbitrary.

The following lemma follows easily from Lemma \ref{laplem}, and particularly from the fact that $n - s$ and $s$ are indicial roots of $\Delta_g + s(n - s)$ at $\theta = 0$.

\begin{lemma}
  \label{lapmaplem}
  Suppose that $u \in \mathcal{P}_s(\tX)$. Then $(\Delta_g + s(n - s))u \in \mathcal{P}_s(\tX)$ as well. Moreover, suppose that for any fixed $x \in S$, $u(\theta,x,\rho) \in \mathcal{E}_{s,q}^0$ 
  for some $q$ independent of $x$.  Then for each fixed $x$, there is some $f \in \mathcal{F}_{s,q}$ such that $(\Delta_g + s(n - s))u(\theta,x,\rho) = f(\theta,\rho) + o(\rho^{s + q})$.
\end{lemma}

\begin{proof}[Proof of Theorem \ref{lapthm}]
  We work in the decomposition given by Corollary \ref{polarcor}. In particular, $\tU, \theta$, and $\rho$ are as in that
  corollary. We write $\psi = \psi(x,\rho)$. Throughout the proof, primes will refer to a derivative with respect to $\theta$.

  Define $\tilde{u}_{0} \in \rho^{n - s}\sin^{n - s}(\theta)
  C^{\infty}(\tU)$ by $\tilde{u}_{0}(\theta,x,\rho) = \rho^{n - s}\sin^{n - s}(\theta)\psi(x,\rho)$. Now define $f_{0} \in \mathcal{B}_{n,s,1}\left(\frac{\pi}{2},S  \right)$ by
  $\Delta\tilde{u}_{0} + s(n - s)\tilde{u}_{0} = \rho^{n - s}f_0$. Then by Proposition \ref{indexprop}, we can uniquely solve $I_{s,n-s}\varphi_{0} = -f_{0}$ for $\varphi_{0} \in
  \mathcal{A}_{n,s,1}^0\left( \frac{\pi}{2},S \right)$. Set $u_{0} = \tilde{u}_{s} + \rho^{n - s}\varphi_{0}$. Then plainly,
  $\rho^{s - n}\sin^{s - n}(\theta)u_{0} \equiv \psi$, and $(\Delta_g + s(n-s))u_{0} = O(\rho^{n - s + 1})$. Also, $\partial_{\theta}u_{0}(\frac{\pi}{2}) = 0$.

  Now suppose that either, on the one hand, $2s \notin\mathbb{Z}$ and $k \in \mathbb{N}$ or, on the other hand, that $0 \leq k < 2s - n$; and suppose,
  in either case, that $u_k$ has been smoothly and uniquely defined in $\mathcal{R}_s(\tX)$ 
  so that
  \begin{enumerate}
    \item $\Delta_g u_k = O(\rho^{n - s + k + 1})$;
    \item $\rho^{s - n}\sin^{s - n}(\theta)u_k|_{\theta = 0} = \psi$; and 
    \item $u_k'(\theta_0) = 0$. 
  \end{enumerate}
  We wish to find $\varphi_{k + 1} \in \mathcal{A}_{n,s,1}^0(\frac{\pi}{2},S)$ so that
  $u_{k + 1} = u_k + \rho^{n - s + k + 1}\varphi_{k + 1}$ satisfies each of these conditions with $k$ replaced by $k + 1$.
  Define $f_{k + 1} \in \mathcal{B}_{n,s,1}\left( \frac{\pi}{2},S \right)$ by $f_{k + 1} = \rho^{-(n - s + k + 1)}(\Delta_g + s(n - s))u_k|_{\rho = 0}$. By Lemma \ref{laplem}, $f_{k + 1} = O(\theta^{n - s + 1})$, and indeed,
  since $(\Delta_g(\theta^s) + s(n - s)\theta^s) = O(\theta^{s + 1})$,
  $f_{k + 1} \in \mathcal{B}_{n,1}(\frac{\pi}{2},S)$ as desired.
  Fix $x \in S$. By Proposition \ref{indexprop}, we can uniquely solve $I_{s,n - s + k + 1}\varphi_{k + 1}(\theta,x) = -f_{k + 1}(\theta,x)$ in $\mathcal{A}_{n,s,1}^0$, with
  $\varphi_{k + 1}(0,x) = 0$ and $\varphi_{k + 1}'(\theta_0,x) = 0$. Plainly $\varphi_{k + 1}$ depends smoothly on $x$.
  Thus, $\varphi_{k + 1} \in \mathcal{A}_{n,s,1}^0\left( \frac{\pi}{2},S \right)$ is determined as desired.
  Thus, for any $m \in \mathbb{Z}$ (if $2s \notin \mathbb{Z}$), or for any $m \leq 2s - n$ (otherwise), we can,
  by induction and Proposition \ref{indexprop},
  construct a function $u_{m - 1} \in \mathcal{R}_s(\tX)$, 
  unique through order $n - s + m - 1$, such that $\rho^{s - n}\sin^{s - n}(\theta)u_{m - 1}|_{\tM} = \psi$, such that
  $u_{m - 1}'(\frac{\pi}{2})\equiv 0$, and such that $(\Delta_g + s(n - s)) u_{m - 1} = O(\rho^{n - s + m})$.
  If $2s \notin \mathbb{Z}$, we are done, except for the last paragraph below. From here, we therefore assume that $2s \in \mathbb{Z}$.

  At order $s$, $I_{s,s}$ is not surjective, so 
  our procedure generically fails for $k = 2s - n - 1$ unless, for each $x$, $f_n(\theta,x)$ is orthogonal to $\sin^{s}(\theta)$ with respect to the 
  measure $\sin^{-(n + 1)}(\theta)d\theta$, i.e.,
  unless it is orthogonal to $\csc^{n + 1 - s}(\theta)$. To proceed, we must introduce logarithmic terms, and for this we will use Proposition
  \ref{jprop}. Let $f_{2s - n} = \rho^{-s}(\Delta_g + s(n - s))u_{2s - n - 1}|_{\rho = 0} \in \mathcal{B}_{n,s,1}$ as above, and fix $x$. 
  Then notice that $\rho^sf_{2s - n}(\cdot,x)
  \in \mathcal{F}_{s,0}$. Thus, by Proposition \ref{jprop}
  with $q = 0$, there exists a solution $\Phi_{2s - n} \in \mathcal{E}_{s,0}$ to $J_s\Phi_{2s - n}(\theta) = -\rho^sf_{2s - n}(\theta,x)$,
  which however is determined only up to a term that is a multiple of $\rho^s\sin^s(\theta)$. Set $\varphi_{2s - n} = \rho^{-s}\Phi_{2s - n}$, and
  $u_{2s - n} = u_{2s - n - 1} + \rho^s\varphi_{s}$. Then since this procedure is smooth in $x$, $u_{2s - n}(\theta,x,\rho)$ satisfies $\Delta_g u_{2s - n} + s(n - s)u_{2s - n} = o(\rho^s)$, and the boundary conditions
  $\rho^{s - n}\sin^{s - n}(\theta)u_{2s - n}|_{\tM} = \psi$ and $u_{2s - n}'(\frac{\pi}{2}) = 0$ are satisfied. Notice that $u_{2s - n} \in \mathcal{P}(\tX)$.

  We proceed by induction.
  Suppose that $u_{2s - n + 2j} \in \mathcal{P}(\tX)$ has been successfully defined satisfying $\Delta_g u_{2s - n + 2j} = o(\rho^{s + 2j})$, with
  both boundary conditions as desired, and containing $(j + 1)$st powers of $\log(\rho)$. Then since $\Delta_g + s(n - s)$ is linear, $(\Delta_g + s(n - s))u_{2s - n + 2j}$ will
  likewise contain at most $(j + 1)$st powers of $\log(\rho)$. Fix $x \in S$. Let
  $F_{2s - n + 2j + 1} \in \mathcal{F}_{s,2j + 1}$ be such that $(\Delta_g + s(n - s))u_{2s - n + 2j}(\theta,x,\rho) = F_{n + 2j + 1}(\theta,\rho) + o(\rho^{n + 2j + 1})$.
  Such an $F$ exists by Lemma \ref{lapmaplem}.
  We wish to solve the equation
  $J_{s + 2j + 1}\Phi_{2s - n + 2j + 1} = -F_{2s - n + 2j + 1}$; by Proposition \ref{jprop}, we may do so uniquely, and plainly the solution varies smoothly in $x$.
  Then set $u_{2s - n + 2j + 1} = u_{2s - n + 2j} + \Phi_{2s - n + 2j + 1}$. Clearly, $u_{2s - n + 2j + 1} \in \mathcal{P}(\tX)$ satisfies
  $\Delta_g u_{n + 2j + 1} = o(\rho^{s + 2j + 1})$ and our boundary conditions.

  Next we wish to find $u_{2s - n + 2(j + 1)}$, satisfying our boundary conditions, such that $\Delta_g u_{2s - n + 2(j + 1)} = o(\rho^{s + 2(j + 1)})$.
  But this is exactly the same as the odd case, except that by Proposition \ref{jprop}, the solution will be unique only up to a term of the form
  $\rho^{s + 2(j + 1)}\sin^{n - s}(\theta)w_{j + 1}$, where $w_{j + 1}$ is as in Proposition \ref{piprop}. Hence, by induction, we get an infinite sequence $\{u_{k}\}_{k = 0}^{\infty}$
  such that $\Delta_g u_k + s(n - s)u_k = o(\rho^{n - s + k})$, $u_k|_{\tM} = \psi$, and $\partial_{\theta}u_k|_{\theta = \frac{\pi}{2}} = 0$, 
  and such that each member of the sequence $\left\{ u_k \right\}$ has at most $\left\lfloor\frac{k + 1 + n - 2s}{2}\right\rfloor$ powers of $\log(\rho)$.

  Thus, by Borel's Lemma, as stated in \cite{e56}, there exists a function $u \in \mathcal{P}(\tX)$ such that
  $(\Delta_g + s(n - s)) u = O(\rho^{\infty})$, such that $\partial_{\theta}u|_{\theta = \frac{\pi}{2}} = 0$, and such that
  $\rho^{s - n}\sin^{s - n}(\theta)u|_{\tM} = \psi$. 
\end{proof}

Notice that, in order to uniquely determine a solution for $2s \in \mathbb{Z}$, we would need to specify not only $\psi$, but also
a scalar function $\eta_k \in C^{\infty}(S)$ at order $n + 2k$ for all $k \geq 0$.

We also remark that the power of $\rho$ stated in the boundary condition -- and the lowest appearing in the expansion -- has somewhat more flexibility than the power of $\sin(\theta)$. The expansion can begin
at order $\rho^q\sin^{n - s}(\theta)$ for any $q$ that is not an indicial root of $L_s$. For example, when $2s \notin\mathbb{Z}$, we could simply prescribe that $\sin^{s - n}(\theta)u|_{\theta = 0} = \psi$. In general, the
first $\log(\rho)$ term would be expected to appear at the first indicial root that was integrally separated from the starting power.
As there is a host of different situations that could be studied along these lines and we do not wish to consider them
all here, we leave this problem for now.

%% file: tex/ein.tex
\section{Einstein Metrics: Formal Existence}\label{einsec}

In this section, we consider formal existence of CAH Einstein metrics on a cornered
space $(X,M,Q)$. This has famously been studied in the AH setting in \cite{fg12}, where the boundary
data is a conformal class on the conformal infinity $M$. Once again we require a boundary condition at $Q$,
and we continue to follow
\cite{ntu12} in requiring that $Q$ be totally umbilic and of constant mean curvature,
so that its second fundamental form $K_Q$ satisfies $K_Q = \lambda g|_{TQ}$ away
from the corner $S = Q \cap M$.
As we will see, this boundary condition interacts particularly nicely with the geometry of CAH spaces.

We will take as our data a smooth manifold $M^n$
with boundary $S$, equipped with an asymptotically hyperbolic conformal class $[h]$; and a constant $\lambda \in (-1,1)$.
Ideally, we would like to realize $M$ as the infinite boundary of an appropriate cornered space $(X,M,Q)$, 
and then to construct an Einstein CAH metric $g$
satisfying $\ric(g) + ng = 0$ to as high an order as possible at the corner, and
satisfying $K_Q = \lambda g|_{TQ}$ along a finite boundary $Q$. We know by Section \ref{smoothsec}
that this problem cannot be solved if we take $g$ to be smooth. We thus look for solutions
on a blowup space, smooth in polar coordinates, relaxing the requirement that they be smooth on $X$.

Motivated by Proposition \ref{smoothprop} and Theorem \ref{normformthm} and by our need
to break the gauge in the Einstein equations, we take $\tX = [0,\theta_0] \times M$, where
$\theta_0 = \cos^{-1}(-\lambda)$. We look for metrics in the normal form
(\ref{polarform}).
We then take $\tS = [0,\theta_0]_{\theta} \times S$, $\tQ = \left\{ \theta_0 \right\} \times M$, and
$\tM = \left\{ 0 \right\} \times M$. We look for a metric $g$ on
$\tX \setminus (\tM \cup \tS)$, of the form
\begin{equation}
  \label{normformcop}
  g = \csc^2(\theta)\left[ d\theta^2 + h_{\theta} \right],
\end{equation}
where $h_{\theta}$ is a smooth one-parameter family of AH metrics on $M$ with $h_0 \in [h]$, and satisfying
the Einstein equation to as high an order as possible at $\tS$ and the equation
$K_{\tQ} = \lambda g|_{T\tQ}$ along $\tQ$. Our goal is to prove Theorem
\ref{einthm}.

Note that there is no loss of generality in our choice of $\tX$: although a general cornered space $(X,M,Q)$ might have
boundary components $M$ and $Q$ of differing topology, our construction is formal and at $\tS$, where
the topology is determined by $S = \partial M$ alone. 

Throughout this section, it will be convenient to work explicitly with sections of the 0-edge bundle
${}^{0e}T^*\tX$ and its tensor products, as well as the $0$-bundle ${}^0T^*M$. 
We let $\left\{ x^{\mu} \right\}$ be local coordinates near a point $p$ of $M$, so that $(\theta,x^{\mu})$ gives a coordinate
system on $\tX$ near $[0,\theta_0] \times \left\{ p \right\}$.
We define a frame for ${}^{0e}T^*\tX$ given by 
\begin{align*}
  \omega^0 &= \frac{d\theta}{\sin(\theta)}\\
  \omega^{\mu} &= \frac{dx^{\mu}}{\rho\sin(\theta)}
\end{align*}
(where $\rho = x^n$). 

It will be useful to compute the umbilic condition in normal form.

\begin{lemma}
  \label{klem}
  For a metric $g$ on $\tX$ in the form (\ref{normformcop}),  let $K_{\tQ}$ be the second fundamental form of $\tQ \setminus \tS$, and
  $\lambda = -\cos(\theta_0)$.
  Then the condition $K_{\tQ} = \lambda g|_{T\tQ}$ is equivalent to the condition
  \begin{equation*}
    \partial_{\theta}h_{\theta}|_{\theta = \theta_0} = 0.
  \end{equation*}
\end{lemma}
\begin{proof}
  Plainly, the inward-pointing normal to $\tQ$
  is given by $\nu = -\sin(\theta_0)\frac{\partial}{\partial\theta}$. By Weingarten's equation,
  \begin{align*}
    K_{\mu\nu} &= -g_{k\nu}\nabla_{\mu}\nu^k\\
    &= \sin(\theta_0)\Gamma_{\mu 0\nu},
  \end{align*}
  where $\Gamma_{\mu 0\nu} = \frac{1}{2}(\partial_{\mu}g_{0\nu} + \partial_{\theta}g_{\mu\nu} - \partial_{\nu}g_{0\nu}) =
  \frac{1}{2}\partial_{\theta}g_{\mu\nu}$. Since $g_{\mu\nu} = \csc^2(\theta)h_{\mu\nu}$ and $\partial_{\theta}(\csc^2(\theta)) = -2\csc^2(\theta)\cot(\theta)$,
  the result follows.
\end{proof}

As discussed in the introduction, to prove Theorem \ref{einthm} we will choose a conformal representative $h \in [h]$; we will take it to be
in AH normal form. Then $h_0$ in (\ref{normformcop}), as discussed, will be of the form $\chi h$, where $\chi$ is a function
to be determined. This motivates the following proposition, which we will use to prove the theorem. Recall that $\mathcal{M}(\theta_0,M)$ is defined
in Section \ref{funcsubsec}.

\begin{proposition}
  \label{auxprop}
  Let $n \geq 2$, and suppose $S$ is a smooth manifold of dimension $n - 1$.
  Given a one-parameter family of metrics $k_{\rho}$ on $S$, there exists a one-parameter family of smooth AH metrics
  $\{h_{\theta}:0 \leq \theta \leq \theta_0\}$ on $S \times [0,\varepsilon)_{\rho}$ and a function $\chi \in C^{\infty}(S \times [0,\varepsilon))$,
  unique mod $O(\rho^n)$, such that $h_{\theta} \in 
  \mathcal{M}(\theta_0,S \times [0,\varepsilon))$, and letting $\brh_{\theta} = \rho^2h_{\theta}$, we have
  \begin{enumerate}
    \item\label{firstcond} $\chi|_{\rho = 0} = 1$;
    \item\label{umbcond} $\partial_{\theta}\bar{h}_{\theta}|_{\theta = \theta_0} = 0$ for all $\rho$;
    \item $\partial_{\theta}\bar{h}_{\theta}|_{\rho = 0} = 0$ for all $\theta$;
    \item\label{boundcond} $\bar{h}_0 = \chi(d\rho^2 + k_{\rho})$; and
    \item\label{lastcond} if $g$ is the metric on $[0,\theta_0] \times S \times [0,\varepsilon)$ given by
      \begin{equation}
        \label{polarformcop}
        g = \csc^2(\theta)[d\theta^2 + h_{\theta}],
      \end{equation}
      then $\ric(g) + ng = O_g(\rho^n)$.
  \end{enumerate}
  Moreover, $\bar{h}_{\theta}$ is even in $\theta$ to order $n$.

  Finally, if the ordinary differential operator $I_{n,k}$ given by (\ref{iform}) has trivial kernel for all $k \in \mathbb{N}$, then
  in fact $h_{\theta}$ may be chosen so that $\ric(g) + ng = O_g(\rho^{\infty})$.
\end{proposition}

We next prove that, assuming Proposition \ref{auxprop} is true, Theorem \ref{einthm} follows.

\begingroup
\allowdisplaybreaks

\begin{proof}[Proof of Theorem \ref{einthm} using Proposition \ref{auxprop}]
  Let $h \in [h]$, and let $\psi:S \times [0,\varepsilon) \to W \subseteq M$ be a diffeomorphism with a neighborhood $W$ of $S$ in $M$ such that
  $\psi^*h = \frac{d\rho^2 + k_{\rho}}{\rho^2}$, with $k_{\rho}$ a one-parameter family of metrics on $S$. Then $\psi$ induces a diffeomorphism
  $id \times \psi: [0,\theta_0] \times S \times [0,\varepsilon) \to \tX = [0,\theta_0] \times M$.

  By Proposition \ref{auxprop}, there exists a one-parameter family $\tilde{h}_{\theta} \in \mathcal{M}(\theta_0,S \times [0,\varepsilon))$ of AH metrics on
  $S \times [0,\varepsilon)$, and a smooth function $\chi \in C^{\infty}(S \times [0,\varepsilon))$
  satisfying conditions \ref{firstcond} - \ref{boundcond}, and such that 
  $\tilde{g} = \csc^2(\theta)(d\theta^2 + \tilde{h}_{\theta})$ is Einstein mod $O_g(\rho^n)$; moreover, both $\overline{\tilde{h}}_{\theta}$ and
  $\chi$ are uniquely defined mod $O(\rho^n)$.

  Let $h_{\theta} = (\psi^{-1})^*\tilde{h}_{\theta} \in \mathcal{M}(\theta_0,W)$, and $g = ((id \times \psi)^{-1})^*\tilde{g}$; then it is clear that
  $g = \csc^2(\theta)(d\theta^2 + h_{\theta})$, that $h_0 = ((\psi^{-1})^*\chi)h \in [h]$, and that
  $\ric(g) + ng = O_g(\rho^n)$. Moreover, by condition \ref{umbcond} in Proposition \ref{auxprop}, we conclude that
  $\partial_{\theta}h_{\theta}|_{\theta = \theta_0} = 0$. By Lemma \ref{klem}, and because $\cos(\theta_0) = -\lambda$, it follows
  that along $\tQ \setminus \tS$, we have $K_{\tQ} = \lambda g|_{\tQ}$. Thus, existence is established.

  For uniqueness, suppose now that we have another one-parameter family $h_{\theta}'$ of AH metrics on $\tM$ such that
  $g' = \csc^2(\theta)[d\theta^2 + h_{\theta}']$ is also Einstein mod $O_g(\rho^n)$, $h_0' \in [h]$, and
  $K_{\tQ}' = \lambda g'|_{\tQ}$. Suppose also that $\partial_{\theta}(\rho^2h_{\theta})|_{\tS} = 0$.
  Now let $\tilde{h}_{\theta}' = \psi^*h_{\theta}'$, and $\tilde{g}' = (\id \times \psi)^*g'$. Notice that if we write
  $h_0' = \Omega^2 h_0$, then $\Omega|_{S} = 1$. Then it is easy to see
  that $\tilde{g}'$ and $\tilde{h}_{\theta}'$ satisfy conditions \ref{firstcond} - \ref{lastcond} of Proposition \ref{auxprop}.
  Thus, $\tilde{g} - \tilde{g}' = O_{\tilde{g}}(\rho^n)$. Pushing forward again, we may conclude that $g - g' = O_g(\rho^n)$.

  Finally, since by hypothesis there are no integral solutions to (\ref{hyperchareq}) with $s = n$, it follows by Proposition \ref{charlem} that
  $I_{n,\nu}$ has nontrivial kernel for all integral $\nu$, and the last claim follows.
\end{proof}

We now begin working toward a proof of Proposition \ref{auxprop}. Since we will have no further cause to refer to the setting
of Theorem \ref{einthm}, for the remainder of this section we will for convenience let $\tX = [0,\theta_0] \times S \times [0,\varepsilon)$,
let $\tM = \left\{ 0 \right\} \times S \times [0,\varepsilon)$, let $M = S \times [0,\varepsilon)$, let $\tQ = \left\{ \theta_0 \right\} \times S \times [0,\varepsilon)$, and
let $\tS = [0,\theta_0] \times S \times \left\{ 0 \right\}$.

We begin by computing the following.
\begin{lemma}
  \label{eineqslem}
  Let $S$ be a manifold of dimension $n - 1$, let $\theta_0 \in (0,\pi)$, and let $g$ be a metric in the normal form (\ref{polarformcop})
  on $\tX = [0,\theta_0] \times S \times [0,\varepsilon)_{\rho}$.
  Set $E = \ric(g) + ng$. Then $E = \hat{E}_{ij}dx^idx^j = E_{ij}\omega^i\omega^j$, where
  \begin{align}
    E_{00} = &-\frac{1}{2}\sin^2(\theta)\brh^{\mu\nu}\partial_{\theta}^2\brh_{\mu\nu} + \frac{1}{2}\sin(\theta)
    \cos(\theta)\brh^{\mu\nu}\partial_{\theta}\brh_{\mu\nu}+\frac{1}{4}
    \sin^2(\theta)\left|\partial_{\theta}\brh\right|_{\brh}^2\label{e00}\\
    E_{0\sigma} = &\frac{1}{2}\sin^2(\theta)
    \left[ 
    \brh^{\mu\nu}\partial_{\theta}(\brh_{\mu\nu})\rho_{\sigma} - n\rho^{\mu}\partial_{\theta}\brh_{\sigma\mu}
    + \rho\left(\nabla^{\mu}(\partial_{\theta}\brh_{\sigma\mu}) - \nabla_{\sigma}\left(\left( \partial_{\theta}\brh \right)_{\mu}^{\mu}\right)
    \right)\right]\label{e0sigma}\\
    E_{\mu\nu} = &-\frac{1}{2}\sin^2(\theta)\partial_{\theta}^2\brh_{\mu\nu} + \frac{n - 1}{2}\sin(\theta)\cos(\theta)\partial_{\theta}\brh_{\mu\nu}\label{emunu}\\
    &\quad+\frac{1}{2}\sin^2(\theta)\brh^{\eta\lambda}\partial_{\theta}(\brh_{\mu\eta})\partial_{\theta}(\brh_{\nu\lambda})
    -\frac{1}{4}\sin^2(\theta)\brh^{\eta\lambda}\partial_{\theta}(\brh_{\eta\lambda})\partial_{\theta}(\brh_{\mu\nu})\nonumber\\
    &\quad + \frac{1}{2}\sin(\theta)\cos(\theta)
    \brh^{\eta\lambda}\partial_{\theta}(\brh_{\eta\lambda})\brh_{\mu\nu}
    + (1 - n)\sin^2(\theta)\left( |d\rho|_{\brh}^{2} - 1 \right)\brh_{\mu\nu}\nonumber\\
    &\quad+ (n - 2)\rho\sin^2(\theta)\nabla_{\mu}\rho_{\nu} + \rho\sin^2(\theta)\nabla^{\eta}\rho_{\eta}\brh_{\mu\nu}
    + \rho^2\sin^2(\theta)\ric(\brh)_{\mu\nu}.\nonumber
  \end{align}
  Here indices are raised and covariant derivatives taken with respect to $\brh = \rho^2h$, and $\brh = \brh_{\mu\nu}dx^{\mu}dx^{\nu}$.
\end{lemma}
\begin{proof}
  Using the form (\ref{polarformcop}) of the metric, we compute the Christoffel symbols as follows in coordinates:
  \begin{equation}
    \label{christoff}
    \begin{array}{ll}
      \Gamma_{000} = -\csc^2(\theta)\cot(\theta) & \Gamma_{\mu\nu 0}=\rho^{-2}\csc^2(\theta)\cot(\theta)\brh_{\mu\nu} - 
      \frac{1}{2}\rho^{-2}\csc^2(\theta)\partial_{\theta}\brh_{\mu\nu}\\
      \Gamma_{0\mu0} = 0 & 
      \begin{split}\Gamma_{0\mu\sigma} =& -\rho^{-2}\csc^2(\theta)\cot(\theta)\brh_{\mu\sigma}\\&+ \frac{1}{2}\rho^{-2}\csc^2(\theta)\partial_{\theta}
      \brh_{\mu\sigma}
      \end{split}\\
      \Gamma_{00\sigma} = 0 &
      \begin{split}\Gamma_{\mu\nu\sigma} =& -2\rho^{-3}\csc^2(\theta)\brh_{\sigma(\mu}\rho_{\nu)} + \rho^{-3}\csc^2(\theta)\brh_{\mu\nu}\rho_{\sigma}
      \\&+ \rho^{-2}\csc^2(\theta)\overline{\Gamma}_{\mu\nu\sigma},
    \end{split}
    \end{array}
  \end{equation}
  where $\overline{\Gamma}$ is the Christoffel symbol of $\brh$. The result now follows from a tedious but straightforward computation using
  the equation
  \begin{equation}
    \label{ricci}
    R_{ij} = \frac{1}{2}g^{kl}\left( \partial_{il}^2g_{jk} + \partial_{jk}^2g_{il} - \partial_{kl}^{2}g_{ij} - \partial_{ij}^2g_{kl} \right)
    + g^{kl}g^{pq}\left( \Gamma_{ilp}\Gamma_{jkq} - \Gamma_{ijp}\Gamma_{klq} \right).
  \end{equation}
\end{proof}

We state the following, which will be of use later.
\begin{lemma}
  \label{einmaplem}
  If $h_{\theta} \in \mathcal{H}_{n,l}(\theta_0,M,S^2({}^{0}T^*M))$ for some $l \geq 0$ and $\partial_{\theta}\brh_{\theta}|_{\rho = 0} = 0$, then
  for each $j \geq 0$,
  $\partial_{\rho}^jE(g)|_{\rho = 0} \in \oplus_{i = 0}^{m_j}(\theta^n\log(\theta))^iC^{\infty}(\tS,S^2({}^{0e}T^*\tX))$ for some finite $m_j \geq 0$.
  Moreover, if $\brh_{\theta}$ is even in $\theta$ through order $k \geq 2$, then
  $E(g)$ is even through the same order as a section of $S^2({}^{0e}T^*\tX)$.
\end{lemma}
\begin{proof}
  We sketch the proof for $E_{00}$ in (\ref{e00}). The evenness claim is clear, since the number of factors of $\sin(\theta)$ is the same as the number of derivatives with respect
  to $\theta$ in each term. For the first claim, inspect each term and notice that at each finite power of $\rho$, there is a bounded power of $\theta^n\log(\theta)$ by the hypothesis
  on $h_{\theta}$. The powers of $\theta^n\log(\theta)$ in $\brh^{-1}$ are bounded at each finite order in $\rho$ due to the hypothesis that
  $\partial_{\theta}\brh_{\theta}|_{\rho = 0} = 0$. The proof for the other components of $E$ is similar.
\end{proof}

Our approach to proving Proposition \ref{auxprop} will be to construct the metric term by term in powers of $\rho$ 
by solving the indicial equation, just as for the scalar Laplacian.
There are two complications compared to that case: because the operator is nonlinear and acts on sections of a 0-edge bundle, the definition of the
indicial operator is more involved and depends on the metric;
and because the indicial operator acts differently on different parts of the isotypic 
decomposition of the metric tensor, there are in effect really several indicial operators.
This also occurs in the usual AH case -- see e.g. \cite{gl91}. In that case, however, the various parts of the indicial
operator are all algebraic, not differential operators.

At each order, we will have to solve a regular singular system of ODEs given by the indicial operators.
Because we have gauge-broken the Einstein equations by requiring the metric to be in normal form, the system is
overdetermined -- we have $\frac{n(n + 1)}{2}$ unknowns, but $\frac{(n + 1)(n + 2)}{2}$ equations. We will therefore follow the usual expedient of
using the Bianchi identities to show that the extra equations are automatically satisfied once we have determined the
solution using $\frac{n(n + 1)}{2}$ equations. It is by the Bianchi equations, as we will see, that $\chi$ will be uniquely determined at each order.

Because of the form of metric (\ref{polarformcop}), we will be interested in perturbations 
\begin{equation*}
  g \mapsto g + \rho^{\gamma}\varphi,
\end{equation*}
where $\gamma > 0$ and $\varphi = \varphi_{\mu\nu}\omega^{\mu}\omega^{\nu} = \rho^{-2}\csc^2(\theta)\varphi_{\mu\nu}dx^{\mu}dx^{\nu}$ 
is a section of the bundle $\mathcal{T} =
\left\{ \eta \in S^2({}^{0e}T^*\tX): \sin(\theta)\frac{\partial}{\partial\theta} \into \eta = 0\right\}$. A section $\sigma$ of $T$
can be identified with a one-parameter family $\sigma_{\theta}$ ($0 \leq \theta \leq \theta_0$) of sections
of $S^2({}^0T^*\tM)$ over $\tM$.
We will also refer to $\bar{\varphi} = \rho^2\sin^2(\theta)\varphi = \varphi_{\mu\nu}dx^{\mu}dx^{\nu}$, which is a section of $S^2T^*\tX$ with the property that
$\frac{\partial}{\partial\theta} \into \bar{\varphi} = 0$.
Fix a metric $g \in C^{\infty}(\tX,S^2({}^{0e}T^*\tX))$.
We now define the indicial operator $I^{\gamma}:C^{\infty}(\tS,\mathcal{T}) \to C^{\infty}(\tS,S^2({}^{0e}T^*\tX))$, depending on $g$, as follows.
Let $\varphi \in C^{\infty}(\tS,\mathcal{T})$ be a section, and let $\tilde{\varphi} \in C^{\infty}(\tX,\mathcal{T})$ be any smooth
extension of $\varphi$ to $\tX$.
Then define $I^{\gamma}(\varphi)$ by
\begin{equation*}
  I^{\gamma}(\varphi) = \rho^{-\gamma}\left(E(g + \rho^{\gamma}\tilde{\varphi})
   - E(g)\right)|_{\rho = 0},
\end{equation*}
where the restriction to $\rho = 0$ is taken as a section of $S^2({}^{0e}T^*\tX)$. The definition is independent of the extension $\tilde{\varphi}$ chosen.
As in the scalar case, $I^{\gamma}$ is an ordinary differential operator acting in $\theta$.

\begin{proposition}
  \label{einindprop}
  The indicial operator $I^{\gamma}$ for a metric $g$ in the normal form (\ref{polarformcop})
  has the form $I^{\gamma}(\varphi) = I^{\gamma}_{ij}(\varphi)\omega^i\omega^j$, where
  \begin{align}
    2I^{\gamma}_{00}(\varphi) = & -\sin^2(\theta)\brh^{\mu\nu}\partial_{\theta}^2\varphi_{\mu\nu} + \sin(\theta)
    \cos(\theta)\brh^{\mu\nu}\partial_{\theta}
    \varphi_{\mu\nu}\label{I00}\\
    2I^{\gamma}_{0\sigma}(\varphi) = & \sin^2(\theta)\left[\left( \gamma - n\right)\rho^{\mu}\partial_{\theta}\varphi_{\mu\sigma}
    - \left( \gamma - 1\right) \rho_{\sigma}\brh^{\mu\nu}\partial_{\theta}\varphi_{\mu\nu}\right]\label{I0sigma}\\
    \begin{split}
    2I^{\gamma}_{nn}(\varphi) = & -\sin^2(\theta)\partial_{\theta}^2\varphi_{nn} + (n - 1)\sin(\theta)\cos(\theta)\partial_{\theta}\varphi_{nn}\\
    &+ \sin(\theta)\cos(\theta)\partial_{\theta}(\brh^{\mu\nu}\varphi_{\mu\nu})
    + (\gamma - 2)(\gamma + 1 - n)\sin^2(\theta)\varphi_{nn}\\
    &+ \gamma(2 - \gamma)\sin^2(\theta)\brh^{\mu\nu}\varphi_{\mu\nu}
   \end{split}\label{Inn}\\
   \begin{split}
     2\brh^{\mu\nu}I_{\mu\nu}^{\gamma}(\varphi) = &-\sin^2(\theta)\partial_{\theta}^2(\brh^{\mu\nu}\varphi_{\mu\nu}) + (2n - 1)
     \sin(\theta)\cos(\theta)\partial_{\theta}(
     \brh^{\mu\nu}\varphi_{\mu\nu})\label{itr}\\
     &+ 2(\gamma - n)(\gamma + 1 - n)\sin^2(\theta)\varphi_{nn} - 2\gamma(\gamma - n)\sin^2(\theta)\brh^{\mu\nu}\varphi_{\mu\nu}
   \end{split}\\
   2I_{sn}^{\iroot}(\varphi) =& -\sin^2(\theta)\partial_{\theta}^2\varphi_{sn} + (n - 1)\sin(\theta)\cos(\theta)\partial_{\theta}\varphi_{sn}\label{Isn}\\
   \begin{split}
   2\mathring{I}_{st}^{\iroot}(\varphi) =& -\sin^2(\theta)\partial_{\theta}^2\mathring{\varphi}_{st} + 
   (n - 1)\sin(\theta)\cos(\theta)\partial_{\theta}\mathring{\varphi}_{st}\\
   & - \gamma(\gamma + 1 - n)\sin^2(\theta)
   \mathring{\varphi}_{st}\label{Itf},
   \end{split}
  \end{align}
  where $\mathring{\varphi}_{st} = \varphi_{st} - \frac{1}{n - 1}\brh^{rq}\varphi_{rq}\brh_{st}$, and similarly for $\mathring{I}_{st}$.
\end{proposition}

Note that on each fiber $F$ of $\tS$, $I$ restricts to an operator $I:C^{\infty}(F,\mathcal{T}) \to C^{\infty}(F,S^2({}^{0e}T^*\tX))$.

\begin{proof}
  We set $\hat{g} = g + \rho^{\gamma}\varphi_{\mu\nu}\omega^{\mu}\omega^{\nu}$. Writing $\hat{g}$ in the form (\ref{polarformcop}),
  we see that the change $g \mapsto \hat{g}$ is equivalent to the change $\brh_{\mu\nu}dx^{\mu}dx^{\nu} \mapsto
  (\brh_{\mu\nu} + \rho^{\gamma}\varphi_{\mu\nu})dx^{\mu}dx^{\nu}$.
  We use this expression in equations (\ref{e00}) - (\ref{emunu}) to compute $I_{00}^{\gamma}$, $I_{0\sigma}^{\gamma}$, and
  $I_{\mu\nu}^{\gamma}$,
  using the formula $I^{\gamma}(\varphi) = \rho^{-\gamma}\left[ E(g + \rho^{\gamma}\tilde{\varphi}) - E(g) \right]|_{\rho = 0}$.
  We then specialize $I_{\mu\nu}^{\gamma}$ with various choices of $\mu$ and $\nu$ to obtain the result. We will carry out only the
  computation for $I_{00}$; the rest are similar.

  Because $\partial_{\theta}\brh_{\theta}|_{\rho = 0} = 0$, then in particular $\partial_{\theta}\brh_{\theta} = O(\rho)$ and $\partial_{\theta}\brh^{-1} = O(\rho)$.
  It follows that the effect of the perturbation on the inverse metric $\brh^{\mu\nu}$ may be ignored, as may the last
  term in (\ref{e00}). Formula (\ref{I00}) then follows immediately.
\end{proof}

Before proving Proposition \ref{auxprop}, we define some notation that will be useful.
Notice that, by the product structure of $\tX$, there is a natural decomposition ${}^{0e}T\tX \approx {}^{0e}T[0,\theta_0] \oplus {}^{0e}TS \oplus {}^{0e}T[0,\varepsilon)$,
where the summands on the right have their obvious meanings.
Similarly, there is a natural decomposition ${}^0T\tM \approx {}^0TS  \oplus {}^0T[0,\varepsilon)$. For a section $T \in C^{\infty}(\tX,S^2({}^{0e}T^*\tX))$, we let
$T|_{TS}$ be the restriction of $T$ to the middle factor in the above three-way decomposition.
Now if $k$ is a metric on $S$, then $\rho^2k$ is a metric on ${}^0TS$. For $T \in C^{\infty}(\tX,S^2({}^{0e}T^*\tX))$ and $k$ a metric on $S$, we will define the notation
\begin{equation*}
  \tf_kT := \tf_{\rho^2k}(T|_{TS}),
\end{equation*}
so that $\tf_kT$ is a section in $\csc^2(\theta)C^{\infty}(\tX,S^2({}^0TS))$.
In components, this takes the usual form
\begin{equation*}
  (\tf_kT)_{st} = T_{st} - \frac{1}{n - 1}k^{pq}T_{pq}k_{st}.
\end{equation*}
Similarly, if $T$ is a section of $S^2({}^0T^*M)$, then $\tf_{k}T$ will refer to $\tf_k(T|_{TS})$. We will also use the notation $\tf_kT$ in its usual sense when $T$ is an ordinary
symmetric two-tensor.

\endgroup

\begin{proof}[Proof of Proposition \ref{auxprop}]
  We will construct a solution order-by-order in $\rho$. At each step, we will solve the regular singular system of ODEs given by the operators
  (\ref{I00}) - (\ref{Itf}). As mentioned earlier, we will actually use only some of the equations to solve for $\varphi$, and will show that the others are satisfied by
  our solution using the Bianchi identity.

  We are determining $h_{\theta}$ in (\ref{polarformcop}); although we will work instead with $\brh_{\theta} = \rho^2h_{\theta}$. Our boundary condition at
  $\theta = \theta_0$ is that $\partial_{\theta}\brh_{\theta}|_{\theta = \theta_0} = 0$. 
  At $\tM$, our boundary condition is that $\brh_0 = \chi(d\rho^2 + k_{\rho})$, where $\chi$ is as-yet an undetermined function.

  We assume for now that $n > 2$.

  We define $\brh^{(0)}_{\theta} = d\rho^2 + k_{\rho}$, and $g^{(0)} = \csc^2(\theta)(d\theta^2 + \rho^{-2}\brh^{(0)})$. 
  It is straightforward to show using
  Lemma \ref{eineqslem} that $E^{(0)} := \ric(g^{(0)}) + ng^{(0)} = O_{g}(\rho)$ -- only terms on the last two lines
  of (\ref{emunu}) are nonvanishing. We similarly define $\chi^{(0)} \in C^{\infty}(M)$
  by $\chi^{(0)} \equiv 1$.

  We will now proceed by induction. Let $1 \leq \gamma \leq n - 1$, and suppose for purpose of induction that we have a
  metric $g^{(\gamma - 1)} = \csc^2(\theta)\left[ d\theta^2 + \rho^{-2}\brh^{(\gamma - 1)}_{\theta} \right]$ and a smooth function 
  $\chi^{(\iroot - 1)} \in C^{\infty}(\tM)$
  such that 
  \begin{enumerate}[{(i)}]
    \item $\partial_{\theta}\brh^{(\gamma - 1)}_{\theta}|_{\rho = 0} = 0$;\label{firstindcond}
    \item $\brh^{(\gamma - 1)}_{\theta} \in \mathcal{H}_{n,l}(\theta_0,M,S^2T^*M)$ for some $l$;
    \item $\chi \in C^{\infty}(M)$;
    \item $\partial_{\theta}\brh^{(\gamma - 1)}_{\theta}|_{\theta = \theta_0} = 0$;
    \item $\chi^{(\gamma - 1)}|_{\rho = 0} = 1$;
    \item $\brh^{(\gamma - 1)}_{\theta}|_{\theta = 0} = \chi^{(\gamma - 1)}(d\rho^2 + k_{\rho})$;\label{0cond}
    \item\label{eincond} $E^{(\gamma - 1)} := E(g^{(\gamma - 1)}) = O_g(\rho^{\gamma})$;
    \item $\brh^{(\gamma - 1)}_{\theta}$ is even in $\theta$ through order $n$ if $n$ is even, or infinite order if $n$ is odd;\label{auxcond1}
    \item $\rho^{-\gamma}\tf_{k_0}E^{(\gamma - 1)}|_{\rho = 0} \in \mathcal{B}_{n,m}(\theta_0,S,S^2({}^0T^*M))$ for some $m$; and\label{auxcondB}\label{auxcondlast}
    \item $\chi^{(\gamma - 1)}$ and $\brh^{(\gamma - 1)}_{\theta}$ satisfying conditions (\ref{firstindcond}) - (\ref{eincond}) are uniquely defined modulo $O(\rho^{\gamma})$.\label{lastindcond}
  \end{enumerate}
  We wish to show that we can construct a function $\chi^{(\gamma)}$ and a family of metrics $\brh^{(\gamma)}_{\theta}$ such that these conditions
  are all satisfied with $\gamma$ everywhere replaced by $\gamma + 1$. (Conditions (\ref{auxcond1}) - (\ref{auxcondlast}) will be used at several points
  in the induction step.)
  Put differently, we wish to show that we may uniquely define perturbations
  $\varphi^{(\gamma)} \in \csc^2(\theta)\mathcal{H}_{n,l'}(\theta_0,S,S^2({}^0T^*M))$ (for some $l'$) and 
  $\psi^{(\gamma)} \in C^{\infty}(S)$ such that, taking $\brh^{(\gamma)} = \brh^{(\gamma - 1)} + \rho^{\gamma}\bar{\varphi}^{(\gamma)}$ and
  $\chi^{(\gamma)} = \chi^{(\gamma - 1)} + \rho^{\gamma}\psi^{(\gamma)}$, the desired conditions are satisfied. Actually, condition
  (\ref{0cond}) will be satisfied only through order $\gamma$, due to the impact on higher-order terms on the right-hand side of changing
  $\chi$ at order $\gamma$; we will restore (\ref{0cond}), however, without affecting any other conditions
  by adding one more perturbation that is independent of $\theta$.
  We will henceforth refer just to $\varphi$ and $\psi$, leaving the $(\gamma)$ implicit.

  Define $f = \rho^{-\gamma}E^{(\gamma - 1)}|_{\rho = 0} \in C^{n - 1}(\tS,S^2({}^{0e}T^*\tX)|_{\tS})$.
  It is easy to see that we will have completed the induction if we can find $\varphi$ and $\psi$ such that
  \begin{enumerate}[{(1)}]
    \item $I^{\gamma}(\varphi) = -f$, where $I^{\gamma}$ is the indicial operator defined above;\label{firststepcond}
    \item $\varphi_{nn}|_{\theta = 0} = \psi$\label{nnbound};
    \item $\brh^{\mu\nu}\varphi_{\mu\nu}|_{\theta = 0} = n\psi$\label{tracebound};
    \item $\varphi_{ns}|_{\theta = 0} = 0$;
    \item $\mathring{\varphi}_{st}|_{\theta = 0} = 0$;
    \item $\partial_{\theta}\bar{\varphi}|_{\theta = \theta_0} = 0$;\label{laststepcond}
    \item $\bar{\varphi}$ is even in $\theta$ through order $n$ for even $n$, or infinite order if $n$ is odd;
    \item $\bar{\varphi} \in \mathcal{H}_{n,l'}(\theta_0,S,S^2T^*M)$ for some $l'$;
    \item $\psi \in C^{\infty}(\tM)$; and
    \item $\left.\rho^{-(\gamma + 1)}\tf_{k_0}E\left(g^{(\gamma)}\right)\right|_{\rho = 0} \in \mathcal{B}_{n,m'}(\theta_0,S,S^2({}^0T^*M))$ for some $m'$\label{auxshow},
  \end{enumerate}
  with $\varphi$ and $\chi$ determined uniquely by conditions (\ref{firststepcond}) - (\ref{laststepcond}).

  Fix $x \in \tM \cap \tS \approx S$, which we regard as determining a fiber. We will determine $\varphi$ and $\psi$ on the fiber
  $[0,\theta_0] \times \left\{ x \right\} \times \left\{ 0 \right\}$.
  Since our constructions will all depend smoothly on $x$, we suppress it when convenient and write $\varphi$ as a function of $\theta$ alone.
  We first regard $\psi(x)$ as a free parameter and show that, for any choice of $\psi(x)$, $\varphi(x,\theta)$ is uniquely determined. Thus, for now regard
  $\psi(x)$ as given, and $\chi^{(\iroot)} = \chi^{(\iroot - 1)} + \rho^{\iroot}\psi$. 
  
  We first determine $\mathring{\varphi}_{st} = \tf_{k_0}\bar{\varphi}$. 
  Our boundary condition at $\theta = \theta_0$ is, as noted above, $\partial_{\theta}\bar{\varphi}|_{\theta = \theta_0} = 0$. Our boundary condition at
  $\theta = 0$ is $\mathring{\varphi}_{st} = 0$. Now we wish to solve
  $\mathring{I}_{st}(\varphi) = -\mathring{f}_{st}$.
  By (\ref{Itf}), $\mathring{I}$ acts as a scalar on $\mathring{\varphi}$. Moreover, $\mathring{I}$ is merely $-\frac{1}{2}$ times
  the indicial operator of the scalar Laplacian, by Lemma \ref{indexlem}. Now by (\ref{auxcondB}), $\mathring{f}_{st} \in \mathcal{B}_{n,m}(\theta_0,S,S^2({}^0T^*M))$, and so
  by Propositions \ref{specprop} and  \ref{indexprop}, the equation $\mathring{I}_{st}(\varphi) = 
  -\mathring{f}_{st}$
  has a unique solution $\mathring{\varphi}_{st}$ in $\mathcal{A}_{n,m}(\theta_0,S,S^2({}^0T^*M))$. By induction and (\ref{emunu}), $\mathring{f}_{st}$ is even in $\theta$ through order $n$ if $n$ is even,
  or infinite order otherwise. Thus, by the form (\ref{greensop}) of the Green's operator,
  we also may conclude that if $n$ is even, then $\mathring{\varphi}_{st}$ is also even to the order $n$ in $\theta$.

  Now suppose $n$ is odd. Because $\mathring{f}_{st}$ is even to infinite order, and in particular through order $n + 2$, it follows by Proposition \ref{indexprop}
  that $\mathring{\bar{\varphi}}$ is smooth and even to infinite order, and contains no logarithmic terms. Whether $n$ is even or odd, then, $\mathring{\varphi}_{st} \in
  \mathcal{H}_{n,m}(\theta_0,S,S^2T^*M)$.

  We next determine the trace $\bar{h}^{\mu\nu}\varphi_{\mu\nu}$, which for convenience we denote $\ell(\theta)$ for the remainder
  of this proof. Because of the overdetermined
  nature of our system, it would appear \emph{a priori} possible to use either $I_{00}^{\iroot}$ or $\bar{h}^{\mu\nu}I_{\mu\nu}^{\iroot}$ to do this.
  However, $I_{00}^{\gamma}$ is simpler because it involves only the trace of $\varphi$, whereas $\brh^{\mu\nu}I_{\mu\nu}^{\gamma}$ involves
  both the trace and $\varphi_{nn}$. (Because $\brh$ at $\rho = 0$ is independent of $\theta$, we may regard $I_{00}^{\gamma}$ as giving a differential
  equation for $\ell$.) We thus proceed with
  $I_{00}^{\iroot}$. As usual, we wish to solve the equation
  $I_{00}^{\iroot}(\varphi) = -f_{00}$, subject to the conditions $\ell(0) = n\psi(x)$ and $\ell'(\theta_0) = 0$. 
  We claim that $f_{00} = O(\theta^4)$: by the induction hypothesis, $\partial_{\theta}\brh^{(\gamma - 1)} = O(\theta)$, so the last
  term in (\ref{e00}) is $O(\theta^4)$. The other two terms may be written
  $-\sin^2(\theta)(\brh^{\mu\nu}\partial_{\theta}^2\brh_{\mu\nu} - \cot(\theta)\brh^{\mu\nu}\partial_{\theta}\brh_{\mu\nu})$.
  But since $\brh^{(\gamma - 1)}$ is even in $\theta$ to order $n$, the term in parenthesis is $O(\theta^{2})$, as claimed. Moreover,
  since $\partial_{\theta}\brh_{\theta}$ is even through order $n$ or infinity in $\theta$ (depending on parity), it follows from (\ref{e00}) that $f_{00}$ is as well.
  It is easy to verify that a solution to the equation $I_{00}^{\gamma}\ell = -f_{00}$ satisfying $\ell(0) = n\psi(x)$ and
  $\ell'(\theta_0) = 0$ is given by
  \begin{equation*}
    \begin{split}
      \ell(\theta) =& 2\left(\cos(\theta)\int_{\theta}^{\theta_0} \csc^3(\phi)f_{00}(\phi)d\phi + \int_0^{\theta}\cot(\phi)\csc^2(\phi)f_{00}(\phi)d\phi\right.\\
      &- \left.\int_0^{\theta_0}\csc^3(\phi)f_{00}(\phi)d\phi\right) + n\psi(x).
  \end{split}
  \end{equation*}
  The solution is easily shown to be unique: the homogeneous equation is linear, with general solution $a\cos(\theta) + b$. Given the requirements
  that $\ell(0) = 0 = \ell'(\theta_0)$, we may deduce that $a = 0 = b$. Furthermore, since $f_{00}$ is even in $\theta$ through order $n$ or infinity, the solution
  $\ell(\theta)$ is as well. This may be easily verified by differentiating the above solution formula, obtaining
  $\ell'(\theta) = -\sin(\theta)\int_{\theta}^{\theta_0}\csc^3(\phi)f_{00}(\phi)d\phi$. Similarly, Lemma \ref{einmaplem} implies that
  $f_{00} \in \mathcal{H}_{n,l'}(\theta_0,S)$ for some $l'$, and since $\int \theta^p \log(\theta)^qd\theta = O(\theta^{p + 1}\log(\theta)^q)$, we conclude that
  $\ell(\theta) \in \mathcal{H}_{n,l'}(\theta_0,S)$ as well.
  
  Next we wish to determine $\varphi_{nn}$, for which we use $I_{0n}^{\gamma}$. We get the equation
  \begin{equation*}
    (\iroot - n)\partial_{\theta}\varphi_{nn} = (\iroot - 1)\ell'(\theta) - f_{0n}(\theta),
  \end{equation*}
  with conditions $\varphi_{nn}(0) = \psi(x)$ and $\varphi_{nn}'(\theta_0) = 0$. Now $\ell'(\theta_0) = 0$ by construction. Moreover,
  by (\ref{e0sigma}) and our inductive hypothesis -- according to which $\partial_{\theta}\brh^{(\gamma - 1)}|_{\theta = \theta_0} = 0$ --
  we see that $f_{0n}(\theta_0) = 0$ as well. Thus, the right hand side vanishes at $\theta = \theta_0$, and
  our boundary condition at $\theta_0$ is satisfied automatically. We can therefore integrate to
  uniquely determine $\varphi_{nn}$ subject to the condition that $\varphi_{nn}(0) = \chi(x)$. By construction of $\ell(\theta)$ and
  by (\ref{e0sigma}), the right-hand side of the above equation is odd through order $n - 1$ or through order infinity; therefore, $\varphi_{nn}(\theta)$ is even
  through order $n$ or infinity (depending on parity). A similar argument as for the trace also shows that $\varphi_{nn} \in \mathcal{H}_{n,l'}(\theta_0,S)$.

  We have only to determine $\varphi_{ns}$, which is to say, $\left(\frac{\partial}{\partial \rho} \into \bar{\varphi}\right)|_{TS}$. To do this, we use $I_{0s}$. We get the equation
  \begin{equation*}
    (\iroot - n)\partial_{\theta}\varphi_{ns} = -f_{0s}(\theta).
  \end{equation*}
  As in the previous case, the boundary condition at $\theta_0$ is automatically satisfied, and we can integrate to get a unique solution satisfying
  our conditions; parity is preserved as desired, and $\varphi_{ns} \in \mathcal{H}_{n,m}(\theta_0,S,T^*M)$.

  We have determined $\varphi$, and thus have constructed a $g^{(\gamma)}$ so that $E_{00}^{\gamma}, E_{0\sigma}^{\gamma}$, and $\mathring{E}_{st}^{\gamma} =
  O_g(\rho^{\gamma + 1})$. However, it remains to analyze
  $\bar{h}^{\mu\nu}E_{\mu\nu}^{\gamma}$, $E_{nn}^{\gamma}$, and $E_{ns}^{\gamma}$, since their corresponding indicial operators were not used in our construction.
  (We will henceforth omit the $(\gamma)$ from $E$ for clarity.)
  For this, we will use the contracted Bianchi identities, which state that $2\nabla^iE_{ij} = \nabla_jE_{i}^i$; or working now in the coordinate frame, that
  \begin{equation*}
    0 = B_i := 2g^{jk}\partial_k\hat{E}_{ij} - g^{jk}\partial_i\hat{E}_{jk} - 2g^{jk}g^{ql}\Gamma_{jkq}\hat{E}_{il}.
  \end{equation*}
  We apply this to $g^{(\gamma)}$ using our earlier computations (\ref{christoff}) of Christoffel symbols. Still working in the coordinate frame, we find
  \begin{align*}
    \begin{split}
      B_0 =& \sin^2(\theta)\partial_{\theta}\hat{E}_{00} + 2\rho^2\sin^2(\theta)\bar{h}^{\mu\nu}\partial_{\nu}\hat{E}_{0\mu} - \rho^2\sin^2(\theta)\brh^{\mu\nu}
      \partial_{\theta}\hat{E}_{\mu\nu}\\
      &+ 2(1 - n)\sin(\theta)\cos(\theta)\hat{E}_{00} + \sin^2(\theta)\brh^{\mu\nu}\partial_{\theta}(\brh_{\mu\nu})\hat{E}_{00}\\
      &+ 2(2 - n)\rho\sin^2(\theta)
      \rho^{\lambda}\hat{E}_{0\lambda} - 2\rho^2\sin^2(\theta)\brh^{\mu\nu}\brh^{\eta\lambda}\overline{\Gamma}_{\mu\nu\eta}\hat{E}_{0\lambda};
  \end{split}\numberthis\label{B0}\\
  \intertext{and}
  \begin{split}
    B_{\sigma} =& 2\sin^2(\theta)\partial_{\theta}\hat{E}_{0\sigma} - \sin^2(\theta)\partial_{\sigma}\hat{E}_{00} + 2\rho^2\sin^2(\theta)\brh^{\mu\nu}
    \partial_{\nu}\hat{E}_{\mu\sigma}\\
    &-\rho^2\sin^2(\theta)\brh^{\mu\nu}\partial_{\sigma}\hat{E}_{\mu\nu} + 2(1 - n)\sin(\theta)\cos(\theta)\hat{E}_{0\sigma} \\
    &+\sin^2(\theta)\brh^{\mu\nu}\partial_{\theta}(\brh_{\mu\nu})\hat{E}_{0\sigma}
    +2(2 - n)\rho\sin^2(\theta)\rho^{\lambda}\hat{E}_{\sigma\lambda}\\
    & -2\rho^2\sin^2(\theta)\brh^{\mu\nu}\brh^{\eta\lambda}\overline{\Gamma}_{\mu\nu\eta}\hat{E}_{\sigma\lambda},
  \end{split}\numberthis\label{Bsigma}
  \end{align*}
  where $\overline{\Gamma}$ is the Christoffel symbol of $\brh$.

  We evaluate $B_0$ mod $O(\rho^{\gamma + 1})$, using the fact that we already know the following:
  \begin{equation*}
    \begin{array}{ccc}
      \hat{E}_{00} = O(\rho^{\iroot + 1}) & \hat{E}_{0\mu} = O(\rho^{\iroot}) & \mathring{\hat{E}}_{st} = O(\rho^{\iroot - 1})\\
      \hat{E}_{ns} = O(\rho^{\iroot - 2}) & \hat{E}_{nn} = O(\rho^{\iroot - 2}) & \brh^{\mu\nu}\hat{E}_{\mu\nu} = O(\rho^{\iroot - 2}).
    \end{array}
  \end{equation*}
  The first row are all $O_{g}(\rho^{\iroot + 1})$, as desired, but the second row are one order lower. Putting these into the equation for $B_0$
  and setting it equal to $0$ yields
  \begin{equation*}
    \brh^{\mu\nu}\partial_{\theta}\hat{E}_{\mu\nu} = O(\rho^{\iroot - 1}).
  \end{equation*}
  This says that $\rho^{1 -\iroot}\brh^{\mu\nu}\hat{E}_{\mu\nu}|_{\rho = 0}$ is a constant, say $\frac{c}{2}$. We need $c = 0$. By definition of the indicial operator, the equation
  $\rho^{1-\iroot}\brh^{\mu\nu}\hat{E}_{\mu\nu}|_{\rho = 0} = \frac{c}{2}$ is equivalent to saying
  $2\brh^{\mu\nu}I_{\mu\nu}(\varphi) + 2\brh^{\mu\nu}f_{\mu\nu} = c$. The left-hand side of this latter equation, of course, is already determined up
  to choice of $\psi$, since
  $\varphi$ is. Notice in (\ref{itr}) that $\brh^{\mu\nu}I_{\mu\nu}$ depends on $\varphi_{nn}$ and $\brh^{\mu\nu}\varphi_{\mu\nu}$, which in turn
  we have determined using the operators (\ref{I0sigma}) and (\ref{I00}), respectively. Neither of these operators has a zeroth-order part, and so
  adding $\delta$ to $\psi$ adds $\delta$ to $\varphi_{nn}$ and $n\delta$ to $\brh^{\mu\nu}\varphi_{\mu\nu}$, by our boundary conditions
  (\ref{nnbound}) and (\ref{tracebound}). Now using equation (\ref{itr}), but shifting it to the coordinate frame, we see that adding $\delta$ to $\psi$ adds
  $2(1 - n)(\gamma - n)(\gamma + 1)\delta$ to $2I_{\mu\nu}(\varphi)$. Thus, since $\gamma \neq n$, there is a unique choice of $\psi(x)$
  such that $c = 0$; and so we find that $\chi^{(\gamma)} = \chi^{(\gamma - 1)} + \rho^{\gamma}\psi$ is uniquely determined up through order
  $\gamma$ so that $\brh^{\mu\nu}E_{\mu\nu} = O_g(\rho^{\gamma})$; and there remains no further freedom in our system.

  It remains to analyze $\hat{E}_{n\sigma}$. We next look at $B_s$. We find that
  \begin{equation*}
    2\rho\partial_{\rho}\hat{E}_{ns} + 2(2 - n) \hat{E}_{ns} = O(\rho^{\iroot - 1}).
  \end{equation*}
  Now write $\hat{E}_{ns} = \xi_s\rho^{\iroot - 2}$, which we may do by the above computations. Putting this into our equation, we find
  \begin{equation*}
    (\nu - n)\xi_s = O(\rho^{\iroot - 1}),
  \end{equation*}
  as desired.

  Before proceeding to $\hat{E}_{nn}$, we note that we can write
  $\hat{E}_{\mu\nu} = \alpha\rho^{\iroot - 2}\rho_{\mu}\rho_{\nu} + 2\xi_{(\mu}\rho_{\nu)} + \frac{1}{n}\ell \brh_{\mu\nu} + \eta_{\mu\nu},$ where
  $\rho^{\mu}\xi_{\mu} = 0 = \brh^{\mu\nu}\eta_{\mu\nu}$ and $0 = \rho^{\mu}\eta_{\mu\nu}$, and finally $\eta_{\mu\nu} =
  \eta_{(\mu\nu)}$. Then every term here except $\alpha\rho^{\gamma - 2}$ is $O(\rho^{\iroot - 1})$ by our earlier analysis. Now using the Bianchi identity
  $B_n = 0$, we find
  \begin{equation*}
    2\rho\partial_{\rho}\hat{E}_{nn} + 2(2 - n)\hat{E}_{nn} = O(\rho^{\gamma - 1}).
  \end{equation*}
  Since $\hat{E}_{nn} = \alpha\rho^{\iroot - 2}$, we find that $(\iroot - n)\alpha = O(\rho)$, and thus, since $\iroot \neq n$,
  we conclude that $\hat{E}_{nn} = O(\rho^{\iroot - 1})$.

  Thus, $E(g^{(\gamma)}) = O_{g}(\rho^{\iroot + 1})$. Now $\brh^{(\gamma)}_{\theta}$ 
  lies in $\mathcal{H}_{n,l'}(\theta_0,S,S^2T^*M)$ for some $l'$, is even to order $n$ in $\theta$, and satisfies our boundary
  conditions. Also it is unique subject to these conditions. It remains only to show that (\ref{auxshow}) obtains. This is trivial
  if $n$ is odd; so let $n$ be even.
  
  Consider equation (\ref{emunu}). Let $v$ be any term on the right hand side except for the first two and except for
  \begin{equation*}
    w = \frac{1}{2}\sin(\theta)\cos(\theta) \brh^{\eta\lambda}\partial_{\theta}(\brh_{\eta\lambda})\brh_{\mu\nu};
  \end{equation*}
  then since $\brh_{\theta} \in \mathcal{A}_{n,l'}(\theta_0,S,S^2T^*M)$,
  it follows easily that for every $j \geq 0$, we have
  \begin{equation*}
    \partial_{\rho}^jv|_{\rho = 0} \in \mathcal{B}_{n,m'}(\theta_0,S,S^2({}^0T^*M)) \text{ (some $m'$)}.
  \end{equation*}
  For example, take $v = \frac{1}{2}\sin^2(\theta)\brh^{\eta\lambda}\partial_{\theta}(\brh_{\mu\eta})\partial_{\theta}(\brh_{\nu\lambda})$.
  The lowest order at which $\log(\theta)$ can appear in $\partial_{\theta}\brh_{\theta}$ is at power $\theta^{n - 1}$. Since there is a factor
  of $\sin^2(\theta)$, $\log(\theta)$ therefore does not appear in $v$ before order $\theta^{n + 1}$ (and in fact $\theta^{n + 2}$, since
  $\partial_{\theta}\brh_{\theta} = O(\theta)$ as well). Similarly, for $i > 1$, $\log(\theta)^i$ never appears before order
  $\theta^{in}$, due to the factor of $\sin^2(\theta)$ and the hypothesis that $\brh_{\theta} \in \mathcal{A}_{n,l'}(\theta_0,S,S^2T^*M)$. The remaining terms
  are similar.
  Likewise, if $v$ is the \emph{sum} of the first two terms, we have the same result, because $n$ is an indicial root at $\theta = 0$
  of the operator $-\sin^2(\theta)\partial_{\theta}^2 + (n - 1)\sin(\theta)\cos(\theta)\partial_{\theta}$.
  Now $0 = \partial_{\rho}^{\gamma}\left(E_{\mu\nu}^{(\gamma)}\right)|_{\rho = 0}$; since the only term in (\ref{emunu}) that might contribute a term
  of the form $\theta^n\log(\theta)$ is $w$, we may conclude that, in fact, $w$ does not contribute such a term for
  $\brh^{(\gamma)}$ (as there is nothing to cancel it out). Since $\partial_{\theta}\brh_{\eta\lambda}^{(\gamma)} = O(\theta)$, and thus the $\theta^n\log(\theta)$ terms that might
  be present in $\brh^{\eta\lambda}$ and $\brh_{\mu\nu}^{(\gamma)}$ cannot contribute a $\log(\theta)$ to $w$ at order $\theta^n$,
  we conclude that $\partial_{\rho}^{j}(\sin(\theta)\cos(\theta)\brh^{\eta\lambda}\partial_{\theta}\brh_{\eta\lambda}^{(\gamma)})|_{\rho = 0}
  \in \mathcal{B}_{n,m'}(\theta_0,S)$ for $0 \leq j \leq \gamma$.

  Now consider $\tf_{k_0}(\rho^{-(\gamma + 1)}E^{(\gamma)})|_{\rho = 0}$. As we have seen by analyzing (\ref{emunu}), every term lies in
  $\mathcal{B}_{n,m'}(\theta_0,S,S^2({}^0T^*M))$ except possibly 
  \begin{align*}
    \tf_{k_0}(\rho^{-(\gamma + 1)}w)|_{\rho = 0} &= \frac{1}{(\gamma + 1)!}\tf_{k_0} \partial_{\rho}^{\gamma + 1}w|_{\rho = 0}\\
  &= \frac{1}{2(\gamma + 1)!}\tf_{k_0}\partial_{\rho}^{\gamma + 1}(\sin(\theta)\cos(\theta)
  \brh^{\eta\lambda}\partial_{\theta}(\brh_{\eta\lambda})\brh_{\mu\nu})|_{\rho = 0}.
  \end{align*}
  But since $\tf_{k_0}\brh_{st} = 0$, this term vanishes unless at least
  one factor of $\partial_{\rho}$ falls on $\brh_{\mu\nu}$. This leaves at most $\gamma$ derivatives to fall on $\sin(\theta)\cos(\theta)\brh^{\eta\lambda}\partial_{\theta}(\brh_{\eta\lambda})$; and as just
  seen, for any $j \leq \gamma$, $\partial_{\rho}^{j}(\sin(\theta)\cos(\theta)\brh^{\eta\lambda}\partial_{\theta}\brh_{\eta\lambda})|_{\rho = 0}
  \in \mathcal{B}_{n,m'}(\theta_0,S)$. We therefore may conclude that 
  \begin{equation*}
    \tf_{k_0}\rho^{-(\gamma + 1)}E^{(\gamma)}|_{\rho = 0} \in \mathcal{B}_{n,m'}(\theta_0, S,S^2({}^0T^*M)).
  \end{equation*}
  Thus, $\brh^{(\gamma)}_{\theta}$ satisfies all our desired conditions except (\ref{0cond}). As mentioned earlier, however,
  this is easily fixed. Set $b = \chi^{(\gamma)}(d\rho^2 + k_{\rho}) - \brh_0^{(\gamma)} \in C^{\infty}(\tM,S^2T^*\tM)$, and extend it to a section in
  $C^{\infty}(\tX,S^2T^*\tM)$ by making it constant in $\theta$. Now replace
  $\brh^{(\gamma)}$ by $\brh^{(\gamma)} + b$. Since $b$ is independent of $\theta$, this obtains condition
  (\ref{0cond}) without compromising our other conditions.

  And so by induction, we may construct $g^{(n - 1)}$ such that
  $E(g^{(n - 1)}) = O_g(\rho^n)$, satisfying the desired boundary condition at $\tM$ to order $\rho^n$ and to infinite order at $\tQ$. This completes
  the proof for $n > 2$, except for the claim that if $I_{n,n,\nu}$ is injective for all integral $\nu$, then we may solve the system to infinite order.
  As observed above, of course, $I_{n,n,\nu}$ is simply twice the tracefree part of the indicial operator of the Einstein problem; the absence of
  integral indicial roots simply means that we will be able to solve the tracefree equation $\mathring{I}_{st}^{\gamma}(\varphi) = -\mathring{f}_{st}$ for any integral $\gamma$
  (which, of course, are the only $\gamma$ we will encounter on our induction).
  Meanwhile, if $\gamma \notin \mathbb{Z}$ is some value for which $\mathring{I}^{\gamma}$ actually does fail to be injective, we may nevertheless simply choose the coefficient
  of $\rho^{\gamma}$ in the expansion of $\bar{h}$ to be $0$, without affecting our ability to expand indefinitely. In fact, we \emph{must} so choose the coefficients of $\rho^{\gamma}$ for non-integral $\gamma$,
  as our metric is supposed to be smooth. Hence there is no loss of uniqueness.

  Thus, our induction can proceed indefinitely, using the above arguments, with a single problem: at order $\gamma = n$, several crucial coefficients in the indicial operators vanish,
  causing our above induction-step arguments for components other than the tracefree tangential component to fail. Thus, we now provide an argument that at $\gamma = n$ we may (non-uniquely) extend
  $\brh^{(n - 1)}$ to $\brh^{(n)}$, so long as $n$ is not an indicial root; 
  the above arguments then go through once more at every higher order, so we can complete our induction. Note that the only loss of uniqueness occurs at order $n$, so given a single
  scalar choice at that order, uniqueness otherwise remains to infinite order.

  Let $\gamma = n$, then, and assume once more that conditions (\ref{firstindcond}) - (\ref{lastindcond}) hold. Let $f$ be as before. We again wish to find $\varphi$ and $\psi$ such that (\ref{firststepcond}) - (\ref{auxshow}) hold,
  except in this case not uniquely. In particular, $\psi$ will remain undetermined in this argument (and parametrizes our freedom). Thus, let $\psi \in C^{\infty}(S)$ be arbitrary. We have already seen, by the above remarks,
  that we may uniquely find $\mathring{\varphi}_{st}$ such that $\mathring{I}_{st}(\mathring{\varphi}) = -\mathring{f}_{st}$. The same arguments given in the previous case establish that $\mathring{\varphi}_{st}$ is
  even in $\theta$ to order $n$, or to infinite order if $n$ is odd, and is also smooth if $n$ is odd. Thus, $\mathring{\varphi}_{st} \in \mathcal{H}_{n,m}(\theta_0,S,S^2T^*M)$.

  Once again, we next determine the trace, $\ell = \brh^{\mu\nu}\varphi_{\mu\nu}$. However, this time we use the trace of the indicial operator; that is, we wish to solve the equation
  $\brh^{\mu\nu}I_{\mu\nu}(\varphi) = -\brh^{\mu\nu}f_{\mu\nu}$. Notice in (\ref{itr}) that, for $\gamma = n$ only, the operator $\brh^{\mu\nu}I_{\mu\nu}^{\gamma}$ is uncoupled from
  $\varphi_{nn}$. It is easy to see that the equation $\brh^{\mu\nu}I_{\mu\nu}^n\ell = -\brh^{\mu\nu}f_{\mu\nu}$, with initial conditions $\ell'(\theta_0) = 0$ and $\ell(0) = n\psi(x)$ has the unique solution
  \begin{equation*}
    \ell(\theta) = n\psi(x) - 2\int_0^{\theta_0}\sin^{2n - 1}(\phi)\int_{\theta_0}^{\theta}\csc^{2n - 1}(\beta)\brh^{\mu\nu}f_{\mu\nu}(\beta)d\beta d\phi.
  \end{equation*}
  If $n$ is odd, then $\brh^{\mu\nu}f_{\mu\nu}$ is even to infinite order, and thus $\ell$ is smooth and even to infinite order as well. If $n$ is even, $\ell(\theta)$ is smooth and even through (at least) order $n$, as
  desired.

  To determine the $\varphi_{nn}$ component, we can no longer use $I_{0n}$. However, as we have already determined $\ell$, we can use $I_{nn}$, as given in (\ref{Inn}). It is straightforward to see
  that the unique solution to $I_{nn}^n\varphi_{nn} = -f_{nn}$ satisfying $\varphi_{nn}(0) = \psi$ and $\varphi_{nn}'(\theta_0) = 0$ is
  $\varphi_{nn}(\theta) = \psi + u(\theta)$, where $u(\theta)$ is the unique solution to $I_{n,n - 2}u = 2f_{nn} + \sin(\theta)\cos(\theta)\ell'(\theta) - n(n - 2)\sin^2(\theta)\ell(\theta) + (n - 2)\sin^2(\theta)\psi(x)$
  satisfying $u(0) = 0 = u'(0)$, and where
  where $I_{n,n - 2}$ is the indicial operator for the scalar Laplacian analyzed in Proposition \ref{indexprop}. Since $I_{n,n - 2}$ is a bijection by that proposition, the existence and uniqueness of $u$ follow,
  and the desired parity and smoothness properties for $\varphi_{nn}$ follow from the same proposition and from the already-determined properties of $f$ and of $\ell$.

  Finally, $\varphi_{sn}$ may be easily and uniquely determined from the equation $I_{sn}^{n}\varphi_{sn} = -f_{sn}$, using (\ref{Isn}). Thus, we have uniquely determined $\varphi^{(n)}$ subject to our freedom in choosing
  $\psi$.

  It remains now to use the Bianchi identities again, this time to show that $E_{00}, E_{0n}$, and $E_{0s}$ vanish to the desired orders in $\rho$. Letting $E = \hat{E}_{ij}dx^idx^j$ be the Einstein tensor
  of $g^{(n)}$, we know the following by construction:
  \begin{equation*}
    \begin{array}{ccc}
      \hat{E}_{00} = O(\rho^{n}) & \hat{E}_{0\mu} = O(\rho^{n - 1}) & \mathring{\hat{E}}_{st} = O(\rho^{n - 1})\\
      \hat{E}_{ns} = O(\rho^{n - 1}) & \hat{E}_{nn} = O(\rho^{n - 1}) & \brh^{\mu\nu}\hat{E}_{\mu\nu} = O(\rho^{n - 1}).
    \end{array}
  \end{equation*}
  This time, the last entry in the first row and the entire second row are all $O_g(\rho^{n + 1})$, as we would like, but the first two are only $O_g(\rho^n)$ \emph{a priori}.
  We introduce functions $\alpha(\theta,x,\rho)$ and $\beta(\theta,x,\rho)$ defined by $\hat{E}_{00} = \alpha\rho^n$ and $\hat{E}_{0n} = \beta\rho^{n - 1}$. Using the Bianchi identity
  $B_0 = 0$ with (\ref{B0}) now yields
  \begin{equation*}
    \partial_{\theta}\hat{E}_{00} + 2(1 - n)\cot(\theta)\hat{E}_{00} + 2\rho^2\partial_{\rho}E_{0n} + 2(2 - n)\rho\hat{E}_{0n} = O(\rho^{n + 1})
  \end{equation*}
  which, at $\rho = 0$, gives us the equation
  \begin{equation}
    \label{balpheq}
    \alpha'(\theta) + 2(1 - n)\cot(\theta)\alpha(\theta) + 2\beta(\theta) = 0,
  \end{equation}
  where for notational convenience we regard $\alpha$ and $\beta$ as functions only of $\theta$ when $\rho = 0$.
  Similarly, using (\ref{Bsigma}) and the equation $B_{n} = 0$, we find
  \begin{equation}
    \label{bbeteq}
    2\beta'(\theta) + 2(1 - n)\cot(\theta)\alpha(\theta) - n\alpha(\theta) = 0
  \end{equation}
  Now, $\alpha$ and $\beta$ satisfy $\alpha(\theta_0) = 0 = \beta(\theta_0)$, which follows from our induction hypothesis, our construction of $\varphi$, and examination of equations
  (\ref{e00}) and (\ref{e0sigma}). Thus, by the uniqueness of solutions to
  first-order ODEs, we conclude that $\alpha \equiv \beta \equiv 0$. Hence, in fact $\hat{E}_{00} = O(\rho^{n + 1})$ and $\hat{E}_{0n} = O(\rho^{n})$.
  It is now completely straightforward to show, using $B_{s}$, that $\hat{E}_{0s} = O(\rho^{n})$, and we omit the details. Thus, $E(g^{(n)}) = O_g(\rho^{n + 1})$ as desired. The proof that (\ref{auxshow}) obtains
  is identical to the case for $\gamma < n$, and thus we omit it. Since the arguments given for $\gamma < n$ also work for $\gamma > n$, the claim that a solution $g$ exists satisfying $E(g) = O(\rho^{\infty})$ follows
  by induction and Borel's lemma. This concludes the case $n > 2$.

  If $n = 2$, the above proof needs slight modification. At the first step, we define
  $\brh_{\theta}^{(0)} = d\rho^2 + k_0$, which is constant in $\theta$ and also in $\rho$; and also define
  $\chi^{(0)} = 1$. It follows that $E^{(0)} = O_g(\rho^2)$,
  since the only term in the Einstein equations that does not vanish is the Ricci term in (\ref{emunu}). We need only solve
  one more equation, the first-order perturbation, to be done. We set $\brh_{\theta}^{(1)} = \brh_{\theta}^{(0)} + \rho\bar{\varphi}$.
  The equation we wish to solve is $I^{1}(\varphi) = 0$, with boundary
  conditions $\partial_{\theta}\bar{\varphi}|_{\theta = \theta_0} = 0$ and $\bar{\varphi}|_{\theta = 0} = \chi^{(1)}\partial_{\rho}\brh$,
  where $\chi^{(1)} = \chi^{(0)} + \rho\psi$.

  Now $E_{0\sigma}^{(0)} \equiv 0$ and $E_{00}^{(0)} \equiv 0$, so the above analysis of equations (\ref{I00}) and (\ref{I0sigma})
  goes through without problem; this determines $\brh^{\mu\nu}\varphi_{\mu\nu}$, $\varphi_{nn}$, and $\varphi_{ns}$. It remains to determine
  $\mathring{\varphi}_{st}$. But notice that when $n = 2$ and $\gamma = 1$, any constant is a solution to
  $\mathring{I}_{st}(\varphi) = 0$; so we may simply set $\mathring{\bar{\varphi}}_{st} = \chi\tf_{\brh}\partial_{\rho}\brh|_{\rho = 0}$. We have thus
  determined $g^{(1)}$, subject to the freedom in $\psi$; the Bianchi analysis goes through as before, determining $\psi$. Finally, if $n = 2$ is not an indicial
  root, the order-$n$ analysis above allows construction to higher order, as before.
\end{proof}

Notice that it is clear from the above proof that the Taylor coefficients of $\brh_{\theta}$ in $\rho$ are, through order $n - 1$, universal functions
of $\theta$ and of $k_0$, the derivatives of $k_{\rho}$ at $\rho = 0$, and their tangential derivatives. Moreover, the $j$th Taylor coefficient function
depends only of $\partial_{\rho}^ik_{\rho}|_{\rho = 0}$ for $j \leq i$.

It seems apparent that, by including appropriate powers of $\log(\rho)$ as in the proof of Theorem \ref{lapthm}, the above construction could be extended
to infinite order for any $\theta_0$; but we do not here undertake the calculations demonstrating this.

In general, of course, one might want a solution $g$ satisfying $\ric(g) = O_g(\rho^n\sin^\infty(\theta))$, or even
$O_g(\rho^{\infty}\sin^{\infty}(\theta))$. We here sketch an approach that would yield such a metric, although we omit details.
First, one could use the above theorem to obtain a metric $\tilde{g}$ satisfying $\ric(\tilde{g}) + n\tilde{g} = O_g(\rho^m)$ (where $m$ is as high as possible, or
possibly $\infty$). Then, pulling back by the diffeomorphism $\psi$ between $[0,1]_r \times M$ and $[0,\theta_0]_{\theta} \times M$ induced by defining $r = \frac{2(\csc(\theta) - \cot(\theta))}{\csc(\theta_0) - \cot(\theta_0)}$,
it is easy to show, using calculations from \cite{m16}, that the pullback metric $\psi^*\tilde{g}$ is in the usual AH normal form $\frac{dr^2 + \tilde{g}_r}{r^2}$ (but with each $\tilde{g}_r$ an AH metric), 
and that it still satisfies
$\ric(\psi^*\tilde{g}) + n\psi^*\tilde{g} = O(\rho^m)$. Let $\tilde{g}' = \psi^*\tilde{g}$.
We now can perform the inductive Fefferman-Graham construction (as in \cite{fg12})
at $\tM$ to show that one can find a perturbation $\Phi$, vanishing at $\tM$, and satisfying $\ric(\tilde{g}' + \Phi) + n(\tilde{g}' + \Phi) = O(r^{\infty})$.
In the Fefferman-Graham construction, the indicial operators are simply multiplication operators, so at each order in $r$, the perturbation $\Phi$ will be
$O(\rho^m)$. Therefore, by our analysis of the indicial operators above, we conclude that, letting $g' = \tilde{g}' + \Phi$, we will in fact obtain $\ric(g') + ng' = O(\rho^mr^{\infty})$.
Finally, since $\Phi = O_g(\rho^m)$, we can then take a cutoff function $\eta$ that is $1$ near $\tM$ and 0 near $\tQ$, and set $g = (\psi^{-1})^*(g' + \eta\Phi)$. By construction, then, $g$
will satisfy our boundary conditions at $\tQ$ and at $\tM$, and will also satisfy $\ric(g) + ng = O_g(\rho^m\sin^{\infty}(\theta))$.

If $\theta_0 > \frac{\pi}{2}$, then by Proposition \ref{specprop} there will be an indicial root $\gamma_0$ for $\mathring{I}_{st}^{\gamma}$ 
between $n$ and $n - 1$, and uniqueness in Proposition \ref{auxprop} will not be quite
to order $n$ without the requirement of smoothness; we will expect additional solutions with leading asymptotics at order $\rho^{\gamma_0}$.

We will now focus on the proof of Theorem \ref{invthm}. Suppose that $\theta_0 = \frac{\pi}{2}$. By Propositions \ref{piprop}, \ref{indexprop}, and \ref{einindprop}, we know that $n$ is an indicial root,
and that we can solve the tracefree part of the Einstein equations to order $O_g(\rho^{n + 1})$ by a smooth perturbation only if
the tracefree tangential part of $\rho^{-n}E(g^{(n - 1)})|_{\rho = 0}$ is orthogonal to $w_0 = \sin^{n}(\theta)$ with respect to the measure
$\sin^{-(n + 1)}(\theta)d\theta$.
This suggests a way to define the obstruction tensor promised in Theorem \ref{invthm}. Suppose $M^n$ is a manifold with boundary $S$,
and equipped with a metric $\tau$. Near $S$, we can uniquely define a diffeomorphism $\eta:S \times [0,\varepsilon)_{\rho} \hookrightarrow M$ so 
that $\eta^*\tau = d\rho^2 + k_{\rho}$, and so that $\eta|_{S \times \left\{ 0 \right\}} = \id_S$.
Then by Proposition \ref{auxprop}, there is a metric $g$ in the normal form (\ref{polarformcop}) on $\mathring{\tX}$, where
$\tX = \left[ 0,\frac{\pi}{2} \right] \times S \times [0,\varepsilon)$, and a function $\chi \in C^{\infty}(S \times [0,\varepsilon))$
such that \ref{firstcond} - \ref{lastcond} hold. Now notice that for any section ${}^0T \in C^{\infty}(\tS,S^2({}^{0e}T^*\tX))$ satisfying
$T = O_g(\sin^2(\theta))$, we can get a well-defined corresponding section
$T \in C^{\infty}(\tS,S^2T^*\tX)$ by setting $T = \rho^2({}^0T)$. Now observe that $E(g) \in C^{\infty}(S^2({}^{0e}T^*\tX))$, with 
$E(g) = O_g(\rho^n)$. In particular, $\rho^{-n}\tf_{k_0}E(g)|_{\rho = 0} \in C^{\infty}(\tS,S^2({}^{0e}T^*\tX))$. Moreover,
$E(g) = O_g(\sin^2(\theta))$. This follows easily from equations (\ref{e00}) - (\ref{emunu}), remembering that
$\partial_{\theta}\bar{h}_{\theta} = O(\theta)$ by evenness in $\theta$. Thus,
$\rho^2\left[ \rho^{-n}\tf_{k_0}E(g)|_{\rho = 0} \right] \in C^{\infty}(\tS,S^2T^*\tX)$. For shorthand, we write
$\rho^{2 - n}\tf_{k_0}E(g)|_{\rho = 0} \in C^{\infty}(\tS,S^2T^*\tX)$.
Then we define a smooth symmetric tracefree tensor $\mathcal{K}(\tau)$ on $S$ by
\begin{equation}
  \label{hinvform}
  \begin{split}
    \mathcal{K}(\tau) &= \left\langle \rho^{2 - n}\tf_{k_0}E(g)|_{\rho = 0},w_0\right\rangle_{\sin^{-(n + 1)}(\theta)d\theta}\\
    &= \rho^{2-n}\left.\int_0^{\frac{\pi}{2}}\csc(\theta)\mathring{E}(g)d\theta\right|_{\rho = 0} \in S^2T^*S,
  \end{split}
\end{equation}
where $\mathring{E}$ here refers to $\tf_{k_0}E$. 

\begin{proof}[Proof of Theorem \ref{invthm}]
  We first must show that $\mathcal{K}(\tau)$ is well defined. First, the integral (\ref{hinvform}) converges,
  since $E_{\mu\nu} = O(\theta)$ by (\ref{emunu}). Next, although $\brh_{\theta}$ is only determined mod $O(\rho^n)$,
  perturbations of the form $\brh_{\theta} \mapsto \brh_{\theta} + \rho^n\bar{\varphi}$ satisfying $\bar{\varphi}|_{\theta = 0} = 0$ and 
  $\partial_{\theta}\bar{\varphi}|_{\theta = \theta_0} = 0$ leave 
  $\left\langle \tf_{k_0}\rho^{2-n}E(g)|_{\tS,TS},w_0\right\rangle_{\sin^{-(n + 1)}(\theta)d\theta}$ unchanged,
  since $n$ is an indicial root of $\mathring{I}^n_{st}$ and, by Propositions \ref{piprop}, \ref{indexprop}, and \ref{einindprop}, the image of
  $\mathring{I}^n$ is orthogonal to $\sin^{n}(\theta)$. Thus, $\mathcal{K}(\tau)$ is well defined.

  Next, we must show that the conformal transformation law holds. Suppose $\hat{\tau} =
  \Omega^2\tau$, where $\Omega \in C^{\infty}(M)$. Let $\hat{\eta}:S \times [0,\varepsilon)_{\hat{\rho}}
  \hookrightarrow M$ be a diffeomorphism onto a neighborhood of $S$ so that
  $\hat{\eta}^*\hat{\tau} = d\hat{\rho}^2 + \hat{k}_{\hat{\rho}}$ and so that $\hat{\eta}|_{S \times \{0\}} = \id_S$. Let $\hat{h}_{\theta}$ 
  and $\hat{\chi}$ satisfy \ref{firstcond} - \ref{lastcond} in
  Proposition \ref{auxprop}, in particular with $\bar{\hat{h}}_{\theta_0} = \hat{\chi}(d\hat{\rho}^2 + \hat{k}_{\hat{\rho}})$.
  Similarly, we take $\hat{g} = \csc^2(\theta)[d\theta^2 + \hat{h}_{\theta}]$.

  Set $\tilde{h}_{\theta} = \hat{\eta}^*(\eta^{-1})^*h_{\theta}$, as well as $\tilde{\chi} = \hat{\eta}^*(\eta^{-1})^*\chi$ 
  and $\tilde{\rho} = \hat{\eta}^*(\eta^{-1})^*\rho$. Similarly set $\widehat{\Omega} = \hat{\eta}^*\Omega$. Now plainly 
  \begin{equation*}
    \tilde{\rho}^2\tilde{h}_0 = \tilde{\chi}\hat{\eta}^*\tau = \widehat{\Omega}^{-2}\tilde{\chi}(d\hat{\rho}^2 + \hat{k}_{\hat{\rho}}) = \widehat{\Omega}^{-2}\tilde{\chi}\hat{\chi}^{-1}\hat{\rho}^2\hat{h}_0.
  \end{equation*}
  This implies that $\tilde{h}_0 = \frac{\hat{\rho}^2\tilde{\chi}}{\tilde{\rho}^2\widehat{\Omega}^2\hat{\chi}}\hat{h}_0$.
  Since $\tilde{h}_0$ and $\hat{h}_0$ are both AH metrics on $S \times [0,\varepsilon)$, we must therefore have
  \begin{equation*}
    \left.\frac{\hat{\rho}}{\tilde{\rho}}\right|_{\hat{\rho} = 0} = \widehat{\Omega}|_{\hat{\rho} = 0}.
  \end{equation*}
  Now
  \begin{equation*}
    \hat{\rho}^2\tilde{h}_0 = \frac{\hat{\rho}^2}{\tilde{\rho}^2}\tilde{\rho}^2\tilde{h}_0 = \frac{\hat{\rho}^2}{\tilde{\rho}^2}\widehat{\Omega}^{-2}\tilde{\chi}(d\hat{\rho}^2 + \hat{k}_{\hat{\rho}}),
  \end{equation*}
  where $\left.\frac{\hat{\rho}^2}{\tilde{\rho}^2}\widehat{\Omega}^{-2}\tilde{\chi}\right|_{\hat{\rho} = 0} = 1$. Thus, by the uniqueness statement in Proposition \ref{auxprop},
  $\tilde{h}_{\theta} = \hat{h}_{\theta}$ mod $O(\hat{\rho}^n)$. Set $\tilde{g} = \csc^2(\theta)[d\theta^2 + \tilde{h}_{\theta}] =
  \hat{\eta}^*(\eta^{-1})^*g$. It then follows from the discussion in the first paragraph of this proof and the fact that $\eta|_{\rho = 0} = \hat{\eta}|_{\hat{\rho} = 0} = \id$ that
  \begin{align*}
    \mathcal{K}(\tau) &= \left\langle\rho^{2 - n}\tf_{k_0}E(g)|_{\rho = 0},w_0\right\rangle_{\sin^{-(n + 1)}(\theta)d\theta}\\
    &= \left\langle \tilde{\rho}^{2 - n}\tf_{\hat{k}_0}E(\tilde{g})|_{\rho = 0},w_0\right\rangle_{\sin^{-(n + 1)}(\theta)d\theta}\\
    &= \left\langle\tilde{\rho}^{2 - n}\tf_{\hat{k}_0}E(\hat{g})|_{\rho = 0},w_0\right\rangle_{\sin^{-(n + 1)}(\theta)d\theta}\\
    &= \left.\frac{\hat{\rho}^{n - 2}}{\tilde{\rho}^{n - 2}}\left\langle\hat{\rho}^{2 - n}\tf_{\hat{k}_0}E(\hat{g})|_{TS},w_0\right\rangle_{\sin^{-(n + 1)}(\theta)d\theta}\right|_{\rho = 0}\\
    &= \widehat{\Omega}|_{S}^{n - 2}\mathcal{K}(\hat{\tau})\\
    &= \Omega|_S^{n - 2}\mathcal{K}(\hat{\tau}),
  \end{align*}
  which is the desired result.

  We need finally to show that $\mathcal{K}$ is generically nontrivial. We will do this by showing that 
  \begin{equation}
    \label{Knovan}
    \mathcal{K}(\tau) = c\tf_{k_0}\partial_{\rho}^nk_{\rho}|_{\rho = 0} +
  \mathcal{K}'(\tau),
\end{equation}
where $c \neq 0$ and $\mathcal{K}'(\tau)$ depends only on $\partial_{\rho}^jk_{\rho}|_{\rho = 0}$ for $j < n$.
  To proceed, let $\kappa = \partial_{\rho}^nk_{\rho}|_{\rho = 0}$, and extend it to be a section in $C^{\infty}(\tS,S^2T^*S)$ by taking it to be constant in $\theta$.
  Define $\bar{\varphi} \in \mathcal{H}_n(\theta_0,S,S^2T^*S)$ by
  $\partial_{\rho}^n\brh_{st} = \kappa_{st} + \partial_{\rho}^n((\chi - 1)k_{\rho})|_{\rho = 0} + \bar{\varphi}$, where the second term, like $\kappa$, is extended to be constant in $\theta$.
  Notice that the second term depends only on $\partial_{\rho}^jk_{\rho}|_{\rho = 0}$ for $j < n$.
  Plainly, we have $\bar{\varphi}|_{\theta = 0} = 0$ and $\partial_{\theta}\bar{\varphi}|_{\theta = \theta_0} = 0$. It follows then that the inner product
  $\langle\rho^{2 - n}\tf_{k_0}E|_{\rho = 0},w_0\rangle$ is independent of $\bar{\varphi}$, by the discussion in the first paragraph of this proof.
  We wish to find the coefficient of $\mathring{\kappa}_{st}$
  in $\langle \rho^{2 - n}\tf_{k_0}E,w_0\rangle$, and in particular to show that it is nonvanishing.
  
  Consider now the last term of $E_{\mu\nu}$ in (\ref{emunu}), which is given by $u_{\mu\nu} = \rho^2\sin^2(\theta)\ric(\brh)_{\mu\nu}$.
  Recall that the expression for $\ric(\brh)_{\mu\nu}$ is
  \begin{equation}
    \label{rich}
    \begin{split}
    \ric(\brh)_{\mu\nu} = \frac{1}{2}\brh^{\eta\lambda}(\partial_{\mu\lambda}^2\brh_{\nu\eta} + \partial_{\nu\eta}^2\brh_{\mu\lambda} -& \partial_{\eta\lambda}^2\brh_{\mu\nu}
    - \partial_{\mu\nu}^2\brh_{\eta\lambda})\\
   &+ \brh^{\eta\lambda}\brh^{\sigma\tau}(\Gamma_{\mu\lambda\sigma}\Gamma_{\nu\eta\tau} - \Gamma_{\mu\nu\sigma}\Gamma_{\eta\lambda\tau})
    \end{split}
  \end{equation}
  Now because $\brh_{\theta}|_{\rho = 0} = d\rho^2 + k_{0}$, we see that the third term contributes to $u$ a term of the form
  $-\frac{1}{2(n - 2)!}\sin^2(\theta)\rho^{n}\partial_{\rho}^n\brh_{\mu\nu}$. No other term of (\ref{rich}) contributes a multiple of $\partial_{\rho}^n\brh_{\mu\nu}$ at order
  $\rho^n$. Thus, in particular, $u$ contributes a term of the form $-\frac{1}{2(n - 2)!}\sin^2(\theta)\rho^n\mathring{\kappa}_{st}$ (as well as terms involving $\mathring{\varphi}_{st}$
  and lower orders of $k_{\rho}$) to $\mathring{E}_{st}$.

  Consider next the second-last term of (\ref{emunu}), which is $\rho\sin^2(\theta)\nabla^{\eta}\rho_{\eta}\brh_{\mu\nu}$. Because of the factor of $\rho$,
  it does not make any contribution of the form $\rho^n\mathring{\kappa}_{st}$ to $\mathring{E}_{st}$.

  Next consider the third-last term, $v_{\mu\nu} = (n - 2)\rho\sin^2(\theta)\nabla_{\mu}\rho_{\nu}$. It is easy to compute that this takes the form
  \begin{equation*}
    v_{\mu\nu} = \frac{n - 2}{2}\sin^2(\theta)\rho(\partial_{\rho}\brh_{\mu\nu} - 2\partial_{(\mu}\brh_{\nu)n}) + O(\rho).
  \end{equation*}
  The third-last term thus contributes a term of the form $\frac{n - 2}{2(n - 1)!}\sin^2(\theta)\rho^n\mathring{\kappa}_{st}$ to $\mathring{E}_{st}$.

  We claim that no other term of (\ref{emunu}) contributes a term involving $\kappa$ to $E$ at order $\rho^n$. The first two terms do not, because
  $\kappa$ does not depend on $\theta$. The next three terms do not because $\partial_{\theta}\brh = O(\rho)$; and the next, because $|d\rho|_{\brh}^2 - 1 = O(\rho)$.
  Thus, the only contributions of $\mathring{\kappa}_{st}$ to $\mathring{E}_{st}$ are those already calculated from the seventh and ninth terms; their sum is
  $\frac{-1}{2(n - 1)!}\rho^n\sin^2(\theta)\mathring{\kappa}_{st}$. But it is plain that the coefficient of $\rho^n$ here is not orthogonal to $\csc(\theta)$; indeed,
  \begin{equation*}
    \begin{split}
      \frac{-1}{2(n - 1)!}\left\langle\sin^2(\theta)\mathring{\kappa}_{st},w_0\right\rangle_{\sin^{-(n + 1)}(\theta)d\theta} &= \frac{-1}{2(n - 1)!}\int_0^{\frac{\pi}{2}}\sin(\theta)\mathring{\kappa}_{st}d\theta\\
      &= \frac{-1}{2(n - 1)!}\mathring{\kappa}_{st}.
    \end{split}
  \end{equation*}

  Thus, (\ref{Knovan}) holds with $c = \frac{-1}{2(n - 1)!}$.
  Now $\tau$ (thus $k_{\rho}$) may be changed at order $\rho^n$ independently of any lower orders; and we have seen above that changing $\varphi$ does not alter the inner product
  $\langle \rho^{2 - n}\mathring{E}|_{\rho = 0},
  w_0\rangle_{\sin^{-(n + 1)}(\theta)d\theta}$.
  Thus, for generic choices of $\tf_{k_0}\partial_{\rho}^nk_{\rho}|_{\rho = 0} = \mathring{\kappa}$, $\mathcal{K}(\tau)$ is nonvanishing.
\end{proof}